%% file: manuscript.tex
\documentclass[11pt,draftclsnofoot,onecolumn]{IEEEtran}

\usepackage{cite}
\usepackage{amsfonts}
\usepackage{commath}
\usepackage{amsfonts}
\usepackage{color}
\usepackage{amsmath, amsthm}
\usepackage{mathrsfs}
\usepackage{extarrows}
\usepackage{amssymb}
\usepackage{graphicx}
\usepackage{times}
\usepackage{bm}
\usepackage{hyperref}

\newtheorem{theorem}{Theorem}
\newtheorem{lemma}{Lemma}
\newtheorem{corollary}{Corollary}

\newtheorem{assumption}{Assumption O.\hspace{-3pt}}

\newtheorem{proposition}{Proposition}
\newtheorem{example}{Example}

\newcommand{\va}{\boldsymbol{a}}
\newcommand{\vb}{\boldsymbol{b}}
\newcommand{\ve}{\boldsymbol{e}}
\newcommand{\vu}{\boldsymbol{u}}
\newcommand{\vv}{\boldsymbol{v}}
\newcommand{\vw}{\boldsymbol{w}}
\newcommand{\vx}{\boldsymbol{x}}
\newcommand{\vy}{\boldsymbol{y}}
\newcommand{\vz}{\boldsymbol{z}}

\newcommand{\vA}{\boldsymbol{A}}
\newcommand{\vB}{\boldsymbol{B}}
\newcommand{\vD}{\boldsymbol{D}}
\newcommand{\vH}{\boldsymbol{H}}
\newcommand{\vI}{\boldsymbol{I}}
\newcommand{\vM}{\boldsymbol{M}}
\newcommand{\vQ}{\boldsymbol{Q}}
\newcommand{\vT}{\boldsymbol{T}}
\newcommand{\vU}{\boldsymbol{U}}
\newcommand{\vV}{\boldsymbol{V}}
\newcommand{\vW}{\boldsymbol{W}}
\newcommand{\vX}{\boldsymbol{X}}

\newcommand{\vbeta}{\boldsymbol{\beta}}
\newcommand{\vepsilon}{\boldsymbol{\epsilon}}
\newcommand{\vxi}{\boldsymbol{\xi}}
\newcommand{\vsigma}{\boldsymbol{\sigma}}

\newcommand{\vGamma}{\boldsymbol{\Gamma}}
\newcommand{\vDelta}{\boldsymbol{\Delta}}
\newcommand{\vSigma}{\boldsymbol{\Sigma}}

\newcommand{\bbE}{\mathbb{E}}
\newcommand{\bbI}{\mathbb{I}}
\newcommand{\bbN}{\mathbb{N}}
\newcommand{\bbR}{\mathbb{R}}

\newcommand{\vzero}{\boldsymbol{0}}

\newcommand\dotp[1]{\left\langle #1 \right\rangle}

\newcommand{\argmin}{\operatorname{arg\,min}}

\renewcommand\t{{\ensuremath{\scriptscriptstyle{\top}}}}

\newcommand{\n}{\nonumber}
\newcommand{\BE}{\begin{eqnarray}}
\newcommand{\EE}{\end{eqnarray}}

\newcommand{\BEQ}{\begin{equation}}
\newcommand{\EEQ}{\end{equation}}

\newcommand{\BS}{\begin{split}}
\newcommand{\ES}{ \end{split} }

\newcommand{\BI}{\begin{itemize}}
\newcommand{\EI}{\end{itemize}}

\newcommand{\hvb}{\hat{\vbeta}}
\newcommand{\hb}{\hat{\beta}}

\newcommand{\tvepsilon}{\tilde{\vepsilon}}
\newcommand{\bvepsilon}{\bar{\vepsilon}}

\newcommand{\XI}{\vX_{/i}}
\newcommand{\yi}{\vy_{/i}}

\newcommand{\alo}{{\rm ALO}_{\lambda}}
\newcommand{\lo}{{\rm LO}_{\lambda}}
\newcommand{\extra}{{\rm Err}_{{\rm out},\lambda}}

\newcommand{\wampr}{{{\rm AMP}}_{{\rm risk},\lambda}}

\newcommand \loovb[1]{\tilde{\vbeta}^{\backslash#1}} 
\newcommand \lamloovb[1]{\tilde{\vbeta}_{\lambda}^{\backslash#1}} 
\newcommand \loob[1]{\tilde{\beta}^{\backslash #1}}

\newcommand \lopvb[1]{\bar{\vbeta}^{\backslash#1}}
\newcommand \lopb[1]{\bar{\beta}^{\backslash #1}}

\newcommand \lopvy[1]{\bar{\vy}^{\backslash #1}}
\newcommand \lopy[1]{\bar{y}^{\backslash #1}}
\newcommand \lopvtb[1]{\bar{\vbeta}^{\backslash #1}_0}

\newcommand \lophvb[1]{\hat{\vbeta}^{\backslash #1}}
\newcommand \lophb[1]{\hat{\beta}^{\backslash #1}}

\newcommand \lootvb[1]{\breve{\vbeta}^{\backslash #1}}
\newcommand \lootb[1]{\breve{\beta}^{\backslash #1}}

\newcommand \loptvb[1]{\acute{\vbeta}^{\backslash #1}}
\newcommand \loptib[1]{\acute{\beta}^{\backslash #1}}
\newcommand \eloptvb[1]{\check{\vbeta}^{\backslash #1}}
\newcommand \eloptb[1]{\check{\beta}^{\backslash #1}}

\newcommand{\hX}{\hat{\vX}}

\newcommand{\XIb}{\bar {\vX}_{/i}}

\newcommand{\hQ}{\hat{\vQ}}
\newcommand{\bQ}{\bar{\vQ}}
\newcommand{\bD}{\bar{\vD}}

\newcommand{\ide}{\ \stackrel{d}{=}\ }

\newcommand{\spaceequal}{\ =\ } 
\newcommand{\kleq} {\ \leq\ }

\newcommand{\qleq }{\ \leq\ }
\newcommand{\qgeq} {\ \geq \ }
\newcommand{\qand}{\quad \text{and} \quad}
\newcommand{\tw}{\text{weak}}

\newcommand{\tr}{\tilde{r}}

\newcommand\trace{\operatorname{Tr}} 
\newcommand\diag[1]{\operatorname{diag}\left(#1\right)}
\newcommand{\sign}{\operatorname{sgn}}

\newcommand{\fl}{l}
\newcommand{\dfl}{\fl}
\newcommand{\reg}{R}
\newcommand{\dreg}{R}
\newcommand{\cdreg}{\reg}
\newcommand{\creg}{\reg}

\newcommand{\hfl}{\hat{\fl}}

\newcommand{\dhfl}{\hat{\dfl}}

\newcommand{\dbfl}{\bar{\boldsymbol{\fl}}}

\newcommand{\bdfl}{\bar{\boldsymbol{\fl}}}

\newcommand \vrx[1]{\vx_{#1 *}} 
\newcommand \vcx[1]{\vx_{* #1}} 

\newcommand \hvcx[1]{\hat{\vx}_{\cdot #1}}
\newcommand \bvcx[1]{\bar{\vx}_{\cdot #1}}

\newcommand \op[1]{O_p\left(#1\right)} 

\newcommand{\plog}{\operatorname{poly\,log}}
\newcommand{\plogn}{\plog(n)}

\newcommand{\ine}{\rm{Err}_{{\rm in},\lambda}}
\newcommand{\bl}{\bm{\hat \beta}_{\lambda}}

\newcommand{\oA}{\tilde{\vA}}

\newcommand{\pA}{\hat{\vA}}
\newcommand{\pB}{\hat{\vA}}
\newcommand{\bA}{\bar{\vA}}

\newcommand{\cDelta}{\check{\Delta}}

\newcommand{\sd}{\sigma_{\delta}}
\newcommand{\ac}{\kappa_l}
\newcommand{\bc}{c_n}

\newcommand{\new}{\operatorname{new}}

\begin{document}

\title{Consistent Risk Estimation in Moderately High-Dimensional Linear Regression}

\author{ Ji Xu, Arian Maleki, Kamiar Rahnama Rad, and Daniel Hsu

\thanks{J. Xu and D. Hsu are with the Department of Computer Science, Columbia University, New York, USA (e-mails: jixu@cs.columbia.edu, djhsu@cs.columbia.edu).}
\thanks{A. Maleki is with the Department of Statistics, Columbia University, New York, USA (e-mail: arian@stat.columbia.edu).}
\thanks{K. Rahnama Rad is with Department of Information Systems and Statistics, Baruch College, City University of New York, New York, USA (e-mail: kamiar.rahnamarad@baruch.cuny.edu).}
}

\maketitle

\begin{abstract}
\input{abstract}

\end{abstract}

\input{sec-intro}

\input{sec-results}

\bibliography{paper}
\bibliographystyle{unsrt}

\newpage
\input{sec-proof_detail}
\input{suppvarlo}

\input{sec-appendix}

\input{sec-proof}

\end{document}

%% file: abstract.tex
Risk estimation is at the core of many learning systems. The importance of this problem has motivated researchers to propose different schemes, such as cross validation, generalized cross validation, and Bootstrap. The theoretical properties of such estimates have been extensively studied in the low-dimensional settings, where the number of predictors $p$ is much smaller than the number of observations $n$. However, a unifying methodology accompanied with a rigorous theory is lacking in high-dimensional settings.  This paper studies the problem of risk estimation under the moderately high-dimensional asymptotic setting $n,p \rightarrow \infty$ and $n/p \rightarrow \delta>1$ ($\delta$ is a fixed number), and proves the consistency of three risk estimates that have been successful in numerical studies, i.e., leave-one-out cross validation (LOOCV), approximate leave-one-out (ALO), and approximate message passing (AMP)-based techniques. A corner stone of our analysis is a bound that we obtain on the discrepancy of the `residuals' obtained from AMP and LOOCV. This connection not only enables us to obtain a more refined information on the estimates of AMP, ALO, and LOOCV, but also offers an upper bound on the convergence rate of each estimate.

%
%

%% file: sec-intro.tex
\section{Introduction}


\subsection{Objectives}
\label{sec:obj}
In many applications, a dataset $\mathcal{D}=\{(\vrx{1},y_1),(\vrx{2},y_2),\ldots, (\vrx{n},y_n)\}$ with $\vrx{i}\in \bbR^{p}$ and $y_i\in \bbR$ is modeled as
\[
	y_i 
	\spaceequal 
	\vrx{i}^\top\vbeta_0+w_i
	,
\]
where $\vbeta_0\in\bbR^p$ denotes the vector of unknown parameters, and $w_i$ denotes the error or noise. $\vbeta_0$ is typically estimated by the solution to the following optimization problem
\BE
	\hvb_{\lambda}
	\spaceequal
	\argmin_{\vbeta \in \bbR^{p}}\sum_{i=1}^n \fl(y_i-\vrx{i}^{\t}\vbeta)+\lambda \sum_{i=1}^p \reg(\beta_i)
	.	\label{eq:model}
\EE
$\fl$ is the loss function, $\reg$ is  the regularizer, and $\lambda>0$ is a tuning parameter. The performance of $\hvb_{\lambda}$ depends heavily on  $\lambda$. Hence, finding the `optimal'  $\lambda$ is of major interest in machine learning and statistics.  In most applications, one would ideally like to find the  $\lambda$ that minimizes the out-of-sample prediction error:
\[
	\extra 
	\ \triangleq \ 
	\bbE[\fl(y_{\new}-\vrx{\new}^{\t}\hvb_{\lambda})\big|\mathcal{D}]
	,
\]
where $(\vrx{\new},y_{\new})$ is a new data point generated (independently of $\mathcal{D}$) from the same distribution as $\mathcal{D}$.

The problem of estimating $\extra$ from $\mathcal{D}$ has been  studied for (at least) the past 50 years, and the corresponding literature is too vast to be covered here. Methods such as cross validation (CV) \cite{S74,G75}, Allen's PRESS statistic \cite{A74}, generalized cross validation (GCV) \cite{CW79,GHW79}, and bootstrap \cite{E83} are seminal ways to estimate $\extra$. 

Since the past studies have focused on the data regime $n \gg p$, reliable risk estimates supported by rigorous theory are lacking in high-dimensional settings. In this paper, we study the problem of risk estimation under a moderately high-dimensional asymptotic setting where both the number of features and observation go to infinity, while their ratio remains constant, i.e., $n/p=\delta>1$  as $n,p \rightarrow \infty$.\footnote{We should emphasize that we do not have any sparsity (or other structures) assumption on $\vbeta_0$. Hence, $\delta>1$ ensures that when there is no noise in the observations, $\vbeta_0$ can be recovered exactly.} We will call $\delta$ the sample-feature ratio. Under this asymptotic setting, the optimal $\lambda$ that achieves the best sample prediction error converges to a non-zero constant as $n,p\rightarrow \infty$ (See e.g. \cite{dobriban2018high} for ridge regression and \cite{mousavi2018consistent} for LASSO). Therefore in this paper, we consider $\lambda=O(1)$ as $n,p\rightarrow \infty$. Suppose that $\widehat{\rm Err}_{{\rm out}, \lambda}$ is an estimate of $\extra$ obtained from dataset $\mathcal{D}$. The fundamental consistency property we want for an estimate $\widehat{\rm Err}_{{\rm out}, \lambda}$ of $\extra$ is:

\begin{itemize}
\item[($\mathcal{P}_0$)] \begin{center}$\ \ \ \ \ \ \ \ \  \ \ \ |\widehat{\rm Err}_{{\rm out}, \lambda}-\extra| \rightarrow 0$ in probability, as $n,p \rightarrow \infty$ and $n/p = \delta>1$. \end{center}
\end{itemize}
As is clear from Figure \ref{fig:1}, standard techniques such as $3$-fold and $5$-fold cross validation exhibit large biases and do not satisfy $\mathcal{P}_0$. 

The first contribution of this paper is to prove that the following three risk estimation techniques, which have been successful in numerical studies, satisfy $\mathcal{P}_0$: 
\begin{enumerate}
\item Leave-one-out cross validation (LOOCV): Given its negligible bias shown in Figure \ref{fig:loovscv5}, it is expected that $\lo$, the estimate given by LOOCV, satisfies $\mathcal{P}_0$. It is noted that LOOCV is computationally demanding and hence impractical in many applications. 

\item Approximate leave-one-out (ALO): The high computational complexity of $\lo$ prompted several authors to adapt the existing heuristic arguments \cite{Stone77Asymptotic,A74} to approximate $\lo$ and obtain another risk estimate called $\alo$ \cite{beirami2017optimal,rad2018scalable,giordano2018return,wang2018approximate,RZM20}. We formally present ALO in Section \ref{sec:galo}

\item Approximate message passing (AMP): Assuming that $\fl (u, y) = \frac{1}{2}(u-y)^2$,  estimators of the out-sample prediction error  have been presented using the approximate message passing (AMP) framework  \cite{weng2018overcoming,MMB13Asymptotic,OK16Cross,DM16High,bayati2013estimating}. In particular  \cite{MMB13Asymptotic,bayati2013estimating} showed that AMP-based estimate satisfies $\mathcal{P}_0$ for squared loss and bridge regularizers. In this paper, we first generalize AMP-based method to other loss functions and regularizers in Section \ref{sec:reAMP}.  Then, we prove that this estimate satisfies $\mathcal{P}_0$.

\end{enumerate}
 
The consistency is a minimum requirement a risk estimate should satisfy in high-dimensional settings; if the convergence $|\widehat{\rm Err}_{{\rm out}, \lambda}-\extra| \rightarrow 0$ is slow, then the risk estimate will not be useful in practice. This leads us to the next question we would like to address in this paper:

\begin{itemize}
\item Is the convergence  $|\widehat{\rm Err}_{{\rm out}, \lambda}-\extra| \rightarrow 0$ fast? 
\end{itemize}
The second contribution of this paper is to answer this question for the three estimates mentioned above.  To answer this question, we develop tools which are expected to be used in the study of other risk estimates or in other applications. For instance, the connection we derive between the residuals of the leave-one-out and AMP has provided a more refined information on the estimates that are obtained from AMP. Such connections can be useful for the analysis of estimates that are obtained from the empirical risk minimization \cite{donoho2015variance}.

 \subsection{Related work} 
 
 \subsubsection{Out-of-sample prediction error}
 
 The asymptotic regime of this paper was first considered in \cite{huber1973robust}, but only received a considerable attention in the past fifteen years \cite{donoho2015variance,el2013robust,bean2013optimal,el2018impact,sur2017likelihood,weng2018overcoming,johnstone2001distribution,bayati2012lasso,thrampoulidis2015regularized,amelunxen2014living,chandrasekaran2012convex,cai2016geometric}. The inaccuracy of the standard estimates of $\extra$ in high-dimensional settings has been recently noticed by many researchers, see e.g. \cite{rad2018scalable} and the references therein. Hence, several new estimates have been proposed from different perspectives. For instance, \cite{beirami2017optimal,rad2018scalable,giordano2018return,wang2018approximate,RZM20} used different approximations of the leave-one-out cross validation to obtain computationally efficient risk estimation techniques. In another line of work, \cite{obuchi2018accelerating,bayati2013estimating,mousavi2018consistent} used either statistical physics heuristic arguments or the framework of message passing to obtain more accurate risk estimates. 

While most of these proposals have been successfully used in empirical studies, their theoretical properties have not been studied. The only exceptions are \cite{bayati2013estimating,mousavi2018consistent}, in which the authors have shown that the AMP-based estimate satisfy $\mathcal{P}_0$ when $\fl (u-y) = \frac{1}{2}(u-y)^2$ and the regularizer is bridge, i.e., $ R(\beta) = |\beta|^q$ for $q\geq 1$.  Furthermore, the convergence rate is not known for any risk estimate when $n/p\rightarrow \delta>1$. In this paper, we study the three most promising proposals, and present a detailed theoretical analysis under the asymptotic $n,p \rightarrow \infty$ and $n/p \rightarrow \delta>1$. The tools we develop here are expected to be used in the study of other risk estimates or in other applications. For instance, the connection we derive between the leave-one-out estimate and that of approximate message passing has enabled the message passing framework to provide more refined information (such as convergence rate) for different estimates.

In another line of work, \cite{HM17} studies $K$-fold CV when $K=o(n)$; in other words, their results  do not apply to the $n$-fold CV problem considered in our paper. Moreover,  \cite{HM14} prove the risk consistency of lasso when the smoothing parameter is chosen via cross-validation, assuming strong sparsity, namely $\| \bm{\beta^*} \|_0 = O(1)$. In our paper, we make no sparsity assumptions, and hence, the results of  \cite{HM14}  do not apply to the moderate-high dimensional setting considered in this paper.

Before we move on to the next section, we would like to comment on the performance of $K$-fold cross validation schemes given that they are probably the most popular risk estimation technique in high-dimensional settings. As illustrated in Figure \ref{fig:1} and Figure \ref{fig:loovscv5}, in  high-dimensional settings, we empirically observe that $K$-fold cross validation risk estimate is a biased estimator (for fixed values $K$). Furthermore, the optimal values of $\lambda$ selected by $K$-fold cross validation depend on $K$ in general and do not achieve the best prediction risk. The bias of such risk estimates can be explained by the moderately high-dimensional setting $n/p\rightarrow \delta>1$ studied in this paper. Intuitively speaking, under this setting, the $K$-fold cross validation risk estimate is an unbiased risk estimator for the expected prediction risk $\bbE_{\mathcal{D}}\extra$ with a smaller sample feature ratio, $\frac{K-1}{K}\delta$. In the case of ridge regression, one can characterize the exact asymptotic formula for the out-of-sample prediction error $\extra$ for every fixed regularization parameter $\lambda>0$ and sample-to-feature ratio $\delta>0$ \cite{dobriban2018high, belkin2019two, hastie2019surprises, wu2020optimal}. These asymptotic formulas change as delta changes which proves the bias of the $K$-fold cross validation. Note that leave-one-out risk estimate makes minimal changes to the sample-to-feature ratio, and hence the bias vanishes asymptotically as will be clarified in the paper. In addition, we can see from the asymptotic formulas that the optimal value of $\lambda$ is a non-trivial function of $\delta$ except for the cases when the features have isotropic covariance matrix or the true coefficients are generated from an isotropic prior, in which cases the optimal value of $\lambda$ will not depend on $K$ (hence the biases do not change the model we select in the particular case of isotropic features).

\subsubsection{In-sample prediction error} Another approach for obtaining the best value of $\lambda$  is to use the notion of in-sample prediction error instead of the out-of-sample prediction error. The in-sample prediction error is defined as
\begin{eqnarray*}
\ine &\triangleq&\frac{1}{n} \sum_{i=1}^n \bbE_{y^{\rm{new}}_i} \bigl [  l(y_i^{\rm{new}} -  \bm{x_{i*}}^\top \bl)  \big | \mathcal{D} \bigr],
\end{eqnarray*}
where  $y_i^{\rm{new}}$ indicates a hypothetical new data point with the same distribution but independent of the original data point $y_i$. Many strategies have been proposed in the literature to obtain good estimates of $\ine$. Mallow's $C_p$ \cite{M73}, Akaike's Information Criterion (AIC) \cite{Ak74,HT89}, Stein's Unbiased Risk Estimate (SURE) \cite{S81}, Efron's Covariance Penalty \cite{E86}, and Generalized Information Criterion (GIC) \cite{ZLT10} belong to this class of model selection criteria that approximate the in-sample prediction error. There is a vast literature studying the performance of GIC type of risk estimators, eg.  \cite{FHS13,KKC12,Y07,YWF15}. The conclusions made based on studying the performance of GIC type of risk estimators can be extended to leave-one-out CV in the $p/n\rightarrow 0$ regime. When $n$ is much larger than $p$,  the in-sample prediction error is expected to be close to the out-of-sample prediction error. However, this intuition is certainly violated in high-dimensional settings, where $n$ is of the same order as (or even smaller than) $p$. To summarize,  in the $p/n \rightarrow 0$ regime, using Taylor expansions, it can be shown that (approximate) leave-one-out cross-validation is (nearly) equal to GIC-type of estimators but in this paper, we make no assumption about sparsity, and consider the $p/n$-fixed regime where such an expansion is not accurate, and hence GIC-type of conclusions do not extend to leave-one-out CV.

\subsection{Notations}\label{sec:notation}
Let $\ve_j \in \bbR^p$ stand for a  vector filled with zeros except for the $j$th element which is one. Let $\vrx{i}^{\t} \in \bbR^{1 \times p}$ stand for the $i$th row of $\vX \in \bbR^{n \times p}$. Let $\yi \in \bbR^{(n-1) \times 1}$ and $\XI \in \bbR^{(n-1) \times p}$ stand for $\vy$ and $\vX$, excluding the $i$th entry $y_i$ and the $i$th row $\vrx{i}^\top$, respectively. Moreover, let $\vcx{i} \in \bbR^{1 \times n}$ stand for the $i$th column of $\vX$ and $\XIb \in \bbR^{n \times (p-1)}$ stand for $\vX$, excluding the $i$th column. Further let $\lopvtb{i}$ denote the corresponding vectors $\vbeta_0$ without $i$th component and $\lopvy{i}= \XIb\lopvtb{i}+\vw=\vy-\beta_{0,i}\vcx{i}$. For a vector $\va$, we use  $a_i$ to denote its $i$th entry. For any function $\psi(\cdot)$, we use $\psi(\va)$ to indicate the vector $\left [ \psi(a_1),\cdots, \psi(a_d) \right]^{\t}$. The vector $\ve_i$ is filled with zeros except for the $i$th entry which is one. The diagonal matrix whose diagonal elements are $\va$ is referred to as $\diag{\va}$. The component-wise ratio of two vectors $\va$ and $\vb$ is denoted by $\va / \vb$. Moreover, $\dotp{\va}$ stands for the mean of the components of $\va$. We define
\BE
	\hvb_{\lambda} 
	&\triangleq&  
	\underset{\vbeta \in \bbR^p}{\argmin}  \Bigl \{   n\dotp{ \fl\left( \vy  - \vX \vbeta \right)} +\lambda p \dotp{\reg(\vbeta)}   \Bigr \} 
	,	\label{eq:bl}\\
	\lamloovb{i} 
	&\triangleq&  
	\underset{\vbeta \in \bbR^p}{\argmin}  \Bigl \{   (n-1)\dotp{ \fl\left( \yi  - \XI \vbeta \right)} +\lambda p \dotp{\reg(\vbeta)}   \Bigr \} 
	,	\label{eq:bli}\\
	\lopvb{i}_{\lambda} 
	&\triangleq&  
	\underset{\vbeta \in \bbR^{p-1}}{\argmin}  \Bigl \{   n \dotp{ \fl\left( \lopvy{i}  - \XIb \vbeta \right)} +\lambda (p-1) \dotp{\reg(\vbeta)}   \Bigr \} 
	,	\label{eq:blib}
\EE
where $\hvb_{\lambda}$ is the full model and data estimate, and $\lamloovb{i}$ is the leave observation-$i$ out estimate. We refer to $\lopvb{i}_{\lambda}$ as the leave predictor-$i$ out estimate. We may omit subscript $\lambda$ from $\hvb_{\lambda}$ or $\lopvb{i}_{\lambda}$ for simplification reasons. Further, $\fl'(x)$, $\fl''(x)$, $\reg'(x)$, $\reg''(x)$ stand for the first and second derivatives for $\fl$ and $\reg$ respectively. Finally, $\sigma_{\max}(\vM)$ and $\sigma_{\min}(\vM)$ denote the maximum and minimum eigenvalues of a matrix $\vM$ respectively.



\begin{figure}
\begin{center}
        \includegraphics[width=0.8\textwidth]{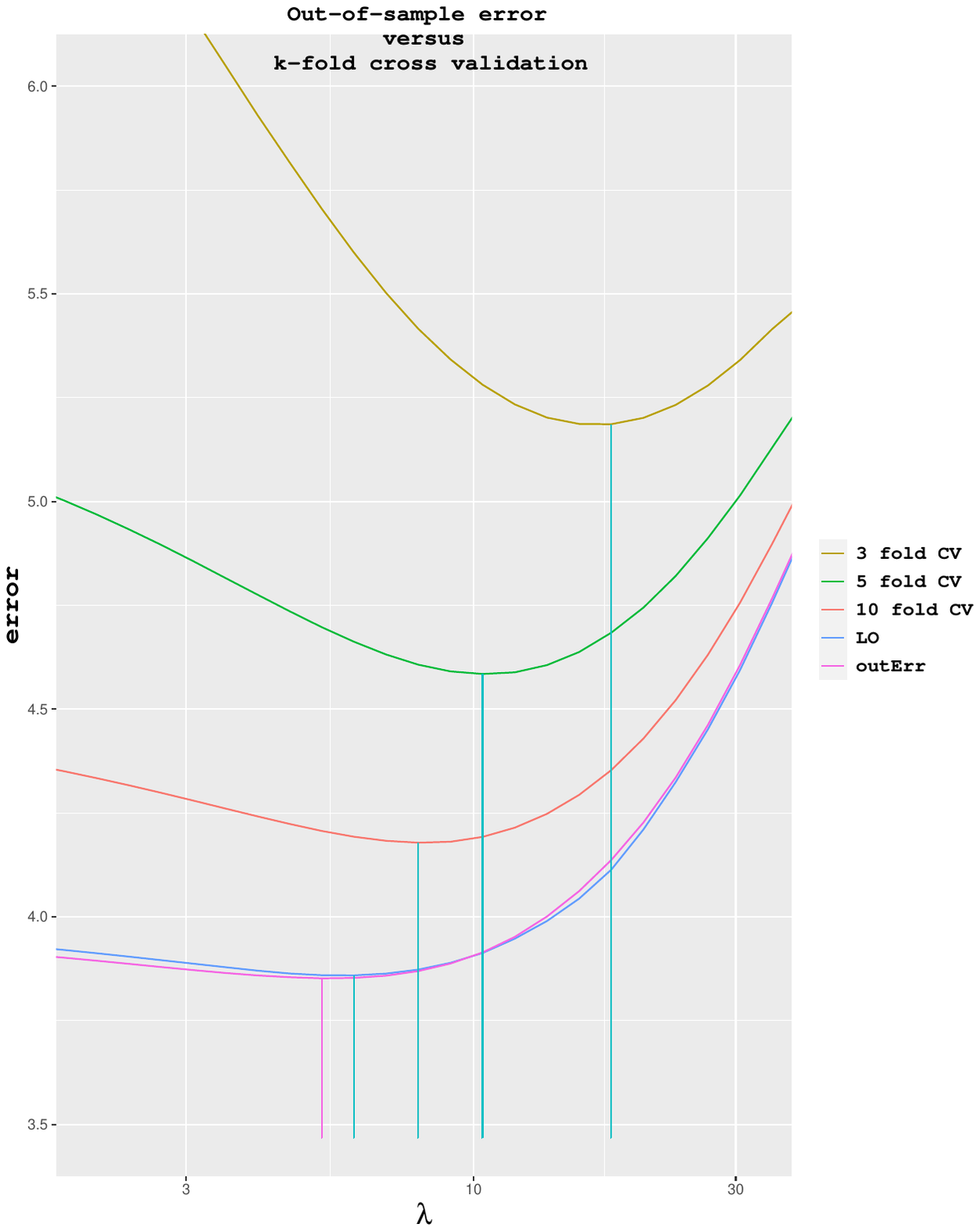}
          \caption{ Comparison of  $K$-fold cross validation (for $K=2,3,5$) and leave-one-out cross validation with the true (oracle-based)  out-of-sample error for the smoothed-$\ell_1$ problem where $\fl( y-\vx^{\t}\hvb )=\frac{1}{2}(y-\vx^{\t} \hvb)^2$ and $\reg(\hvb)=\sum_{i=1}^p g_u(\beta_i)$ (defined in \eqref{eq:slasso}) for $u=100$. In high-dimensional settings the upward bias of $K$-fold CV clearly decreases as number of folds increase. Data is $\vy \sim \mathcal{N}(\vX\vbeta^*, 2\vI)$ where $\vX \in \mathbb{R}^{n \times p}$. Here $n=800$ and $p=400$, and all the components of true coefficients $\vbeta_0$ are set to $\frac{1}{6\sqrt{2}}$. The entries of $\vX$ are independent zero mean unit variance Gaussian random variables. Out-of-sample test data is $y_{\rm new} \sim \mathcal{N}(\bm{x}_{\rm new}^\top \bm{\beta}^*, 2)$ where the entries of $\bm{x}_{\rm new}$ are independent zero mean unit variance Gaussian random variables. The true (oracle-based) out-of-sample prediction error is $\extra =  \mathbb{E} [  ( y_{\rm new}-\bm{x}_{\rm new}^\top \bm{\hat{\beta}} )^2 | \bm{y,X} ]$.  All depicted quantities are averages based on 50 random independent samples. Vertical lines indicates the minimums of the corresponding risk estimates. Clearly, as the number of folds increase the corresponding minimums get closer to the minimum of the vtrue (oracle-based) out-of-sample prediction error.}
          \label{fig:1}
          \end{center}
\end{figure}

\begin{figure}
\begin{center}
        \includegraphics[width=0.8\textwidth]{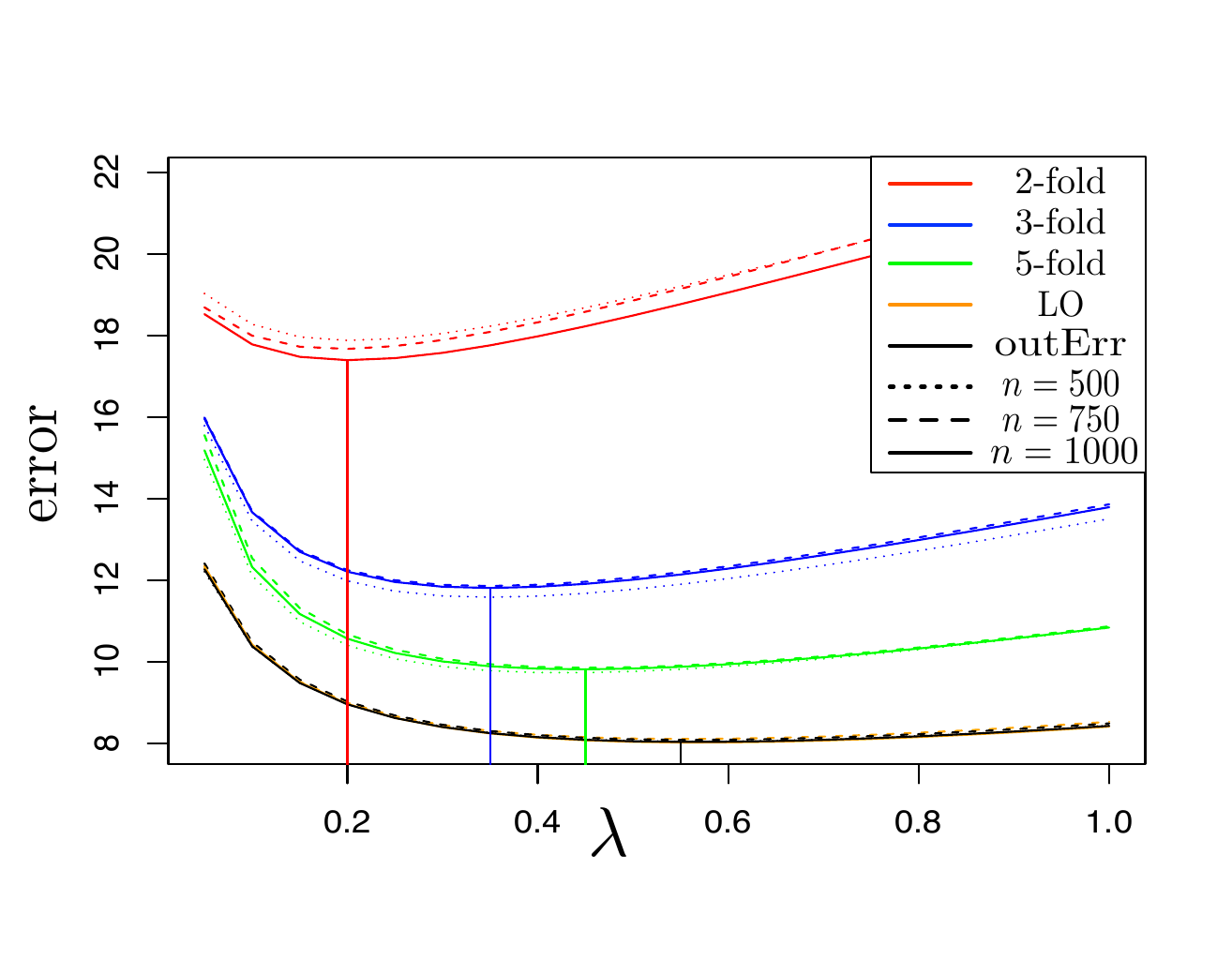}
          \caption{ Comparison of  $K$-fold cross validation (for $K=2,3,5$) and leave-one-out cross validation with the true (oracle-based)  out-of-sample error for the ridge problem where $\fl( y-\vx^{\t}\hvb )=\frac{1}{2}(y-\vx^{\t} \hvb)^2$ and $\reg(\hvb)=\frac{1}{2}\|\hvb\|^2_2$. In high-dimensional settings the upward bias of $K$-fold CV clearly decreases as number of folds increase. Data is $\vy \sim \mathcal{N}(\vX\vbeta^*, 2\vI)$ where $\vX \in \mathbb{R}^{n \times p}$. The components of true coefficients $\vbeta_0$ and the rows of $\sqrt{n}\vX$ independently follow $\mathcal{N}(\vzero,\vSigma)$ where $\vSigma$ is a diagonal matrix with first half of the diagonal elements being 4 and the second half being 1. Dimensions are $(n,p)=(500,400), (750, 600), (1000,800)$ represented by the dot lines, dash lines and solid lines respectively. Out-sample test data is $y_{\rm new} \sim \mathcal{N}(\bm{x}_{\rm new}^\top \bm{\beta}^*, 2)$ where  $\bm{x}_{\rm new} \sim N(\bm{0}, \frac{1}{n}\bm{\Sigma})$. The true (oracle-based) out-of-sample prediction error is $\extra =  \mathbb{E} [  ( y_{\rm new}-\bm{x}_{\rm new}^\top \bm{\hat{\beta}} )^2 | \bm{y,X} ]$. Vertical lines indicates the minimums of the corresponding risk estimates. All depicted quantities are averages based on 50 random independent samples. } 
          \label{fig:loovscv5}
          \end{center}
\end{figure}

\subsection{ The estimates $\lo$ and $\alo$}\label{sec:galo}
Leave-one-out cross validation ($\lo$) offers the following estimate for $\extra$:
\[
	\lo 	\ := \ 	\frac{1}{n} \sum_{i=1}^n \fl(y_i-\vrx{i}^{\t}\lamloovb{i}), 
\]
where 
\BE
	\lamloovb{i}
	\ :=\ 
	\argmin_{\vbeta \in \bbR^{p}}\sum_{j\neq i}^n \fl(y_j-\vrx{j}^{\t}\vbeta)+\lambda \sum_{i=1}^p \reg(\beta_i)
	.	\label{eq:modelLO}
\EE
$\lo$ is computationally  infeasible when both $n$ and $p$ are large. To alleviate this problem, \cite{rad2018scalable} used the following single step of the Newton algorithm (with initialization $\hvb_{\lambda}$)  to approximate the solution of \eqref{eq:bli}:
\BE
	\lamloovb{i}
	&\approx& 
	\hvb_{\lambda} - \fl'(y_i-\vrx{i}^{\t}\hvb_{\lambda})\left(\sum_{j\neq i}^n \vrx{j} \vrx{j}^\top \fl''(y_j-\vrx{j}^{\t}\hvb_{\lambda})+\lambda \diag{\reg''(\hvb_{\lambda})}\right)^{-1} \vrx{i}
	.	\n\\
	\label{eq:aloformula1}
\EE
Then, using Woodbury matrix inversion lemma, in \cite{rad2018scalable} the following approximate leave-one-out formula was derived:
\BE
	\alo
	&:=&
	\frac{1}{n} \sum_{i=1}^n \fl \left(y_i-\vrx{i}^{\t}\hvb_{\lambda}+\frac{\fl'(y_i-\vrx{i}^{\t}\hvb_{\lambda})}{\fl''(y_i-\vrx{i}^{\t}\hvb_{\lambda})}\cdot \frac{H_{ii}}{1-H_{ii}}\right)
	,	\label{eq:looeq}
\EE
where $H_{ii}$ is the $i$th diagonal element of $\vH$ defined by:
\[
	\vX\left(\vX^{\t}\diag{\dfl''(\vy-\vX\hvb_{\lambda})}\vX+\lambda \diag{\dreg''(\hvb_{\lambda})}\right)^{-1}\vX^{\t}\diag{\dfl''(\vy-\vX\hvb_{\lambda})}.
\]

\subsection{Risk estimation with AMP}\label{sec:reAMP}

Besides $\lo$ and $\alo$, another risk estimation technique we study in this paper is based on the approximate message passing (AMP) framework \cite{M11PhD}. For squared loss $\fl (u, y) = \frac{1}{2}(u-y)^2$  and bridge regularizers, \cite{mousavi2018consistent}  used the AMP framework to obtain consistent estimates of $\extra$. In this section, we explain how an estimate $\wampr$ of $\extra$ can be obtained for the more general class of estimators we consider in \eqref{eq:model}. The heuristic approach that leads to the following construction of $\wampr$ is explained in Appendix \ref{sec:construct}. 

\begin{enumerate}

\item Compute $\hvb_{\lambda}$ from \eqref{eq:model}.

\item Find $\hat{\tau}$ that satisfies the following equation:  
\BE
	\lambda
	&=&
	\dotp{\frac{\fl''(\vy-\vX\hvb_{\lambda})}{\frac{1}{\hat{\tau}}+\frac{1}{\delta\lambda}\dotp{\frac{\hat{\vsigma}^2}{\hat{\vsigma}^2+\hat{\tau}\reg''(\hvb_{\lambda})}}\cdot \fl''(\vy-\vX\hvb_{\lambda})}}
	,	\label{eq:taudef}
\EE
where the $i$th component of $\hat{\vsigma}$ is $\|\vcx{i}\|$.

\item Using $\hvb_{\lambda}$ from step 1, and  $\hat{\tau}$ from step 2, define
\BE
	\hat{\theta}
	&:=&
	\frac{1}{\delta\lambda}\dotp{\frac{\hat{\tau}\hat{\vsigma}^2}{\hat{\vsigma}^2+\hat{\tau} \reg''(\hvb_{\lambda})}}.
		\label{eq:thetadef}
\EE

\item Finally, the AMP-based risk estimator is given by 
\BE
	\wampr 
	&:=&
	\frac{1}{n}\sum_{i=1}^n\fl \left(y_i-\vrx{i}^{\t}\hvb_\lambda+ \hat{\theta}\cdot\fl'(y_i-\vrx{i}^{\t} \hvb_{\lambda}) \right)
	.	\label{eq:amprhat}
\EE

\end{enumerate} 
To ensure the existence of $\hat{\tau}$ in step 2, we will show in Lemma \ref{lem:AMP=GLM1} that there is a one-to-one relationship between $\lambda$ and $\hat{\tau}$. Further, we will show in Section \ref{ssec:discamp_lo} that  $y_i-\vrx{i}^{\t}\hvb+ \hat{\theta}\cdot\fl'(y_i-\vrx{i}^{\t} \hvb_{\lambda})$ is close to $y_i-\vrx{i}^{\t}\loovb{i}$. In other words, the term $\hat{\theta} \cdot \fl'(\vy-\vX\hvb_{\lambda})$ corrects the optimistic training error $y_i-\vrx{i}^{\t}\hvb_{\lambda}$, and pushes it closer to the out-of-sample error.



%% file: sec-results.tex
\section{Main results}


\subsection{Assumptions}\label{sec:assume}
In this section, we present and discuss the assumptions used in this paper. Note that we do not require all of the assumptions for any individual result, and some of the assumptions can be weakened or replaced by other assumptions, as we also discuss below. The first few assumptions are about the structural properties of the loss function and regularizer. 

\begin{assumption}\label{ass:Convex}
Loss function $\fl(\cdot)$ and regularizer $\reg(\cdot)$ are convex and have continuous second order derivatives. Moreover, the minimizer of $\reg(\cdot)$ is finite.
\end{assumption}

\begin{assumption}\label{ass:Holder}\label{ass:Smoothness2}
(H{\"o}lder Assumption) The second derivatives of the loss function $\fl$ and regularizer $\reg$ are H{\" o}lder continuous: there exists constants $\alpha\in (0,1]$ and $C_l,C_r>0$ such that for all $|x-x'|\leq 1$, we have
\[
	|\fl''(x)-\fl''(x')|
	\qleq 
	C_l|x-x'|^{\alpha} \qand |\reg''(x)-\reg''(x')|\qleq C_r|x-x'|^{\alpha}
	.
\]
This implies that there exists constants $C>0$ and $\rho\in [0,1]$ such that for all $x\in \bbR$, we have
\BE
	\max \{\fl''(x),\reg''(x)\}
	&\leq& 
	C(1+|x|^\rho)
	,	\n\\ 
	\max \{\fl'(x),\reg'(x)\}
	&\leq& 
	C(1+|x|^{\rho+1})
	,	\n\\
	\max \{\fl(x),\reg(x)\}
	&\leq& 
	C(1+|x|^{\rho+2})
	.	\n
\EE
\end{assumption}

Assumption O.\ref{ass:Holder} ensures that the second derivatives are locally smooth. Given that the original assumptions in the derivation of $\alo$ and AMP are twice differentiability of the loss function and regularizer, Assumption O.\ref{ass:Holder} is only slightly stronger than the twice differentiability assumptions that were used in deriving $\alo$ formula \eqref{eq:looeq}. Note that for non-differentiable cases, one can apply a smoothing scheme similar to the ones proposed in \cite{koh2017understanding,mousavi2018consistent}, and still use these risk estimates. We will explain the smoothing in a few examples below. 

\begin{assumption} \label{ass:Smoothness}
There exists $\ac>0$ such that
\[
	\inf_{x\in \bbR} \fl''(x)
	\qgeq 
	\ac
	.
\]
\end{assumption}
Assumption O.\ref{ass:Smoothness} ensures the uniqueness of the solutions of our optimization problems. In that vain, even if we replace Assumption O.\ref{ass:Smoothness} with $\inf_{x\in \bbR} \lambda\reg''(x) \geq \ac$, then most of our results will still hold. The only exceptions is Lemma \ref{lem:G}. As will become clear, the proof of Lemma \ref{lem:G} requires  $\frac{1}{n} \sum_{i=1}^n \fl''(y_i-\vrx{i}^{\t}\hvb_\lambda) \geq \kappa_\ell$. This in turn, only requires that a constant fraction of the residual fall in the regions at which the curvature of $\ell$ is positive.  Below, we will show several examples in which Assumption O.\ref{ass:Smoothness} is violated, but we still have all our results hold. 

Finally, we should emphasize that since in our proofs we calculate the curvature at and around $\lamloovb{i}$, we only require a lower bound in a neighborhood of these estimates. Furthermore, if the curvature in such neighborhoods goes to zero `slowly', still our risk estimates will be consistent. We will keep the dependency of our bounds on $\ac$ for the readers who are interested in the cases where $\ac$ is not constant and goes to zero. However, for notational simplicity we have considered a global lower bound for the curvature in Assumption O.\ref{ass:Smoothness}, and in all the results will see $\ac$ as a constant. 

 Below we mention several well-known examples that satisfy our assumptions. Note that in many applications non-smooth losses and regularizers, such as LASSO, seem to offer better performance. Given that the constructions of both ALO and AMP-based risk estimates are using the smoothness of the loss and regularizer to apply these formulas we can smooth-out the loss and/or regularizer. Smoothing of the function have also been used extensively for solving such non-differentiable problems \cite{becker2011templates}. For instance, as suggested in \cite{schmidt2007fast}, one can use the following smooth approximation for $|\cdot|$:
 \begin{equation} \label{eq:slasso}
 g_u(x) = (\ln(1+e^{ux}))+\ln(1+e^{-ux}))/u.
 \end{equation}
 It is straightforward to check that $\sup_x |g_u(x)-|x| | \rightarrow 0$ as $u \rightarrow \infty$. 
 
\begin{example}(Smoothed elastic-net) Consider the case were $\ell(x)= \frac{1}{2}x^2$ and $\reg(x) =  \gamma x^2 + (1- \gamma)g_u(x)$. It is clear that both the loss function and regularizer are convex and have continuous second order derivatives and achieve one unique minimizer at $0$. Further, note that 
\begin{eqnarray}
\frac{\partial^2g_u(x)}{\partial x^2}&=& \frac{2u\cdot e^{ux}}{(1+e^{ux})^2} > 0\n\\
\left|\frac{\partial^3g_u(x)}{\partial x^3}\right| & = & \left|\frac{2u^2e^{ux}(1-e^{ux})}{(1+e^{ux})^3}\right| \kleq  \left|\frac{2u^2e^{ux}}{(1+e^{ux})^2}\right|\kleq \frac{u^2}{2}.\label{eq:g3}
\end{eqnarray}
Hence, it is straightforward to check that Assumption O.\ref{ass:Convex} holds. Furthermore, Assumption O.\ref{ass:Holder} holds with constant $C_r = u^2/2, \alpha = 1$ and any positive constant $C_l>0$. Finally, it is clear that Assumption O.\ref{ass:Smoothness} holds due to $\ell''(u)= 1$.
 \end{example}
 
 \begin{example} (Smoothed-bridge estimators) Consider an estimation problem  with $\ell(x) = \frac{1}{2}x^2$ and $\reg(x) = g_u^q(x)$, where $q>1$. Note that $R(x)$ is a smooth approximation for bridge estimators. Similar to Example 1, Assumption O.\ref{ass:Convex} and Assumption O.\ref{ass:Smoothness} hold. Furthermore, Assumption O.\ref{ass:Holder} holds with constant $C_r = u^2/2, \alpha = 1$ and any positive constant $C_l>0$. 
 \end{example}
 
 \begin{example}\label{ex:PH_ELAS} (Pseudo-Huber loss and elastic-net) Consider the estimation problem $\ell(x) = h_{v}(x)$ and $\reg(x) =  \gamma x^2 + (1- \gamma)g_u(x)$, where $h_{v}(x)$ is the Pseudo-Huber loss with parameter $v$, i.e., 
 $$h_{v}(x):=v^2(\sqrt{1+(x/v)^2}-1).$$
 The Pseudo-Huber loss function is used in robust estimation and is s smooth approximation of the Huber loss function. The second and third derivatives of $h_{v}(x)$ are given by
 \begin{eqnarray}
\frac{\partial^2 h_{v}(x)}{\partial x^2} &=& \left(1+(x/v)^2\right)^{-\frac{3}{2}}, \n\\
\frac{\partial^3 h_{v}(x)}{\partial x^3} &=& -\frac{3}{v}\frac{x/v}{\left(1+(x/v)^2\right)^{\frac{5}{2}}}.\n
\end{eqnarray} 
Hence, combining these results with \eqref{eq:g3}, we conclude that Assumptions O.\ref{ass:Convex} and O.\ref{ass:Holder} hold with constants $C_l = \frac{3}{2v}, C_r = u^2/2, \alpha = 1$. Note that $\frac{\partial^2 h_{v}(x)}{\partial x^2} \rightarrow 0$ as $x\rightarrow \infty$, and therefore Pseudo-Huber loss function does not directly satisfy Assumption O.\ref{ass:Smoothness}. However, since $\reg''(x) \geq 2\gamma>0$, as mentioned in previous discussion, our theorems hold when there exists a constant $\kappa_\ell>0$ such that $\frac{1}{n} \sum_{i=1}^n \fl''(y_i-\vrx{i}^{\t}\hvb_\lambda) \geq \kappa_\ell$. It is clear that this additional assumption holds for this example when a non-zero fraction of residuals are bounded. We will verify this claim heuristically using AMP framework in Appendix \ref{sec:construct}.   
\end{example}
 
 \begin{example}\label{ex:LAD_ELAS}
 (Smoothed least absolute deviation and elastic-net) Consider an estimation problem with $\ell(x) = g_{v}(x)$ and $\reg(x) =  \gamma x^2 + (1- \gamma)g_u(x)$. Assumption O.\ref{ass:Convex} and Assumption O.\ref{ass:Holder} hold with constant $C_l = C_r = u^2/2$ and $\alpha = 1$. Note that from \eqref{eq:g3}, $\frac{\partial^2g_u(x)}{\partial x^2}\rightarrow 0$ as $x\rightarrow \infty$, therefore the loss function does not directly satisfy Assumption O.\ref{ass:Smoothness}. However, similar to Example 3, we can replace Assumption O.\ref{ass:Smoothness} with the assumption that a fraction of residuals $y_i-\vrx{i}^{\t}\hvb_\lambda$ are bounded and we will verify this heuristically in Appendix \ref{sec:construct}.    
 \end{example}

So far, the assumptions have been concerned with the geometric properties of the loss function and the regularizer. The rest of our assumptions are about the statistical properties of the problem.

\begin{assumption}\label{ass:True}
Let $\vbeta_0, \vw$ and $\vX$ be mutually independent random variables. Furthermore, we assume each data point $(y_i,\vrx{i})$ is i.i.d.~generated, $y_i \spaceequal \vrx{i}^\top\vbeta_0+w_i$, and the $j$th element of $\vrx{i}$ is an independent mean $0$ random variables with variance $\sigma_j^2/n$. Let $\vSigma=\diag{\sigma_1^2,\ldots,\sigma_p^2}$ and we assume there exists absolute constants $c_l,c_u>0$ such that $c_l\leq \sigma_i^2\leq c_u$ for all $i$. We assume that the entries of $\vw, \vX(\vSigma/n)^{-\frac{1}{2}}$ are subGaussian random variables respectively, i.e., there exists a constant $C$ such that for any fixed $r\geq 1$, and for all $i,j$, we have
\[
	(\bbE |w_i|^r)^{\frac{1}{r}}
	\qleq C \sqrt{r} \qand
	(\bbE |x_{i,j}|^r)^{\frac{1}{r}}
	\qleq C \sigma_j\sqrt{\frac{r}{n}}.
\]
Finally, we make the following assumption on the true coefficients $\vbeta_0$:
\begin{itemize}
\item When $\vbeta_0(p)$ is a deterministic sequence indexed with $p$, we assume that 
\[
\sup_i|\beta_{0,i}|=O\left(\plogn\right), \quad \text{as~}p,n\rightarrow \infty,
\] 
and $\frac{1}{p}\sum_{i}\beta_{0,i}^{2(\rho+1)}\leq C$ holds for a universal constant $C$ for all $p>0$ where $\rho$ is the constant in Assumption O.\ref{ass:Holder}.
\item When $\vbeta_0$ is a random vector, we assume that the entries of $\vbeta_0$ are i.i.d.~subGaussian random variables, i.e., there exists a constant $C$ such that for any fixed $r\geq 1$ and $i$, we have $(\bbE |\beta_{0,i}|^r)^{\frac{1}{r}}\leq C \sqrt{r}$.
\end{itemize} 
\end{assumption}

Assumption O.\ref{ass:True} is a standard assumption  in the high-dimensional asymptotic analysis of  regularized estimators \cite{donoho2015variance,bradic2015robustness,el2013robust,bean2013optimal,el2018impact,sur2017likelihood,weng2018overcoming,johnstone2001distribution,bayati2012lasso,thrampoulidis2015regularized,amelunxen2014living,chandrasekaran2012convex,cai2016geometric}. Note that extensive empirical results presented elsewhere \cite{weng2018overcoming}, \cite{mousavi2018consistent} have confirmed that the conclusions obtained from this framework are also accurate even when the elements of the design matrix $\vX$ are weakly dependent.   Also, using the techniques proposed in \cite{bayati2012lasso} the assumption of independence of $\vbeta_0$ can be weakened to the assumption that the empirical CDF of the regression coefficients converge weakly to a valid CDF.

Suppose that we have a sequence of problem instances $(\vbeta_0(p),\vw(p), \vX(p))$ indexed with $p$ (with fixed $n/p=\delta >1$), and each problem instance satisfies Assumption O.\ref{ass:True}. Then, solving \eqref{eq:model} for the sequence of problem instances $(\vbeta_0(p),\vw(p), \vX(p))$  leads to a sequence of estimates $\hvb_{\lambda}(p)$. Our last assumption is about this sequence.

\begin{assumption}\label{ass:boundhvb}
Every component of $\hvb_{\lambda}(p)$ remains bounded by a sufficiently small power of $n$. More specifically, 
\[
	\sup_{i=1,\cdots,p} |\ve_i^{\t}\hvb_{\lambda} (p)|^\rho 
	\qleq 
	\op{\bc}
	,
\]
where $\bc$ is a constant that satisfies $c_n = o(n^{\alpha^2/4})$. $\alpha, \rho$ are the constants stated in O.\ref{ass:Holder}. 
\end{assumption}

Note that in this paper, we use $\hvb_{\lambda}$ and $\hvb_{\lambda}(p)$ interchangeably. Assumption O.\ref{ass:boundhvb} requires every component of the original estimate $\hvb_{\lambda}$ to be bounded. One can heuristically argue that this assumption holds given Assumption O.\ref{ass:Convex}-O.\ref{ass:True}. Let us mention a heuristic argument here. Suppose that Assumptions O.\ref{ass:Convex}-O.\ref{ass:True} hold. We can show that (See \eqref{eq:add_tue_2.33pm} and \eqref{eq:add_tue_2.35pm} in the proof for Lemma \ref{lem:supnormDi} in Section \ref{sec:supnormDi}) 
\[
	\frac{1}{\sqrt{n}}\|\hvb_{\lambda}-\vbeta_0\|
	\spaceequal
	\op{\frac{1}{\ac}}
	.
\]
Hence, on average, the component-wise distance between $\hvb_{\lambda}$ and $\vbeta_0$ should be $\op{\frac{1}{\ac}}$. Note that according to Assumption O.\ref{ass:True}, every component of the true signal $\vbeta_0$ can be bounded by $\op{\plogn}$ (See Lemma \ref{lem:xwnorm}). Therefore, intuitively, every component of $\hvb_{\lambda}$ should be bounded by $\op{\frac{\plogn}{\ac}}$ as well. In fact, we can show that O.\ref{ass:boundhvb} holds with $\bc=\op{\frac{\plogn}{\ac^{15\rho}}}$, if we assume O.\ref{ass:Convex}-O.\ref{ass:True} hold, and the regularizer satisfies an extra condition. This is described in the following lemma: 

\begin{lemma}\label{lem:ABOUTASSUMPO5}
Suppose that Assumptions O.\ref{ass:Convex}-O.\ref{ass:True} are satisfied. Furthermore, suppose that the regularizer satisfies at least one of the following conditions: 
\BI
\item[(a)] $\reg''(x)$ is Lipchitz and $\sup_i\reg''(\ve_i^{\t}\hvb_{\lambda})=\op{\plogn}$.
\item[(b)] There exists constant $c>0$ such that $\sup_{x\in \bbR}\reg''(x)<c$.
\item[(c)] There exists constant $c>0$ such that $\inf_{x\in \bbR}\lambda\reg''(x)>c$.
\EI
Then, $\sup_i|\ve_i^{\t}\hvb_{\lambda} (p)|^\rho  =\op{\frac{\plogn}{\ac^{15\rho}}}$.  
\end{lemma}
Since the proof of this lemma uses some of the results we will prove in later sections, we postpone it to Appendix \ref{sec:disbi}.

\subsection{Main results} \label{sec:results}

In this section, we address the questions about the convergence rate mentioned in the introduction for  $\lo$, $\alo$ and $\wampr$. Our first result bounds the discrepancy of $\wampr$ and $\lo$.
\begin{theorem}\label{THM:MSE} Assuming O.\ref{ass:Convex}-O.\ref{ass:boundhvb}, for any fixed $\lambda>0$, we have
\[
	\left|\lo-\wampr\right|
	\spaceequal
	\op{\frac{\plogn \cdot \bc^{1+\alpha}}{n^{\frac{\alpha^2}{2}}\cdot \ac^{72\rho+22}}}
	.
\] 
\end{theorem}

A proof sketch of the theorem is presented in Section \ref{ssec:discamp_lo} and details are in Section \ref{sec:MSE}. Next, we obtain an upper bound for the discrepancy between $\lo$ and $\alo$. 

\begin{theorem}\label{thm:jixu}
Under Assumptions O.\ref{ass:Convex}-O.\ref{ass:True}, we have 
\BE
	\left|\lo-\alo\right|
	\spaceequal
	\op{\frac{\plogn}{n^{\frac{\alpha^2}{2}}\cdot \ac^{13\rho+12}}}
	.	\label{eq:Con_eq_1}
\EE 
\end{theorem}
The proof of Theorem \ref{thm:jixu} is presented in Section \ref{ssec:disclo_lao}. Finally, we find an upper bound for $|\lo-\extra|$.

The following result provides an upper bound on the difference between $\lo$ and $\extra$. 

\begin{theorem}\label{thm:LOconsistency}
Under Assumptions O.\ref{ass:Convex}-O.\ref{ass:True}, and $\vrx{i}\stackrel{\text{i.i.d.}}{\sim}\mathcal{N}(\boldsymbol{0},\frac{1}{n}\vSigma)$, we have 
\[
	|\lo-\extra |
	\spaceequal
	\left|\frac{1}{n}\sum_{i=1}^n\fl(y_i-\vrx{i}^{\t}\loovb{i}_{\lambda})-\bbE[\fl(y_{new}-\vrx{new}^{\t}\hvb_{\lambda})\big|\mathcal{D}]\right|
	\spaceequal
	\op{\frac{\plog n}{\ac^{4\rho+5}\sqrt{n}}}
	.	
\]
\end{theorem}

A proof sketch of Theorem \ref{thm:LOconsistency} is presented in Section \ref{ssec:disc_lo_extra}.  Note that Theorem \ref{thm:LOconsistency} requires $\vrx{i}$ follow the Gaussian distribution. There is one place that this assumption is required in the proof and that is in Lemma \ref{lem:socomp}. If one finds a way to prove that $\bbE (\sigma_{\min}(\vX^{\t}\vX))^{-r}$ can be upper bounded by a constant for all $r>0$ for subgaussian matrices (which is expected to hold), then we can obtain the same result for subgaussian matrices.

\subsection{Tightness of the results}

First let us discuss the tightness of Theorem \ref{thm:LOconsistency}. Suppose that after obtaining $\hvb_{\lambda}$ an oracle would give us $n$ independent new samples for estimating the risk. It is then straightforward to use the central limit theorem and argue that even if we use $n$ new samples the error of our risk estimate will be $O_p(\frac{1}{\sqrt{n}})$.  Hence, the result of Theorem \ref{thm:LOconsistency} is tight up to a logarithmic factor.   
Regarding Theorem \ref{thm:jixu} first note that if for instance the loss function and the regularizer are three times continuously differentiable, then we have
\[
\left|\lo-\alo\right| = 	\op{\frac{\plogn}{n^{\frac{1}{2}}}}
\] 
Note that since the difference of $\lo$ and out-of-sample prediction error is $O_p(\frac{\plogn}{\sqrt{n}})$, the error of our approximation is at the same order as the error of $\lo$. Hence, the approximation is as good as we want it to be. That said, we should emphasize that this argument is not claiming that the result we obtain for $\left|\lo-\alo\right|$ is tight for all three times differentiable losses and regularizers. In order to obtain sharp results one should make some assumptions about the third order derivatives of the loss and regularizer as well. Note that the closer the function is to the quadratic, e.g. the closer the third derivative is to zero, we expect the approximation error of $\alo$ to be smaller than $\op{\frac{\plogn}{n^{\frac{1}{2}}}}$. For instance, if the loss and regularizer are quadratic, $\left|\lo-\alo\right|=0$. Given that obtaining such accurate results require more assumptions and do not offer any particular gain, we did not pursue that direction.  
The result of Theorem \ref{THM:MSE} is also similar to the result of Theorem \ref{thm:jixu} and a similar argument can be given about the tightness of the result. Hence, we do not repeat the argument here.

\subsection{Proof sketch of the main results}
\subsubsection{Proof sketch of Theorem \ref{THM:MSE}}\label{ssec:discamp_lo}

Below, we sketch the proof of Theorem \ref{THM:MSE}. Details are in Section \ref{sec:MSE}. As the first step, in Lemma \ref{lem:AMP=GLM1}, we show that $\hat{\tau}$, introduced in \eqref{eq:taudef}, is uniquely defined. Hence, the heuristic recipe we mentioned in Section \ref{sec:reAMP} leads to a well-defined estimate for $\extra$. Note that Lemma \ref{ssec:discamp_lo} does not provide any information on the quality of this estimate.  

\begin{lemma}\label{lem:AMP=GLM1}
Under Assumption O.\ref{ass:Convex}, for any $(\vX,\vy,\vbeta)$, 
\BE
	\gamma
	\spaceequal
	\dotp{\frac{\fl''(\vy-\vX\vbeta)}{\frac{1}{\tau}+\frac{1}{\delta\gamma}\dotp{\frac{\hat{\vsigma}^2}{\hat{\vsigma}^2+\tau\reg''(\vbeta)}}\cdot \fl''(\vy-\vX\vbeta)}},
	\label{eq:amp=glm_eq1}
\EE
defines a one-to-one mapping between $\gamma\in \bbR^{+}$ and $\tau\in \bbR^{+}$.
\end{lemma}
The proof  is given in Section \ref{sec:uniquesolution}.  According to this lemma, a unique value of $\hat{\tau}$ satisfies \eqref{eq:taudef}. Using this unique value we can calculate the unique $\hat{\theta}$ that satisfies \eqref{eq:thetadef}, and obtain the following estimate of $\extra$:
\begin{equation}\label{eq:wampr_estimate}
\wampr = \sum_{i=1}^n \ell(y_i-\vrx{i}^{\t}\hvb_{\lambda}+ \hat{\theta}\cdot\fl'(y_i-\vrx{i}^{\t} \hvb_{\lambda}))
\end{equation}
 To compare this risk estimate with $\lo$, we first simplify $\lo$ in the following proposition.

\begin{proposition}\label{thm:looo}
Under Assumptions O.\ref{ass:Convex}-O.\ref{ass:True}, we have 
\[
	\sup_{i,j}\left|\ve_j^{\t}(\hvb_{\lambda}-\loovb{i}_{\lambda})\right|
	\spaceequal
	\op{\frac{\plogn }{n^{\frac{\alpha}{2}}\cdot \ac^{8\rho+7}}}
	.
\]
Moreover, if we define
\[ \tvepsilon^i
	:=
	\loovb{i}_{\lambda} - \hvb_{\lambda} + \fl'(y_i-\vrx{i}^{\t}\hvb_{\lambda})\oA_i^{-1}\vrx{i}
	,
\]
where 
\[
	\oA_i
	:=
	\vX^{\t}\diag{\dfl''(\vy-\vX\loovb{i}_{\lambda})}\vX+\lambda\diag{\dreg''(\loovb{i}_{\lambda})}-\fl''(y_i-\vrx{i}^{\t}\loovb{i}_{\lambda})\vrx{i}\vrx{i}^{\t}
	,
\] 
then,
\[
	\sup_i\|\tvepsilon^i\|
	\spaceequal
	\op{\frac{\plogn }{n^{\frac{\alpha}{2}}\cdot \ac^{9\rho+10}}}.
\]
\end{proposition}

The proof of Proposition \ref{thm:looo} is given in Section \ref{sec:looo}. For the special case of $\ell_2$ regularizer, a similar upper bound is obtained for $\sup_i\|\tvepsilon^i\|$ in Theorem 2.2 of \cite{el2018impact}. We employ a similar proof strategy. However, due to the lack of lower bound for the curvature of the regularizer, our argument is more involved. Given the definitions of $\tvepsilon^i$ and $\oA_i$, we have
\BE
	y_i-\vrx{i}^{\t}\loovb{i}_{\lambda}
	&=&
	y_i-\vrx{i}^{\t}\hvb_{\lambda}+\fl'(y_i-\vrx{i}^{\t}\hvb_{\lambda})\vrx{i}^{\t}\oA_i^{-1}\vrx{i}-\vrx{i}^{\t}\tvepsilon^i
	,	\label{eq:loomain1}
\EE
and hence 
\begin{equation}
\lo \spaceequal \sum_{i=1}^n l \left (y_i-\vrx{i}^{\t}\hvb_{\lambda}+\fl'(y_i-\vrx{i}^{\t}\hvb_{\lambda})\vrx{i}^{\t}\oA_i^{-1}\vrx{i}-\vrx{i}^{\t}\tvepsilon^i \right).\n
\end{equation}

Next, with the aid of Proposition \ref{thm:looo}, we prove that  the AMP-based residuals $y_i-\vrx{i}^{\t}\hvb_{\lambda}+ \hat{\theta}\cdot\fl'(y_i-\vrx{i}^{\t} \hvb_{\lambda})$ in \eqref{eq:amprhat} are close to the leave-$i$-out residuals $y_i-\vrx{i}^{\t}\lamloovb{i}$ in \eqref{eq:modelLO}. In that vein,
\BE
	\lefteqn{\left| (y_i-\vrx{i}^{\t}\hvb_{\lambda}+ \hat{\theta}\cdot\fl'(y_i-\vrx{i}^{\t} \hvb_{\lambda}))  -(y_i-\vrx{i}^{\t}\loovb{i}_{\lambda})\right|
	=
	\left|(\hat{\theta}-\vrx{i}^{\t}\oA_i^{-1}\vrx{i})\fl'(y_i-\vrx{i}^{\t}\hvb_{\lambda})+\vrx{i}^{\t}\tvepsilon^i\right|}
		\n\\
	&\leq&
	|\fl'(y_i-\vrx{i}^{\t}\hvb_{\lambda})|\cdot  |\hat{\theta}-\vrx{i}^{\t}\oA_i^{-1}\vrx{i}| \n
	+\|\vrx{i}\|\cdot \op{\frac{\plogn }{n^{\frac{\alpha}{2}}\cdot \ac^{9\rho+10}}}
	,	\n \ \ \ \ \ \ \ \ \ \ \ \ \ \ \ \ \ \ \ \ \ \ \ \ \ \ \ \ 
\EE
where the last inequality is due to Proposition \ref{thm:looo}. Recall that based on Assumption O.\ref{ass:True}, the entries of $\vrx{i}$ are independent mean $0$ subGaussian random variables with covariance matrix $\vSigma/n$, resulting in $\sup_i\|\vrx{i}\| = \op{1}$, as proved in Lemma \ref{lem:xwnorm}.  Hence, our next main objective is to bound
\BE
	\sup_{i=1,\cdots,n} |\fl'(y_i-\vrx{i}^{\t}\hvb_{\lambda})|\cdot  |\hat{\theta}-\vrx{i}^{\t}\oA_i^{-1}\vrx{i}| \n
	\EE
Towards this goal, we prove
\begin{eqnarray}
\sup_i\left|\hat{\theta}-\vrx{i}^{\t}\oA_i^{-1}\vrx{i}\right| &=& 	\op{\frac{\plogn \cdot \bc^{1+\alpha}}{n^{\frac{\alpha^2}{2}}\cdot \ac^{67\rho+19}}}, \label{eq:firstboundAMP1}\\
\sup_i|\fl'(y_i-\vrx{i}^{\t}\hvb_{\lambda})|&=& \op{\frac{\plogn}{\ac^{4\rho+2}}}.\label{eq:firstboundAMP2}
\end{eqnarray}

Our first lemma bounds $\sup_i|\fl'(y_i-\vrx{i}^{\t}\hvb_{\lambda})|$.

\begin{lemma}\label{lem:supnormDi}
Under Assumptions O.\ref{ass:Convex}-O.\ref{ass:True}, for large enough $n$, we have
\BE
	\sup_i|y_i-\vrx{i}^{\t}\loovb{i}_{\lambda}|
	&=&
	\op{\frac{\ln n}{\ac}}
	,	\n\\
	\sup_{i}\|\hvb_{\lambda}-\loovb{i}_{\lambda}\|
	&=&
	\op{\frac{\plogn}{\ac^{\rho+2}}}
	,	\n\\
	\sup_i|y_i-\vrx{i}^{\t}\hvb_{\lambda}|
	&=&
	\op{\frac{\plogn}{\ac^{\rho+2}}}
	.	\n
\EE
\end{lemma}

The proof can be found in Section \ref{sec:supnormDi}. By Lemma \ref{lem:supnormDi} and Assumption O.\ref{ass:Smoothness2}, we have 
\BE
	\sup_i|\fl'(y_i-\vrx{i}^{\t}\hvb_{\lambda})|
	&\leq&
	O(1)\cdot (\sup_i|y_i-\vrx{i}^{\t}\hvb_{\lambda}|^{\rho+1}+1)
	\qleq
	\op{\frac{\plogn}{\ac^{4\rho+2}}}
	.	\n
\EE	
Hence, \eqref{eq:firstboundAMP2} holds. The final step of the proof is to bound $\sup_i\left|\hat{\theta}-\vrx{i}^{\t}\oA_i^{-1}\vrx{i}\right|$. Toward this goal we first want to prove that 
\begin{eqnarray}
	\sup_i\left|\vrx{i}^{\t}\oA_i^{-1}\vrx{i}-\frac{1}{n}\trace(\oA_i^{-1}\vSigma)\right|
	\spaceequal
	\op{\frac{\plogn}{\sqrt{n}\cdot \ac}}. \label{eq:argument1}
\end{eqnarray}
Note that $\oA_i$ is the Hessian matrix of the objective in \eqref{eq:bli} evaluated at the corresponding leave observation-$i$ out estimate. Therefore it is independent of $\vrx{i}$. Further, entries of $\vrx{i}$ have subGaussian tails. Hence, the following lemma, which is a standard concentration result, can address this issue:
\begin{lemma}\label{lem:TAIL}
Let $\vx_i \in \bbR^p, i\in[n]$ be $n$ mean-zero random vectors with covariance matrix $\frac{1}{n}\vI$. Further all entries of $\vx_i, i\in[n]$ are independent with subGaussian tails and the subGaussian parameters are uniformly bounded by some absolute constant. Let $\vGamma_i\in \bbR^{p\times p}, i=1,\ldots n$ be $n$ random matrices. Each $\vGamma_i$ is independent of $\vx_i$. Further, let $C_{n,\delta}$ be an upper bound for the the maximum eigenvalues of all $\vGamma_i$ with probability $1-\delta'$. Then for large enough $n$, with probability $1-\delta'-\frac{2}{n}$, there exists a constant $c$ independent of $n$ such that
\[
	\sup_i|\vx_i^{\t}\vGamma_i\vx_i-\frac{1}{n}\trace(\vGamma_i)|
	\qleq
	c\cdot\frac{C_{n,\delta}\ln n}{\sqrt{n}}
	.
\]  
\end{lemma}
See the proof of this lemma in Section \ref{sec:tail}. Lemma \ref{lem:TAIL} requires the maximum eigenvalues of all $\oA_i^{-1}$s to be bounded. Note that, for the minimal eigenvalue of $\oA_i$, we have
\BEQ
\label{eq:lowerbound_example_eq1}
\BS
	\inf_i\sigma_{\min}(\oA_i)
	&=
	\inf_i\min_{\|\vu\|=1}\vu^{\t}\left(\vX^{\t}\diag{\dfl''(\vy-\vX\loovb{i}_{\lambda})}\vX-\fl''(y_i-\vrx{i}^{\t}\loovb{i}_{\lambda})\vrx{i}\vrx{i}^{\t}\right)\vu
		\\
	&\quad +\vu^{\t}\lambda\diag{\dreg''(\loovb{i}_{\lambda})}\vu
		\\
	&\stackrel{\text{(i)}}{\geq}
	\inf_i\min_{\|\vu\|=1}\vu^{\t}\left(\XI^{\t}\cdot \ac\vI\cdot \XI\right)\vu
		\\
	&\stackrel{\text{(ii)}}{=}
	\Omega_p\left(\ac\right)
	, 	
\end{split}	
\EEQ
where Inequality (i) is due to Assumption O.\ref{ass:Convex} and O.\ref{ass:Smoothness} and Inequality (ii) is due to Lemma \ref{lem:minev} (stated in Section \ref{sec:Preliminaries}). Hence, the maximum eigenvalue of $\oA_i^{-1}$ is upper bounded by $\op{1/\ac}$. Therefore Lemma \ref{lem:TAIL} implies that
\begin{eqnarray}
	\sup_i\left|\vrx{i}^{\t}\oA_i^{-1}\vrx{i}-\frac{1}{n}\trace(\oA_i^{-1}\vSigma)\right|
	\spaceequal
	\op{\frac{\plogn}{\sqrt{n}\cdot \ac}}
	.	\label{eq:argument1}
\end{eqnarray}
Hence, in order to prove \eqref{eq:firstboundAMP1} we need to prove that
\[
	\sup_i\left|\hat{\theta}-\frac{1}{n}\trace(\oA_i^{-1}\vSigma)\right|
	\spaceequal
	\op{\frac{\plogn \cdot \bc^{1+\alpha}}{n^{\frac{\alpha^2}{2}}\cdot \ac^{67\rho+19}}}
	.
\]
To achieve this goal, let us define
\[
	\pB
	\spaceequal
	\vX^{\t}\diag{\dfl''(\vy-\vX\hvb_{\lambda})}\vX+\lambda \diag{\cdreg''(\hvb_{\lambda})}
	.
\] 

It turns out that $\frac{1}{n}\trace(\pB^{-1}\vSigma)$ is very close to $\frac{1}{n}\trace(\oA_{i}^{-1}\vSigma)$ and hence, we only need to bound $\left|\hat{\theta}-\frac{1}{n}\trace(\pB^{-1}\vSigma)\right|$. The next lemma proves this claim. 
\begin{lemma}\label{lem:ABswitch}
Under Assumptions O.\ref{ass:Convex}-O.\ref{ass:True}, for large enough $n$, we have
\BE
	\sup_i\left|\frac{1}{n}\trace(\pA_{i}^{-1}\vSigma)-\frac{1}{n}\trace(\pB^{-1}\vSigma)\right|
	&=&
	\op{\frac{1}{n\cdot \ac}}
		,\n\\
	\sup_i\left|\frac{1}{n}\trace(\pA_{i}^{-1}\vSigma)-\frac{1}{n}\trace(\oA_{i}^{-1}\vSigma)\right|
	&=&
	\op{\frac{\plogn }{n^{\frac{\alpha^2}{2}}\cdot \ac^{8\rho+9}}}
	,	\n	
\EE
where
\[
	\pA_{i}
	\ :=\
	\vX^{\t}\diag{\dfl''(\vy-\vX\hvb_{\lambda})}\vX-\fl''(y_i-\vrx{i}^{\t}\hvb_{\lambda})\vrx{i}\vrx{i}^{\t}+\lambda \diag{\cdreg''(\hvb_{\lambda})}
	.
\]

\end{lemma}
The proofs of these two lemmas can be found in Sections \ref{sec:ABswitch}.  As we described above the goal is to bound $\left|\hat{\theta}-\frac{1}{n}\trace(\pB^{-1}\vSigma)\right|$. We remind the reader that the parameter $\hat{\theta}$ is obtained from \eqref{eq:taudef} and \eqref{eq:thetadef}. In other words, one has to solve the fixed point equation \eqref{eq:taudef} and then plug that in \eqref{eq:thetadef} to obtain $\hat{\theta}$. However, it is also clear that by rearranging \eqref{eq:taudef} and \eqref{eq:thetadef} 
we can see $\hat{\theta}$ as a  solution of a fixed point equation too. More specifically, it is straightforward to plug \eqref{eq:taudef} in \eqref{eq:thetadef} and obtain
\BE
	\frac{\lambda}{\hat{\tau}}
	&=& 
	\dotp{\frac{\fl''(\vy-X\hvb_{\lambda})}{1+\hat{\theta}\fl''(\vy-X\hvb_{\lambda})}}
	.	\label{eq:limiteqc_1}
\EE
We can use this equation to obtain $\hat{\tau}= \lambda \dotp{\frac{\fl''(\vy-X\hvb_{\lambda})}{1+\hat{\theta}\fl''(\vy-X\hvb_{\lambda})}}^{-1}$.  Finally, note that \eqref{eq:taudef} can be expressed in the following form:
\BE
	\frac{1}{\delta}\dotp{\frac{\hat{\vsigma}^2}{\hat{\vsigma}^2+\hat{\tau}\reg''(\hvb_{\lambda})}}
	&=&
	\dotp{\frac{\frac{1}{\delta\lambda}\dotp{\frac{\hat{\vsigma}^2}{\hat{\vsigma}^2+\hat{\tau}\reg''(\hvb_{\lambda})}}\cdot\fl''(\vy-\vX\hvb_{\lambda})}{\frac{1}{\hat{\tau}}+\frac{1}{\delta\lambda}\dotp{\frac{\hat{\vsigma}^2}{\hat{\vsigma}^2+\hat{\tau}\reg''(\hvb_{\lambda})}}\cdot \fl''(\vy-\vX\hvb_{\lambda})}}
	,	\label{eq:62920_eq1}
\EE
which is equivalent to  
\BE
	\frac{1}{\delta}\dotp{\frac{\hat{\vsigma}^2}{\hat{\vsigma}^2+\hat{\tau}\reg''(\hvb_{\lambda})}}
	&=&
	1-\dotp{\frac{1}{\frac{1}{\delta\lambda}\dotp{\frac{\hat{\tau}\hat{\vsigma}^2}{\hat{\vsigma}^2+\hat{\tau}\reg''(\hvb_{\lambda})}}\fl''(\vy-\vX\hvb_{\lambda})+1}}
	.	\label{eq:tauconditionc}
\EE
By plugging $\hat{\tau}= \lambda \dotp{\frac{\fl''(\vy-X\hvb_{\lambda})}{1+\hat{\theta}\fl''(\vy-X\hvb_{\lambda})}}^{-1}$ and \eqref{eq:thetadef} in this equation we obtain
\begin{equation}\label{eq:thetahateq}
	\dotp{\frac{1}{1+\hat{\theta}\fl''(\vy-\vX\hvb_{\lambda})}}+\frac{1}{\delta}\dotp{\frac{1}{1+\lambda\dotp{\frac{\fl''(\vy-\vX\hvb_{\lambda})}{1+\hat{\theta}\fl''(\vy-\vX\hvb_{\lambda})}}^{-1}\frac{\reg''(\hvb_{\lambda})}{\hat{\vsigma}^2}}}=1.
\end{equation}
Given that the solution for $\hat{\tau}$ is unique  (according to Lemma \ref{lem:AMP=GLM1}), the solution for $\hat{\theta}$ shall be unique as well. Define
\begin{equation}\label{eq:defG}
	G(\theta)
	\spaceequal
	\dotp{\frac{1}{1+\theta\fl''(\vy-\vX\hvb_{\lambda})}}+\frac{1}{\delta}\dotp{\frac{1}{1+\lambda\dotp{\frac{\fl''(\vy-\vX\hvb_{\lambda})}{1+\theta\fl''(\vy-\vX\hvb_{\lambda})}}^{-1}\frac{\reg''(\hvb_{\lambda})}{\hat{\vsigma}^2}}}
	.
\end{equation}
Since we would like to prove that $\left|\hat{\theta}-\frac{1}{n}\trace(\pB^{-1}\vSigma)\right|$ is small, we expect $G(\frac{1}{n}\trace(\pB^{-1}\vSigma))$ to be close to $G(\hat{\theta})=1$. Our next lemma shows how we can obtain an upper bound on  $|G(\frac{1}{n}\trace(\pB^{-1}\vSigma)) - G(\hat{\theta})|$. The next step will be to use the mean value theorem to obtain an upper bound on $\left|\hat{\theta}-\frac{1}{n}\trace(\pB^{-1}\vSigma)\right|$. 
\begin{lemma}\label{lem:G}
Suppose Assumptions O.\ref{ass:Convex}-O.5 hold. Consider the function $G$ defined in \eqref{eq:defG}
\BE
	\left|G(\hat{\theta})-G\left(\frac{1}{n}\trace(\pB^{-1}\vSigma)\right)\right|
	\spaceequal
	\op{\frac{\plogn\cdot \bc^{1+\alpha}}{n^{\frac{\alpha^2}{2}}\cdot \ac^{64\rho+16}}},
	\label{eq:add_goal_21}
\EE
and
\BE
	\left|G'(\theta)\right|\geq\Omega_p\left(\frac{\ac^{3\rho+1}}{(1+\theta^2)\plogn}\right)
	.	\label{eq:add_goal_22}
\EE
\end{lemma}
The proof of Lemma \ref{lem:G} can be found in Section \ref{sec:G}. As we discussed before, the next step is to use \eqref{eq:add_goal_21} and the mean value theorem to obtain an upper bound on $\left|\hat{\theta}-\frac{1}{n}\trace(\pB^{-1}\vSigma)\right|$. The main issue however, is that $\theta$ appears in the lower bound of the derivative in \eqref{eq:add_goal_22}. Hence, before applying the mean value theorem we have to prove that both $\hat{\theta}$ and $\frac{1}{n}\trace(\pB^{-1}\vSigma)$ are bounded. Note that, for the minimal eigenvalue of $\pB$, we have
\BEQ
\label{eq:lowerbound_example_eq2}
\BS
	\sigma_{\min}(\pB) 
	&=
	\min_{\|\vu\|=1}\vu^{\t}\left(\vX^{\t}\diag{\dfl''(\vy-\vX\hvb_{\lambda})}\vX+\lambda \diag{\cdreg''(\hvb_{\lambda})}
\right)\vu
		\\
	&\stackrel{\text{(i)}}{\geq}
	\min_{\|\vu\|=1}\vu^{\t}\left(\vX^{\t}\cdot \ac\vI\cdot \vX\right)\vu
		\stackrel{\text{(ii)}}{\geq}
	\Omega_p\left(\ac\right)
	, 
\end{split}	
\EEQ
where Inequality (i) is due to Assumption O.\ref{ass:Smoothness}, and Inequality (ii) is due to Lemma \ref{lem:minev}. Hence, the eigenvalues of $\pB^{-1}$ are upper bounded by $\op{1/\ac}$ and therefore, we have 
\BE
	\frac{1}{n}\trace(\pB^{-1}\vSigma)
	&=&
	\op{\frac{1}{\ac}}
	.	\label{eq:upperbound}
\EE
To bound $\hat{\theta}$, we plug \eqref{eq:thetadef} in the RHS of \eqref{eq:62920_eq1} and obtain
\BE
	\dotp{\frac{\hat{\theta}\fl''(\vy-\vX\hvb_{\lambda})}{1+\hat{\theta}\fl''(\vy-\vX\hvb_{\lambda})}}
	&=&
	\frac{1}{\delta}\dotp{\frac{\hat{\vsigma}^2}{\hat{\vsigma}^2+\hat{\tau}\reg''(\hvb_{\lambda})}}
	.	\label{eq:amp=glmc11}
\EE
Hence, by \eqref{eq:amp=glmc11} and Assumption O.\ref{ass:Convex}, we have
\BE
	\frac{1}{\delta}
	&\geq&
	\dotp{\frac{\hat{\theta}\fl''(\vy-\vX\hvb_{\lambda})}{1+\hat{\theta}\fl''(\vy-\vX\hvb_{\lambda})}}
	.	\label{eq:limiteqc_41} \n
\EE
This implies that
\[
	\frac{1}{1+\hat{\theta}\inf_i\fl''(y_i-\vrx{i}^{\t}\hvb_{\lambda})}
	\qgeq
	\dotp{\frac{1}{1+\hat{\theta}\fl''(\vy-\vX\hvb_{\lambda})}}
	\geq
	1-\frac{1}{\delta}
	.
\]
Therefore, with Assumption O.\ref{ass:Smoothness}, we have\footnote{If we replace Assumption O.\ref{ass:Smoothness} by $\inf_{x\in \bbR} \reg''(x) \geq \ac$, we can upper bound $\hat{\theta}$ by $\op{\frac{1}{\ac}}$ via its construction in \eqref{eq:thetadef}}
\BE
	\hat{\theta}
	&\leq&
	\op{\frac{1}{\ac}}
	.	\label{eq:upperbound2}
\EE
 Lemma \ref{lem:G} and the mean value theorem will then imply that
\[
	\left|\hat{\theta}-\frac{1}{n}\trace(\pB^{-1}\vSigma)\right| 
	\spaceequal 
	\op{\frac{\plogn\cdot \bc^{1+\alpha}}{n^{\frac{\alpha^2}{2}}\cdot \ac^{67\rho+19}}}
	.
\] 
Hence, if we define $\hat{z}_i  = y_i-\vrx{i}^{\t}\hvb_{\lambda}+ \hat{\theta}\cdot\fl'(y_i-\vrx{i}^{\t} \hvb_{\lambda})$, then we have 
\[
	\sup_i\left|\hat{z}_i-(y_i-\vrx{i}^{\t}\loovb{i}_{\lambda})\right|
	\spaceequal
	\op{\frac{\plogn\cdot \bc^{1+\alpha}}{n^{\frac{\alpha^2}{2}}\cdot \ac^{71\rho+21}}}
	.
\]
Finally, by combining the Mean Value Theorem, Lemma \ref{lem:supnormDi} and Assumption O.\ref{ass:Smoothness2}, we have
\BE
	\lefteqn{\left|\lo-\wampr\right|}
		\n\\
	&=&
	\sup_i \left|\hat{z}_i-(y_i-\vrx{i}^{\t}\loovb{i}_{\lambda})\right| \cdot \op{1+\sup_i \max(|\hat{z}_i|^{\rho+1},|y_i-\vrx{i}^{\t}\loovb{i}_{\lambda}|^{\rho+1})}
		\n\\
	&=&
	\op{\frac{\plogn \cdot \bc^{1+\alpha}}{n^{\frac{\alpha^2}{2}}\cdot \ac^{72\rho+22}}}
	. \n
\EE
This completes the proof of the theorem.
$\hfill \square$


\subsubsection{Proof sketch of Theorem \ref{thm:jixu}} \label{ssec:disclo_lao}

First we remind the reader that according to Proposition \ref{thm:looo}, we have
\[
	\loovb{i}_{\lambda}
	\spaceequal
	\hvb_{\lambda}-\fl'(y_i-\vrx{i}^{\t}\hvb_{\lambda})\oA_i^{-1}\vrx{i}+\tvepsilon^i
	,
\]
where 
\[
	\oA_i
	\spaceequal
	\vX^{\t}\diag{\dfl''(\vy-\vX\loovb{i}_{\lambda})}\vX+\lambda\diag{\dreg''(\loovb{i}_{\lambda})}-\fl''(y_i-\vrx{i}^{\t}\loovb{i}_{\lambda})\vrx{i}\vrx{i}^{\t}.
\] 
Furthermore in the same proposition we proved
\[
	\sup_i\|\tvepsilon^i\|
	\spaceequal
	\op{\frac{\plogn }{n^{\frac{\alpha}{2}}\cdot \ac^{9\rho+10}}}. 
\]
By comparing this formula with \eqref{eq:aloformula1}, which was the main formula that led to $\alo$, it is straightforward to confirm that if we obtain a bound on the difference $\vrx{i}^{\t}\pA_i^{-1}\vrx{i}-\vrx{i}^{\t}\oA_i^{-1}\vrx{i}$, with
\[
	\pA_i
	\spaceequal
	\vX^{\t}\diag{\dfl''(\vy-\vX\hvb_{\lambda})}\vX-\fl''(y_i-\vrx{i}^{\t}\hvb_{\lambda})\vrx{i}\vrx{i}^{\t}+\lambda\diag{\dreg''(\hvb_{\lambda})},
\] 
 then we can obtain a bound between $\alo$ and $\lo$. We will show that
 
\begin{lemma}\label{lem:replace3}
Under Assumptions O.\ref{ass:Convex}-O.\ref{ass:True}, for large enough $n$ (note that $n/p = \delta$ remains fixed), we have
\[
	\sup_i\left|\vrx{i}^{\t}(\oA_{i}^{-1}-\pA_{i}^{-1})\vrx{i}\right|
	\spaceequal
	\op{\frac{\plogn }{n^{\frac{\alpha^2}{2}}\cdot \ac^{8\rho+9}}}
	.
\]
\end{lemma}

Notice that the proof of Lemma \ref{lem:replace3} can be easily obtained from the proof of Lemma \ref{lem:ABswitch}. The rest of the proof follows the above lemma immediately. 

$\hfill \square$. \\


\subsubsection{Proof sketch of Theorem \ref{thm:LOconsistency}}\label{ssec:disc_lo_extra}

The main idea is to break the difference $\lo-\extra$ into the following two pieces:
\BE
	\text{part 1:}
	&&
	P_1 \spaceequal \frac{1}{n}\sum_{i=1}^n\fl(y_i-\vrx{i}^{\t}\loovb{i}_{\lambda})-\frac{1}{n}\sum_{i=1}^n\bbE[\fl(y_{new}-\vrx{new}^{\t}\loovb{i}_{\lambda})\big|\mathcal{D}]
	,	\n\\
	\text{part 2:}
	&&
	P_2 \spaceequal \frac{1}{n}\sum_{i=1}^n\bbE[\fl(y_{new}-\vrx{new}^{\t}\loovb{i}_{\lambda})\big|\mathcal{D}]-\bbE[\fl(y_{new}-\vrx{new}^{\t}\hvb_{\lambda})\big|\mathcal{D}]
	.	\n
\EE
For part 1, we note that  
\[
	\bbE[\fl(y_{new}-\vrx{new}^{\t}\loovb{i}_{\lambda})\big|\mathcal{D}]
	\spaceequal \bbE[\fl(y_{new}-\vrx{new}^{\t}\loovb{i}_{\lambda})\big|\mathcal{D}_i] = 
	\bbE[\fl(y_i-\vrx{i}^{\t}\loovb{i}_{\lambda})\big|\mathcal{D}_i]
	,
\]
where $\mathcal{D}_i=\mathcal{D}\backslash\{(\vrx{i},y_i)\}$. Hence, part 1 is equal to 
\begin{eqnarray}
	P_1 \spaceequal \frac{1}{n} \sum_{i=1}^n (\fl(y_i-\vrx{i}^{\t}\loovb{i}_{\lambda})- \bbE[\fl(y_i-\vrx{i}^{\t}\loovb{i}_{\lambda})\big|\mathcal{D}_i])
	\n
\end{eqnarray}
It is clear that the expected value of $P_1$ is equal to zero. Furthermore, we claim that since the correlations among the different terms in the summation are small enough, we can bound the variance by $O\left(\frac{1}{n\ac^{8\rho+10}}\right)$. The following lemma clarifies this claim:

\begin{lemma}\label{lem:crossterm}
Under Assumptions O.\ref{ass:Convex}-O.\ref{ass:True}, for all $i\neq j\in \{1,2,\ldots, n\}$, we have 
\BE
	\bbE \left[\left(\fl(y_i-\vrx{i}^{\t}\loovb{i}_{\lambda})-\bbE[\fl(y_i-\vrx{i}^{\t}\loovb{i}_{\lambda})\big|\mathcal{D}_i]\right)\left(\fl(y_j-\vrx{j}^{\t}\loovb{j}_{\lambda})-\bbE[\fl(y_j-\vrx{j}^{\t}\loovb{j}_{\lambda})\big|\mathcal{D}_j]\right)\right]
		\n
\EE
is at most $O\left(\frac{1}{n\ac^{8\rho+10}}\right)$. Furthermore, we have
\[
	\text{var}(P_1)
	\spaceequal
	O\left(\frac{1}{n\ac^{8\rho+10}}\right)
	.
\]
\end{lemma}
The proof of this lemma can be found in Section \ref{sec:crossterm}. Lemma \ref{lem:crossterm} combined with Markov inequality imply that part 1 is bounded by $\op{\frac{1}{\ac^{4\rho+5}\sqrt{n}}}$. 
Hence, the next step of the proof is to bound $P_2$ by $$\op{\frac{1}{\ac^{3\rho+4}\sqrt{n}}}.$$ By applying the mean value theorem we have 
\[
	|P_2|
	\qleq
	\sup_{i,\xi}\left|\bbE\left[\vx_{\new}^{\t}(\loovb{i}_{\lambda}-\hvb_{\lambda})\fl'\left(y_{\new}-\vx_{\new}^{\t}((1-\xi)\hvb_{\lambda}+\xi\loovb{i}_{\lambda})\right)\big|\mathcal{D}\right]\right|
\]
Then by using the Cauchy-Schwarz inequality and independency between the new copy $(\vx_{\new},y_{\new})$ and the data set $\mathcal{D}$, we have
\BE
	|P_2|^2
	&\leq&
	\sup_{i,\xi}\bbE\left[\left(\vx_{\new}^{\t}(\loovb{i}_{\lambda}-\hvb_{\lambda})\right)^2\big|\mathcal{D}\right]\cdot \bbE\left[\left(\fl'\left(y_{\new}-\vx_{\new}^{\t}((1-\xi)\hvb_{\lambda}+\xi\loovb{i}_{\lambda})\right)\right)^2\big|\mathcal{D}\right]
		\n\\
	&=&
	\sup_{i,\xi}\frac{\|\loovb{i}_{\lambda}-\hvb_{\lambda}\|^2}{n}\cdot \bbE\left[\left(\fl'\left(y_{\new}-\vx_{\new}^{\t}((1-\xi)\hvb_{\lambda}+\xi\loovb{i}_{\lambda})\right)\right)^2\big|\mathcal{D}\right]
		\n\\
	&\leq&
	\sup_{i}\frac{\|\loovb{i}_{\lambda}-\hvb_{\lambda}\|^2}{n}\cdot O(1)\cdot \left(1+\sup_i\frac{\|\loovb{i}_{\lambda}-\vbeta_0\|^{4\rho+4}}{n^{2\rho+2}}+\sup_i \frac{\|\hvb_{\lambda}-\loovb{i}_{\lambda}\|^{4\rho+4}}{n^{2\rho+2}}\right)	
	,	\n
\EE
where the last inequality is due to Assumption O.\ref{ass:Smoothness2}. From the proof of Lemma \ref{lem:supnormDi} in Section \ref{sec:supnormDi} (See \eqref{eq:add_tue_2.33pm} and \eqref{eq:add_tue_2.35pm}), we have
\begin{eqnarray}
	\sup_{i}\frac{\|\loovb{i}_{\lambda}-\hvb_{\lambda}\|^2}{n}
	&=&
	\op{\frac{\plogn}{n \cdot \ac^{2\rho+4}}}
	\n \\
	\sup_i\frac{\|\loovb{i}_{\lambda}-\vbeta_0\|^2}{n}
	&=&
	\op{\frac{1}{\ac^2}}
	.
\end{eqnarray}
Hence, $P_2$ is bounded by $\op{\frac{1}{\ac^{3\rho+4}\sqrt{n}}}$. 
$\hfill \square$.


\section{Discussion and Future Directions}

By developing a unified approach for studying the out-of-sample prediction error, under the high-dimensional asymptotics $n,p \rightarrow \infty$, $n/p \rightarrow \delta>1$, we obtained the first rigorous proof for the consistency of the leave-one-out cross-validation, approximate leave-one-out, and the approximate message passing risk estimate. The main challenge of the rigorous theory presented here was the high-dimensional setting of our framework. To provide practical justification for the success of these risk estimates, we have also obtained upper bounds for their convergence rates, confirming a fast convergence when both the loss function and regularizer are smooth.

Despite the progress that has been made in this paper, several important aspects of the risk estimation and model selection in high-dimensional settings have remained open. We list some of these challenges below:
\begin{enumerate}
\item Uniform convergence over an infinite number of models: The most important application of the risk estimation problem is ``model selection". The consistency result that we obtained in this paper ensures the consistency of the model that is obtained from ALO, AMP and LOOCV based methods among a finite number of models. However, one  may argue that in order to obtain the optimal choice of $\lambda$ in \eqref{eq:model} we need the uniform consistency of  our estimates over $\lambda \in [0, \infty)$. We should first emphasize that under the asymptotic setting of the paper $n/p =\delta >1$, the optimal choice of $\lambda$ converges to a fixed number \cite{dobriban2018high,mousavi2018consistent} and its dependance on $n$ and $p$ will be mild. Hence, practically speaking one would consider a finite partition of the $\lambda$ values and find the value that returns the minimum risk. If a practitioner uses this strategy for picking the optimal $\lambda$ (which we believe is often the case), then our results will imply the consistency of the selected model. However, it is still a mathematically important question whether the uniform consistency holds over $\lambda \in [0, \infty)$. This problem is left for future research. 
\item Non-differentiable regularizers and loss functions: As mentioned in Assumption \ref{ass:Smoothness2}, we only consider twice differentiable losses and regularizers. As was discussed in the paper, one can apply smoothing techniques to convert non-differentiable losses and regularizers to differentiable ones for which our consistency results hold. While smoothing techniques are usually appealing for speeding up optimization algorithms \cite{nesterov2005smooth,becker2011templates}, there are many occasions, such as in variable selection, in which a researcher may prefer to work with non-differentiable functions directly. Hence, it would be more appealing to have consistent risk estimators that work on both differentiable and non-differentiable problems. In the derivation of ALO and AMP risk estimates the twice differentiability of the loss and the regularizer are assumed. However, there has been some work in extending these formulas for non-differentiable losses and regularizers. For instance, \cite{rad2018scalable,wang2018approximate} showed how one can obtain ALO formulas for many non-differentiable losses and regularizers, and confirmed the accuracy of these formulas through extensive simulations. Similarly, \cite{mousavi2018consistent} showed how in the case of LASSO, one can obtain a consistent risk estimate through AMP. Generalizing such results and proving the consistency of the estimates is an important direction that is left for future research.
\item Imperfect models and dependent features: There are two more assumptions we have made in our proofs that can limit the applicability of our results in practical settings. The first one is that we have assumed the underlying model is linear and that the noise in the system is independent of the features. While this is considered to be a standard assumption in the theoretical analysis of linear models, it can be violated in many applications. Hence, proving consistency of risk estimates under more general settings seems to be one of the most important open questions on high-dimensional risk estimation. Another assumption that has been made in our analysis is the independence of features. Simulation results confirm that dependence does not affect the accuracy of ALO and LOOCV risk estimates. However, it can affect the accuracy of AMP-based estimates. Again the theoretical analysis of these estimates under more general assumptions is of great interest and is left for future research. 
\end{enumerate}

\section{Acknowledgement}
The work of Arian Maleki is partially supported by the grant DMS1810888 from the National Science Foundation.  The work of Kamiar Rahnamarad is supported by DMS1810888 from the National Science Foundation and the Eugene M. Lang Fellowship.

%% file: sec-proof_detail.tex
\section{Proof of Theorem \ref{THM:MSE}}\label{sec:MSE}

\subsection{Preliminaries} \label{sec:Preliminaries}

In this section, we gather the existing results (or their straightforward corollaries ) that are required in multiple proofs throughout our manuscript. The first result is concerned with the eigenvalues of several matrices which will appear in our proofs. 

\begin{lemma}\label{lem:minev}
Let $\delta = n/p>1$. Under Assumption O.\ref{ass:True}, for large enough $n$, we have all the following statements hold with probability at least $1-2e^{-n}-2e^{-\sqrt{n}}$.
\begin{itemize}
\item [(i)] $\sigma_{\rm max} (\vX^{\t}\vX)\leq 2c_u$.
\item [(ii)] $\max_{1\leq i \leq n}\sigma_{\rm max} (\XI^{\t}\XI) \leq 2c_u$.
\item [(iii)] $\max_{1\leq i \leq p}\sigma_{\rm max} (\XIb\XIb^{\t} ) = \max_{1\leq i \leq p}\sigma_{\rm max} (\XIb^{\t}\XIb ) \leq 2c_u$, where $\XIb$ is matrix $X$ without $i$th column. 
\item [(iv)] $\sigma_{\rm min} (\vX^{\t}\vX) \geq \frac{c_l}{2}(1-\frac{1}{\delta})\triangleq \sd$.
\item [(v)] $\min_{1 \leq i \leq n} \sigma_{\rm min}(\XI^{\t}\XI) \geq \frac{c_l}{2}(1-\frac{1}{\delta})$.
\item[(vi)] $\min_{1 \leq i \leq p} \sigma_{\rm min}( \XIb^{\t}\XIb) \geq \frac{c_l}{2}(1-\frac{1}{\delta})$
\end{itemize}
\end{lemma}

\begin{proof}
Note that from Assumption O.\ref{ass:True}, we have
\[
\sigma_{\rm max}(\vX^{\t}\vX) \leq c_u\cdot \sigma_{\rm max}(\vX^{\t}\vSigma^{-1}\vX) \qand \sigma_{\rm min}(\vX^{\t}\vX) \geq c_l\cdot \sigma_{\rm min}(\vX^{\t}\vSigma^{-1}\vX).
\]
From Lemma 10 of \cite{bartlett2019benign}, we have with probability $1-2e^{-p}$ that 
\[
\sigma_{\rm min}(\vX^{\t}\vSigma^{-1}\vX) \geq \frac{n-p}{n}=1-\frac{1}{\delta} \qand  
\sigma_{\rm max}(\vX^{\t}\vSigma^{-1}\vX) \leq \frac{p+n}{n} = 1+ \frac{1}{\delta}.
\]
Hence, we have (i) and (iv) hold. Then note that for all $i\in[n]$,
\begin{eqnarray}
\sigma_{\rm max} (\XIb^{\t}\XIb) &=& \max_{\|\vu\|=1,\vu\in\bbR^{p-1}}\vu^{\t} \XIb^{\t}\XIb\vu\ =\   \max_{\|\vu\|=1,\vu\in\bbR^{p}, u_p=0}\vu^{\t} \vX^{\t}\vX\vu\n\\
&\leq&\max_{\|\vu\|=1,\vu\in\bbR^{p}}\vu^{\t} \vX^{\t}\vX\vu \ = \ \sigma_{\rm max} (\vX^{\t}\vX),\n
\end{eqnarray}
and
\begin{eqnarray}
\sigma_{\rm min} (\XIb^{\t}\XIb) &=& \min_{\|\vu\|=1,\vu\in\bbR^{p-1}}\vu^{\t} \XIb^{\t}\XIb\vu\ =\   \min_{\|\vu\|=1,\vu\in\bbR^{p}, u_p=0}\vu^{\t} \vX^{\t}\vX\vu\n\\
&\geq&\min_{\|\vu\|=1,\vu\in\bbR^{p}}\vu^{\t} \vX^{\t}\vX\vu \ = \ \sigma_{\rm min} (\vX^{\t}\vX).\n
\end{eqnarray}
Hence we have (iii) and (vi) hold. For (ii), note that for all $i\in[n]$,
\[
\sigma_{\rm max} (\XI^{\t}\XI) \ = \ \sigma_{\rm max} (\vX^{\t}\vX-\vrx{i}\vrx{i}^{\t}) \ \leq \ \sigma_{\rm max} (\vX^{\t}\vX).
\]
Hence we have (ii) holds. Finally, for (v), note that for all $i\in[n]$, let $\vu_{/i}$ be the eigenvector that corresponds to the minimum eigenvalue of $\XI^{\t}\XI$. Then we have
\[
\sigma_{\rm min} (\XI^{\t}\XI) \ =\  \vu_{/i}^{\t}\XI^{\t}\XI\vu_{/i} \ = \ \vu_{/i}^{\t}\vX^{\t}\vX\vu_{/i} - (\vu_{/i}^{\t}\vrx{i})^2 \geq \sigma_{\rm min} (\vX^{\t}\vX) - (\vu_{/i}^{\t}\vrx{i})^2.
\] 
Note that $\vrx{i}$ is independent of $\vu_{/i}$ and from Assumption O.\ref{ass:True}, elements of $\vrx{i}$ are mean 0 independent random variables with subGaussian tails. Therefore, from Hanson-Wright inequality, we have for all $t>0$,
\[
\Pr\left(\sup_i (\vu_{/i}^{\t}\vrx{i})^2\geq \frac{\ln n}{n}t\right) \leq n e^{-c\min( (t\ln n)^2, t\ln n )},
\]
where $c>0$ is some constant independent of $n$. Hence we have $\min_{1\leq i\leq n}\sigma_{\rm min} (\XI^{\t}\XI) \geq \sigma_{\rm min} (\vX^{\t}\vX)-O(\frac{\ln n}{\sqrt{n}}) \geq \frac{1}{2}(1-\frac{1}{\delta})$ with probability at least $1-2e^{-\sqrt{n}}$. Therefore (v) holds.

\end{proof}

Our second lemma reviews the different concentrations for subGaussian random vectors.  

\begin{lemma}\label{lem:xwnorm}
Under Assumption O.\ref{ass:True}, for large enough $n$, we have
\BE
	\sup_i\|\vrx{i}\|^2
	\spaceequal
	\op{1}
	, 
	&&
	\sup_i|\|\vcx{i}\|^2-\sigma_i^2| \spaceequal \sup_i|\hat{\sigma}_i^2-\sigma_i^2| 
	\spaceequal 
	\op{\frac{\ln n}{\sqrt{n}}}
	,	\n\\
	\sup_i|w_i|
	\spaceequal
	\op{\ln n}
	,
	&& 
	\sup_i|\beta_{0,i}|
	\spaceequal
	\op{\ln n} 
	\qand 
	\sup_{i,j}|x_{ij}|
	\spaceequal
	\op{\frac{\ln n}{\sqrt{n}}}
	.	\n
\EE
\end{lemma}
\begin{proof}
The proof is straightforward due to Hanson-Wright inequality and subGaussian assumption on $\vw$ and $\vX(\vSigma/n)^{-\frac{1}{2}}$ due to Assumption O.\ref{ass:True}. When $\vbeta_0$ is a random vector, $\sup_i|\beta_{0,i}|=\op{\ln n}$ holds due to Hanson-Wright inequality and the subGaussian assumption. When $\vbeta_0$ is deterministic, $\sup_i|\beta_{0,i}|=\op{\ln n}$ holds directly due to Assumption O.\ref{ass:True}.  
\end{proof}

Finally, since we will apply Matrix Inversion Lemma repeatedly, we formally state it here.
\begin{lemma}\label{lem:MIL}(Matrix Inversion Lemma)
Let $\vW\in \bbR^{n_1\times n_1}, \vU\in \bbR^{n_1\times n_2}, \vT\in \bbR^{n_2\times n_2}, \vV\in \bbR^{n_2\times n_1}$. If $\vW^{-1}$ and $\vT^{-1}$ exists, we have
\[
	(\vW+\vU\vT\vV)^{-1}
	\spaceequal 
	\vW^{-1}-\vW^{-1}\vU(\vT^{-1}+\vV\vW^{-1}\vU)^{-1}\vV\vW^{-1}
	.
\]
\end{lemma}



\subsection{Proof of Lemma \ref{lem:AMP=GLM1}}\label{sec:uniquesolution}

We aim to show that given $(\vX,\vy,\vbeta)$, the mapping between $\gamma$ and $\tau$ defined in  \eqref{eq:amp=glm_eq1} is one-to-one. Note that by multiplying $\frac{1}{\delta\gamma}\dotp{\frac{\hat{\vsigma}^2}{\hat{\vsigma}^2+\tau\reg''(\vbeta)}}$ by both sides of \eqref{eq:amp=glm_eq1}, we have: 
\BE
	\frac{1}{\delta}\dotp{\frac{\hat{\vsigma}^2}{\hat{\vsigma}^2+\tau\reg''(\vbeta)}}
	\spaceequal
	\dotp{\frac{\frac{1}{\delta\lambda}\dotp{\frac{\hat{\vsigma}^2}{\hat{\vsigma}^2+\tau\reg''(\vbeta)}}\fl''(\vy-\vX\vbeta)}{\frac{1}{\tau}+\frac{1}{\delta\gamma}\dotp{\frac{\hat{\vsigma}^2}{\hat{\vsigma}^2+\tau\reg''(\vbeta)}}\cdot \fl''(\vy-\vX\vbeta)}}
	,	\n
\EE
which is equivalent to
\BE
	1-\dotp{\frac{1}{ \frac{1}{\delta\gamma}\dotp{\frac{\tau\hat{\vsigma}^2}{\hat{\vsigma}^2+\tau\reg''(\vbeta)}}\fl''(\vy-\vX\vbeta)+1}}
	&=&
	\frac{1}{\delta}\dotp{\frac{\hat{\vsigma}^2}{\hat{\vsigma}^2+\tau\reg''(\vbeta)}}
	.	\label{eq:taulambdacondition}
\EE
Hence, we just need to show that the mapping defined in \eqref{eq:taulambdacondition} is a bijection. First, for all fixed $\tau>0$, the righthand side of \eqref{eq:taulambdacondition} will be a constant between $(0,1)$. Let the lefthand side be a function of $\gamma$, i.e.,
\[
	g_{\tau}(\gamma)
	\ \triangleq \ 
	1-\dotp{\frac{1}{\frac{1}{\delta\gamma}\dotp{\frac{\tau\hat{\vsigma}^2}{\hat{\vsigma}^2+\tau\reg''(\vbeta)}}\fl''(\vy-\vX\vbeta)+1}} 
	\spaceequal
	\dotp{\frac{\dotp{\frac{\tau\hat{\vsigma}^2}{\hat{\vsigma}^2+\tau\reg''(\vbeta)}}\fl''(\vy-\vX\vbeta)}{\dotp{\frac{\tau\hat{\vsigma}^2}{\hat{\vsigma}^2+\tau\reg''(\vbeta)}}\fl''(\vy-\vX\vbeta)+\delta\gamma}}
	.
\]
Then, we have 
\[
	\frac{\dif g_{\tau}(\gamma)}{\dif \gamma}
	\spaceequal
	-\delta\dotp{\frac{\dotp{\frac{\tau\hat{\vsigma}^2}{\hat{\vsigma}^2+\tau\reg''(\vbeta)}}\fl''(\vy-\vX\vbeta)}{\left(\dotp{\frac{\tau\hat{\vsigma}^2}{\hat{\vsigma}^2+\tau\reg''(\vbeta)}}\fl''(\vy-\vX\vbeta)+\delta\gamma \right)^2}}
	\  < \ 
	0
	.
\]
Further, $g_{\tau}(0)=1$ and $\lim_{\gamma\rightarrow \infty}g_{\tau}(\gamma)=0$. Hence, we know for all fixed $\tau>0$, there exists unique $\gamma>0$ satisfying \eqref{eq:taulambdacondition}. On the other hand, for all fixed $\gamma>0$, we can rewrite \eqref{eq:taulambdacondition} in the following way:
\BE
	1
	&=&
	\dotp{\frac{1}{\frac{1}{\delta\gamma}\dotp{\frac{\tau\hat{\vsigma}^2}{\hat{\vsigma}^2+\tau\reg''(\vbeta)}}\fl''(\vy-\vX\vbeta)+1}}+\frac{1}{\delta}\dotp{\frac{\hat{\vsigma}^2}{\hat{\vsigma}^2+\tau\reg''(\vbeta)}}
	.	\label{eq:limiteqc_2}
\EE
Let $\tilde{g}_{\gamma}(\tau)$ be the righthand side of \eqref{eq:limiteqc_2}, i.e.,
\[
	\tilde{g}_{\gamma}(\tau)
	\ \triangleq\ 
	\frac{1}{\delta}\dotp{\frac{\hat{\vsigma}^2}{\hat{\vsigma}^2+\tau\reg''(\vbeta)}}+\dotp{\frac{1}{\frac{1}{\delta\gamma}\dotp{\frac{\tau\hat{\vsigma}^2}{\hat{\vsigma}^2+\tau\reg''(\vbeta)}} \fl''(\vy-\vX\vbeta)+1}}
	.
\] 
Then we have
\BE
	\frac{\dif \tilde{g}_{\gamma}(\tau)}{\dif \tau}
	&=&
	-\frac{1}{\delta}\dotp{\frac{\reg''(\vbeta)\hat{\vsigma}^2}{(\hat{\vsigma}^2+\tau\reg''(\vbeta))^2}}
		\n\\
	& &
	-\dotp{\frac{\fl''(\vy-\vX\vbeta)}{\left(\frac{1}{\delta\gamma}\dotp{\frac{\tau\hat{\vsigma}^2}{\hat{\vsigma}^2+\tau\reg''(\vbeta)}} \fl''(\vy-\vX\vbeta)+1\right)^2}}\cdot \frac{1}{\delta\gamma}\dotp{\frac{\hat{\vsigma}^4}{(\hat{\vsigma}^2+\tau\reg''(\vbeta))^2}}
	\ <\ 
	0
	,	\n
\EE
Further, $\tilde{g}_{\gamma}(0)>1$ and $\lim_{\tau\rightarrow \infty}\tilde{g}_{\gamma}(\tau)<1$. Hence, for all fixed $\gamma>0$, there exists a unique $\tau>0$ satisfying \eqref{eq:taulambdacondition}. 
$\hfill \square$


\subsection{Proof of Lemma \ref{lem:supnormDi}}\label{sec:supnormDi}

We will first show that
\[
	\sup_i|y_i-\vrx{i}^{\t}\loovb{i}|
	\spaceequal
	\op{\frac{\ln n}{\ac}}
	.
\]
Note that $y_i=\vrx{i}^{\t}\vbeta_0+w_i$, we have
\[
	\sup_i|y_i-\vrx{i}^{\t}\loovb{i}|
	\qleq
	\sup_i|\vrx{i}^{\t}(\vbeta_0-\loovb{i})|+\sup_i|w_i|
	\spaceequal
	\sup_i|\vrx{i}^{\t}(\vbeta_0-\loovb{i})|+\op{\ln n}
	,
\]
where last equation is due to Lemma \ref{lem:xwnorm}. Hence, we just need to bound $\sup_i|\vrx{i}^{\t}(\vbeta_0-\loovb{i})|$. Note that
\BE
	\lefteqn{\Pr\left(\sup_i|\vrx{i}^{\t}(\vbeta_0-\loovb{i})|\geq C^*\ln n\right)}
		\n\\
	&\leq&
	\Pr\left(\sup_i|\vrx{i}^{\t}(\vbeta_0-\loovb{i})|\geq C^*\ln n, \sup_i\frac{\|\vbeta_0-\loovb{i}\|}{\sqrt{n}}<C^*\right)+\Pr\left(\sup_i\frac{\|\vbeta_0-\loovb{i}\|}{\sqrt{n}}\geq C^*\right)
		\n\\
	&=&
	\Pr\left(\sup_i|\vrx{i}^{\t}(\vbeta_0-\loovb{i})|\geq C^*\ln n, \frac{\|\vbeta_0-\loovb{i}\|}{\sqrt{n}}<C^*, \forall i\right)+\Pr\left(\sup_i\frac{\|\vbeta_0-\loovb{i}\|}{\sqrt{n}}\geq C^*\right)
		\n\\
	&\leq&
	\sum_{i=1}^n \Pr\left(|\vrx{i}^{\t}(\vbeta_0-\loovb{i})|\geq C^*\ln n, \frac{\|\vbeta_0-\loovb{i}\|}{\sqrt{n}}<C^*, \forall i\right)+\Pr\left(\sup_i\frac{\|\vbeta_0-\loovb{i}\|}{\sqrt{n}}\geq C^*\right)
		\n\\
	&\leq&
	\sum_{i=1}^n\Pr\left(|\vrx{i}^{\t}(\vbeta_0-\loovb{i})|\geq C^*\ln n,\frac{\|\vbeta_0-\loovb{i}\|}{\sqrt{n}}<C^*\right)+\Pr\left(\sup_i\frac{\|\vbeta_0-\loovb{i}\|}{\sqrt{n}}\geq C^*\right)
		\n\\
	&\leq&
	\sum_{i=1}^n\Pr\left(|\vrx{i}^{\t}(\vbeta_0-\loovb{i})|\geq C^*\ln n \  \Big| \ \frac{\|\vbeta_0-\loovb{i}\|}{\sqrt{n}}<C^*\right)+\Pr\left(\sup_i\frac{\|\vbeta_0-\loovb{i}\|}{\sqrt{n}}\geq C^*\right)
	,	\n
\EE
where $C^*>0$ will be determined later. Since $\vrx{i}$ is independent of $\vbeta_0-\loovb{i}$, from Assumption O.\ref{ass:True} and Hanson-Wright inequality, we have
\BE
	\Pr\left(\sup_i|\vrx{i}^{\t}(\vbeta_0-\loovb{i})|\geq C^*\ln n\right)
	&\leq&
	o(1)+\Pr\left(\sup_i\frac{\|\vbeta_0-\loovb{i}\|}{\sqrt{n}}\geq C^*\right)
	.	\n
\EE
Therefore,  to obtain an upped bound for $\Pr\left(\sup_i|\vrx{i}^{\t}(\vbeta_0-\loovb{i})|\geq C^*\ln n\right)$ we only need to bound $\sup_i\frac{\|\vbeta_0-\loovb{i}\|}{\sqrt{n}}$. From the first order conditions we have
\BE
	\vzero
	&=&
	-\vX^{\t}\fl'(\vy-\vX\loovb{i})+\fl'(y_i-\vrx{i}^{\t}\loovb{i})\vrx{i}+\lambda\reg'(\loovb{i})
	.	\label{eq:ithout}
\EE
Plugging $\vy = \vX \vbeta_0 + \vw$ we have 
\[
	-\vX^{\t}\fl'(\vw+\vX(\vbeta_0-\loovb{i}))+\fl'(w_i+\vrx{i}^{\t}(\vbeta_0-\loovb{i}))\vrx{i}+\lambda\reg'(\loovb{i})
	\spaceequal
	\vzero
	.
\]
Using the mean value theorem for $\fl'$ at $\vw$ and $\reg'$ at $\vbeta_0$, we have
\BE
	\lefteqn{\vX^{\t}\fl'(\vw)-\fl'(w_i)\vrx{i}-\lambda\reg'(\vbeta_0)}
		\n\\
	&=&
	\left(\vX^{\t}\diag{\dfl''(\vw_{\vxi})}\vX-\fl''(w_{\vxi,i})\vrx{i}\vrx{i}^{\t}+\lambda\diag{\dreg''(\vbeta^i_{\vxi'})}\right)(\loovb{i}-\vbeta_0)
	,	\n
\EE
where $\ve_j^{\t}\vbeta^i_{\vxi'}=\xi'_{i,j} \ve_j^{\t} \vbeta_0+(1-\xi'_{i,j})\ve_j^{\t} \loovb{i}$ for some $\xi'_{i,j}\in [0,1]$ and $j$th diagonal component of $\diag{\dfl''(\vw_{\vxi})}$ is $\fl''(\xi_{i,j}w_j+(1-\xi_{i,j}) \vrx{j}^{\t}(\vbeta_0-\loovb{i}))$ for some $\xi_{i,j}\in [0,1]$. Hence, we have
\BEQ
\label{eq:add_eq2}
\BS
	\lefteqn{\sup_i\frac{1}{\sqrt{n}}\|\loovb{i}-\vbeta_0\|}
		\\
	&=
	\frac{1}{\sqrt{n}}\sup_i\left\|\left(\vX^{\t}\diag{\dfl''(\vw_{\vxi})}\vX-\fl''(w_{\vxi,i})\vrx{i}\vrx{i}^{\t}+\lambda\diag{\dreg''(\vbeta^i_{\vxi'})}\right)^{-1}\right.
		\\
	&\quad \times\left.(\vX^{\t}\fl'(\vw)-\fl'(w_i)\vrx{i}-\lambda\reg'(\vbeta_0))\right\|
		\\
	&\stackrel{\text{(i)}}{\leq}
	\op{\frac{1}{\ac \sqrt{n}}}\sup_i\|\vX^{\t}\fl'(\vw)-\fl'(w_i)\vrx{i}-\lambda\reg'(\vbeta_0)\|
		\\
	&\leq
	\op{\frac{1}{\ac\sqrt{n}}}(\|\vX^{\t}\fl'(\vw)\|+\sup_i|\fl'(w_i)|\|\vrx{i}\|+\lambda\|\reg'(\vbeta_0)\|)
		\\
	&\stackrel{\text{(ii)}}{\leq}
	\op{\frac{1}{\ac\sqrt{n}}}(\|\vX^{\t}\fl'(\vw)\|+\lambda\|\reg'(\vbeta_0)\|)+\op{\frac{\plogn}{\ac\sqrt{n}}}
	,	
\end{split}
\EEQ
 where Inequality (i) is due to Lemma \ref{lem:minev} and Assumption O.\ref{ass:Smoothness}, and Inequality (ii) holds due to Lemma \ref{lem:xwnorm} and Assumption O.\ref{ass:Smoothness2}. To bound \eqref{eq:add_eq2}, note that since $\vX$ is independent of $\vw$, from Assumption O.\ref{ass:True} and Hanson-Wright inequality we have
\[
	\|\vX^{\t}\fl'(\vw)\|
	\spaceequal
	\op{\frac{\|\fl'(\vw)\|}{\sqrt{n}}\sqrt{p}}
	\spaceequal
	\op{\sqrt{n}}
	,
\]
where the last equality is due to Assumption O.\ref{ass:Smoothness2} and Assumption O.\ref{ass:True}. For $\|\reg'(\vbeta_0)\|$, from Assumption O.\ref{ass:True}, we have $\frac{1}{p}\sum_{i}\beta_{0,i}^{2(\rho+1)}=\op{1}$ for both the case of $\vbeta_0$ being a random vector or the case of $\vbeta_0$ being deterministic. Hence, with Assumption O.\ref{ass:Smoothness2}, we have for some universal constant $C>0$,
\BE
	\frac{1}{n}\|\reg'(\vbeta_0)\|^2
	\qleq
	\frac{C}{n}\sum_{i}(1+\beta_{0,i}^{2(\rho+1)})
	\spaceequal
	\op{1}
	.\label{eq:review_eq1}
\EE
Hence, for \eqref{eq:add_eq2}, we have
\BE
	\sup_i\frac{1}{\sqrt{n}}\|\loovb{i}-\vbeta_0\|
	&\leq&
	\op{\frac{1}{\ac}}
	,	\label{eq:add_tue_2.33pm}
\EE
i.e., for all $\epsilon>0$, there exists constant $C_\epsilon,N_{\epsilon}>0$ such that
\[
	\Pr\left(\sup_i\frac{1}{\sqrt{n}}\|\loovb{i}-\vbeta_0\|\qgeq \frac{C_{\epsilon}}{\ac}\right)
	\qleq 
	\epsilon
	, \ \ \forall n>N_{\epsilon}
	. 
\]
Therefore, by choosing $C^*=C_{\epsilon}/\ac$, we have
\BE
	\sup_i|y_i-\vrx{i}^{\t}\loovb{i}|
	&=&
	\op{\frac{\ln n}{\ac}}
	.	\label{eq:suploor}
\EE
Now we switch to the proof of 
\[
	\sup_{i}\|\hvb-\loovb{i}\|
	\spaceequal 
	\op{\frac{\plogn}{\ac^{\rho+2}}}
	. 
\]
Let $\vDelta^i=\loovb{i}-\hvb$. $\hvb,\loovb{i}$ satisfy the following equations:
\BE
	\vzero
	&=&
	-\vX^{\t}\fl'(\vy-\vX\hvb)+\lambda\reg'(\hvb)
	,	\label{eq:full}\\
	\vzero
	&=&
	-\vX^{\t}\fl'(\vy-\vX\loovb{i})+\fl'(y_i-\vrx{i}^{\t}\loovb{i})\vrx{i}+\lambda\reg'(\loovb{i})
	.	\n
\EE
By subtracting the above two terms and applying the mean value theorem for $\fl'$ and $\reg'$, we have
\BEQ
\label{eq:add_eq3}
\BS
	-\fl'(y_i-\vrx{i}^{\t}\loovb{i})\vrx{i}
	&=
	\lambda\left(\reg'(\hvb)-\reg'(\loovb{i})\right)
	-\sum_{j=1}^n\vrx{j}\left(\fl'(y_j-\vrx{j}^{\t}\hvb)-\fl'(y_j-\vrx{j}^{\t}\loovb{i})\right)
		\\
	&=
	\lambda\diag{\dreg''(\vbeta^{i}_{\vxi'})}\vDelta^i+\vX^{\t}\diag{\dfl''(\vy-\vX\vbeta^i_{\vxi})}\vX\vDelta^i
	,
\end{split}	
\EEQ
where $\ve_j^{\t}\vbeta^i_{\vxi'}=\xi'_{i,j} \ve_j^{\t} \vbeta_0+(1-\xi'_{i,j})\ve_j^{\t} \loovb{i}$ for some $\xi'_{i,j}\in [0,1]$ and the $j$th diagonal component of $\diag{\dfl''(\vy-\vX\vbeta^i_{\vxi})}$ is $\fl''(y_j-\vrx{j}^{\t}(\xi_{i,j}\hvb+(1-\xi_{i,j})\loovb{i}))$ for some $\xi_{i,j}\in [0,1]$. From Lemma \ref{lem:minev} and Assumption O.\ref{ass:Smoothness}, we know that matrix
\[
	\left(\vX^{\t}\diag{\dfl''(\vy-\vX\vbeta^i_{\xi})}\vX+\lambda\diag{\dreg''(\vbeta^{i}_{\vxi'})}\right)^{-1}
\] 
exists and its maximum eigenvalue is at most $1/(\ac\sd)=\op{\frac{1}{\ac}}$, where $\sigma_\delta$ is defined in Lemma \ref{lem:minev}. Hence, if we define 
\[
	\vB^i_{\vxi,\vxi'}
	\ := \
	\vX^{\t}\diag{\dfl''(\vy-\vX\vbeta^i_{\vxi})}\vX+\lambda\diag{\dreg''(\vbeta^{i}_{\vxi'})}
	,
\]
then from \eqref{eq:add_eq3} we have
\[
	\vDelta^i
	\spaceequal
	-\fl'(y_i-\vrx{i}\loovb{i})(\vB^i_{\vxi,\vxi'})^{-1}\vrx{i}
	.
\]
Hence, we have
\BEQ
\label{eq:add_tue_2.35pm}
\BS
	\sup_i\|\loovb{i}-\hvb\|
	&=
	\sup_{i,\vxi,\vxi'}\left|\fl'(y_i-\vrx{i}\loovb{i})\right|\sqrt{\vrx{i}^{\t}\left(\vB^i_{\vxi,\vxi'}\right)^{-2}\vrx{i}}
		\\
	&\leq
	\op{\frac{1}{\ac}}\sup_i|\fl'(y_i-\vrx{i}\loovb{i})|\sup_i\|\vrx{i}\|
		\\
	&\leq
	\op{\frac{\plogn}{\ac^{\rho+2}}}
	,
\end{split}	
\EEQ
where the last inequality is due to Lemma \ref{lem:xwnorm}, Assumption O.\ref{ass:Smoothness2} and \eqref{eq:suploor}.
Finally, apply \eqref{eq:suploor} and Lemma \ref{lem:xwnorm} again, we immediately have
\BE
	\sup_i|y_i-\vrx{i}^{\t}\hvb|
	&\leq&
	\sup_i|y_i-\vrx{i}^{\t}\loovb{i}|+\sup_i|\vrx{i}^{\t}(\loovb{i}-\hvb)|
		\n\\
	&\leq&
	\op{\frac{\ln n }{\ac}}+\sup_{i}\|\vrx{i}\|\sup_i\|\loovb{i}-\hvb\|
		\n\\
	&=&
	\op{\frac{\plogn}{\ac^{\rho+2}}}
	.	\n
\EE


\subsection{Proof of Lemma \ref{lem:TAIL}}\label{sec:tail}
For all $t>0$, 
\BEQ
\label{eq:tail_bound_add_eq_eq1}
\BS
	\lefteqn{\Pr\left(\sup_i\left|\vx_i^{\t}\vGamma_i\vx_i\right|\geq t\right)}
		\\
	&=
	\Pr\left(\sup_i\left|\vx_i^{\t}\vGamma_i\vx_i-\frac{1}{n}\trace(\vGamma_i)\right|\geq t,\ \sup_i\sigma_{\max}(\vGamma_i)\leq C_{n,\delta}\right)
		\\
	&\quad
	+\Pr\left(\sup_i\left|\vx_i^{\t}\vGamma_i\vx_i-\frac{1}{n}\trace(\vGamma_i)\right|\geq t,\ \sup_i\sigma_{\max}(\vGamma_i)\geq C_{n,\delta}\right)
		\\
	&\leq
		\sum_{i=1}^n\Pr\left(\left|\vx_i^{\t}\vGamma_i\vx_i-\frac{1}{n}\trace(\vGamma_i)\right|\geq t,\ \sup_i\sigma_{\max}(\vGamma_i)\leq C_{n,\delta}\right)+\Pr\left(\sup_i\sigma_{\max}(\vGamma_i)\geq C_{n,\delta}\right)
		\\
	&\leq
	\sum_{i=1}^n\Pr\left(\left|\vx_i^{\t}\vGamma_i\vx_i-\frac{1}{n}\trace(\vGamma_i)\right|\geq t,\ \sigma_{\max}(\vGamma_i)\leq C_{n,\delta}\right)+\Pr\left(\sup_i\sigma_{\max}(\vGamma_i)\geq C_{n,\delta}\right)
		\\
	&\leq
	\sum_{i=1}^n\Pr\left(\left|\vx_i^{\t}\vGamma_i\vx_i-\frac{1}{n}\trace(\vGamma_i)\right|\geq t \Big| \sigma_{\max}(\vGamma_i)\leq C_{n,\delta}\right)+\Pr\left(\sup_i\sigma_{\max}(\vGamma_i)\geq C_{n,\delta}\right)
	.
\end{split}	
\EEQ
Hence, we just need to bound $\left|\vx_i^{\t}\vGamma_i\vx_i-\frac{1}{n}\trace(\vGamma_i)\right|$ individually given that $\sigma_{\max}(\vGamma_i)\leq C_{n,\delta}$. From Hanson-Wright inequality, we have for all $t'>0$
\[
\Pr\left(|\vx_i^{\t}\vGamma_i\vx_i-\frac{1}{n}\trace(\vGamma_i)|\geq C_{n,\delta} \frac{\ln n}{\sqrt{n}}t'\right) \leq 2e^{-c (t'\ln n)^2},
\]
where constant $c>0$ is independent of $n$. Hence taking a union bound over all $i\in [n]$, we have Lemma \ref{lem:TAIL} holds.


\subsection{Proof of Proposition \ref{thm:looo}}\label{sec:looo}

It is straightforward to see that  $\loovb{i}$ satisfies 
\BE
	\vzero
	&=&
	-\vX^{\t}\fl'(\vy-\vX\loovb{i})+\fl'(y_i-\vrx{i}^{\t}\loovb{i})\vrx{i}+\lambda\reg'(\loovb{i})
	.	\label{eq:ithout}
\EE
Recall the definition of $\oA_i$, i.e.,
\[
	\oA_i
	\spaceequal
	\vX^{\t}\diag{\dfl''(\vy-\vX\loovb{i})}\vX+\lambda\diag{\dreg''(\loovb{i})}-\fl''(y_i-\vrx{i}^{\t}\loovb{i})\vrx{i}\vrx{i}^{\t}
	.
\]
Note that according to \eqref{eq:lowerbound_example_eq1}, the inverse of $\oA_i$ exists and thus all values are well defined in the theorem.  For a given $i$, let $\tr_i$ be the minimizer of the following optimization:
\[
	\tr_i := \argmin_{r} \frac{1}{2}(r-y_i+\vrx{i}^{\t}\loovb{i})^2+\vrx{i}^{\t}\oA_i^{-1}\vrx{i} \fl(r).
\]
By Assumption O.\ref{ass:Convex}, we know this is a convex optimization and hence $\tr_i$ is unique and satisfies the following equation:
\BE
	\tr_i 
	&=&
	y_i-\vrx{i}^{\t}\loovb{i}-\fl'(\tr_i)\vrx{i}^{\t}\oA_i^{-1}\vrx{i}
	.	\label{eq:tri}
\EE
Now, let
\BE
	\lootvb{i}
	&=&
	\loovb{i}+\fl'(\tr_i)\oA_i^{-1}\vrx{i}
	.	\label{eq:tovb}
\EE
Then, by plugging \eqref{eq:tovb} in \eqref{eq:tri} we have
\BE
	\tr_i
	&=&
	y_i-\vrx{i}^{\t}\lootvb{i}
	.	\label{eq:tri2}
\EE
The important feature of $	\lootvb{i}$ is that if we plug \eqref{eq:tri2} in \eqref{eq:tovb}, then we will obtain
\[
	\loovb{i}
	\spaceequal
	\lootvb{i}-\fl'(y_i-\vrx{i}^{\t}\lootvb{i})\oA_i^{-1}\vrx{i}
	.
\]
By Taylor expansion for $\fl'$ at $y_i-\vrx{i}^{\t}\hvb$, we have 
\[
	\loovb{i}
	\spaceequal
	\hvb-\fl'(y_i-\vrx{i}^{\t}\hvb)\oA_i^{-1}\vrx{i}+\tvepsilon^i
	,
\]
where
\[
	\tvepsilon^i
	\spaceequal
	(\vI+\fl''(y_i-\vrx{i}^{\t}\vbeta^i_{\xi_i})\oA_i^{-1}\vrx{i}\vrx{i}^{\t})(\lootvb{i}-\hvb)
	,
\]
with $\vbeta^i_{\xi_i}=\xi_i\lootvb{i}+(1-\xi_i)\hvb$ for some $\xi_i\in [0,1]$. So far, we have obtained an expression for $\tvepsilon^i$. Next, we want to bound $\sup_i\|\tvepsilon^i\|$. Note that
\BE
	\sup_i\|\tvepsilon^i\|
	&\leq&
	\sup_i\|\lootvb{i}-\hvb\|\left(1+\fl''(y_i-\vrx{i}^{\t}\vbeta^i_{\xi_i})\max_{\|\vu\|=1}\sqrt{\vu^{\t}\vrx{i}\vrx{i}^{\t}\oA_i^{-1}\oA_i^{-1}\vrx{i}\vrx{i}^{\t}\vu}\right)
	.	\n
\EE
Hence,
\BEQ
\label{eq:tvepsilon_eq1_p1}
\BS
	\frac{\sup_i\|\tvepsilon^i\|}{\sup_i\|\lootvb{i}-\hvb\|}
	&\stackrel{\text{(i)}}{\leq}
	 \left(1+\op{1+\max\{|\tr_i|^{\rho},|y_i-\vrx{i}^{\t}\hvb|^{\rho}\}}\cdot \sup_i\|\vrx{i}\| \cdot \sup_i\|\oA_i^{-1}\vrx{i}\|\right)
		\\
	&\stackrel{\text{(ii)}}{\leq}
	 \left(1+\op{1+\max\{|\tr_i|^{\rho},|y_i-\vrx{i}^{\t}\hvb|^{\rho}\}}\cdot \sup_i\|\vrx{i}\|^2\cdot \op{\frac{1}{\ac}}\right)
		\\
	&\stackrel{\text{(iii)}}{\leq}
	 \left(1+\op{1+\max\{|\tr_i|^{\rho},|y_i-\vrx{i}^{\t}\hvb|^{\rho}\}}\cdot \op{\frac{1}{\ac}}\right)
	,
\end{split}
\EEQ
where Inequality (i) is due to Assumption O.\ref{ass:Smoothness2}, Inequality (ii) holds due to \eqref{eq:lowerbound_example_eq1} and Inequality (iii) is due to Lemma \ref{lem:xwnorm}. To bound $|\tr_i|$, let us define $U(r) = r+\fl'(r)\vrx{i}^{\t}\oA_i^{-1}\vrx{i}$, then we have $U'(r)>1$ and $U(\tr_i) = y_i-\vrx{i}^{\t}\loovb{i}$ due to \eqref{eq:tri}. Hence, we have
\begin{eqnarray}
	|y_i-\vrx{i}^{\t}\loovb{i}|
	&=& 
	|U(\tr_i)|
	\spaceequal
	\left|U(0)+\int_0^{\tr_i}U'(r)\dif r\right|
		\n\\
	&\geq&
	\left|\int_0^{\tr_i}U'(r)\dif r\right|-|U(0)|
	\qgeq
	|\tr_i|-|\fl'(0)|\vrx{i}^{\t}\oA_i^{-1}\vrx{i}
	.	\n
\end{eqnarray}
Due to \eqref{eq:lowerbound_example_eq1} and Lemma \ref{lem:xwnorm}, we have $\vrx{i}^{\t}\oA_i^{-1}\vrx{i}= \op{1/\ac}$ and thus
\BE
	\sup_i|\tr_i|
	\ \in\ 
	\left[0,\sup_i|y_i-\vrx{i}^{\t}\loovb{i}|+\op{\frac{1}{\ac}}\right]
	.	\label{eq:suptri}
\EE
Therefore, by \eqref{eq:tvepsilon_eq1_p1}, \eqref{eq:suptri} and Lemma \ref{lem:supnormDi}, we have 
\BE
	\sup_i\|\tvepsilon^i\|
	&\leq&
	\sup_i\|\lootvb{i}-\hvb\|\cdot \op{\frac{\plogn}{\ac^{\rho+3}}}
	.	\n
\EE
Hence, to bound $\sup_i\|\tvepsilon^i\|$, we should show that 
\BE
	\sup_i\|\lootvb{i}-\hvb\|
	&\leq&
	\op{\frac{\plogn }{n^{\frac{\alpha}{2}}\cdot \ac^{8\rho+7}}}
	.	\label{eq:suptovbhvb}
\EE
Let 
\begin{eqnarray}
	L(\vbeta)
	\spaceequal
	-\vX^{\t}\fl'(\vy-\vX\vbeta)+\lambda\reg'(\vbeta)
	.	\label{eq:defL_eq1}
\end{eqnarray}
From \eqref{eq:full}, we know that 
\[
	L(\hvb)
	\spaceequal
	\vzero
	.
\]
Note that the Jacobian matrix of $L(\vbeta)$ is 
\[
	\vX^{\t}\diag{\dfl''(\vy-\vX\vbeta)}\vX+\lambda\diag{\dreg''(\vbeta)},
\] 
and by Lemma \ref{lem:minev} and Assumption O.\ref{ass:Smoothness}, its minimum eigenvalue is at least $\Omega_p(\ac)$. Hence, we have
\BEQ
\label{eq:L111}
\BS
	\sup_i\|\hvb-\lootvb{i}\|
	&\leq
	\op{\frac{1}{\ac}}\cdot \sup_i\|L(\lootvb{i})-L(\hvb)\|
		\\
	&=
	\op{\frac{1}{\ac}} \cdot \sup_i\|L(\lootvb{i})\|
	.	\\
\end{split}
\EEQ
Therefore, we just need to show that 
\[
	\sup_i\|L(\lootvb{i})\|
	\qleq
	\op{\frac{\plogn }{n^{\frac{\alpha}{2}}\cdot \ac^{8\rho+6}}}
	.
\]
Note that, we have
\BEQ
\label{eq:revision_add_eq1}
\BS
	\lefteqn{L(\lootvb{i})\spaceequal-\vX^{\t}\fl'(\vy-\vX\lootvb{i})+\lambda\reg'(\lootvb{i})}
		\\
	&\stackrel{\text{(i)}}{=}
	-\fl'(\tr_i)\vrx{i}-\sum_{j\neq i}(\fl'(y_j-\vrx{j}^{\t}\lootvb{i})-\fl'(y_j-\vrx{j}^{\t}\loovb{i}))\vrx{j}+\lambda (\reg'(\lootvb{i})-\reg'(\loovb{i}))
		\\
	&\stackrel{\text{(ii)}}{=}
	-\fl'(\tr_i)\vrx{i}-\sum_{j\neq i}\fl''(y_j-\vrx{j}^{\t}\vbeta^{i}_{\xi_{i,j}})\vrx{j}\vrx{j}^{\t}(\lootvb{i}-\loovb{i})+\lambda\diag{\dreg''(\vbeta^i_{\vxi'})}(\lootvb{i}-\loovb{i})
		\\
	&=
	-\fl'(\tr_i)\vrx{i}+\left(\sum_{j\neq i}\fl''(y_j-\vrx{j}^{\t}\loovb{i})\vrx{j}\vrx{j}^{\t}+\lambda \diag{\dreg''(\loovb{i})}\right)(\lootvb{i}-\loovb{i})
		\\
	&\quad
	-\sum_{j\neq i}\left(\fl''(y_j-\vrx{j}^{\t}\vbeta^{i}_{\xi_{i,j}})-\fl''(y_j-\vrx{j}^{\t}\loovb{i})\right)\vrx{j}\vrx{j}^{\t}(\lootvb{i}-\loovb{i})
		\\
	&\quad
	+\lambda\diag{\dreg''(\vbeta^i_{\vxi'})-\dreg''(\loovb{i})}(\lootvb{i}-\loovb{i})
		\\
	&\stackrel{\text{(iii)}}{=}
	-\sum_{j\neq i}\left(\fl''(y_j-\vrx{j}^{\t}\vbeta^{i}_{\xi_{i,j}})-\fl''(y_j-\vrx{j}^{\t}\loovb{i})\right)\vrx{j}\vrx{j}^{\t}(\lootvb{i}-\loovb{i})
		\\
	&\quad
	+\lambda\diag{\dreg''(\vbeta^i_{\vxi'})-\dreg''(\loovb{i})}(\lootvb{i}-\loovb{i})
	,	
\end{split}	
\EEQ
where $\vbeta^i_{\xi_{i,j}}=\xi_{i,j}\lootvb{i}+(1-\xi_{i,j})\loovb{i}, \ve_j^{\t}\vbeta^i_{\vxi'}=\xi'_{i,j}\lootb{i}_j+(1-\xi'_{i.j})\loob{i}_j$ for some $\xi_{i,j}, \xi'_{i,j}\in [0,1]$, and Equality (i) is due to \eqref{eq:ithout} and \eqref{eq:tri2}, Equality (ii) is obtained from a Taylor expansion and Equality (iii) holds due to \eqref{eq:tovb}. Next, let $\vu^{\backslash i}\in \bbR^{n-1}$ be the vector defined by the following:
\BE
	\ve_j^{\t}\vu^{\backslash i}=\left\{\begin{aligned}
	&\left(\fl''(y_j-\vrx{j}^{\t}\vbeta^{i}_{\xi_{i,j}})-\fl''(y_j-\vrx{j}^{\t}\loovb{i})\vrx{j}^{\t}\right)(\lootvb{i}-\loovb{i}),	&&j\in [1,i-1]\\
	&\left(\fl''(y_{j+1}-\vrx{j+1}^{\t}\vbeta^{i}_{\xi_{i,j+1}})-\fl''(y_{j+1}-\vrx{j+1}^{\t}\loovb{i})\vrx{j+1}^{\t}\right)(\lootvb{i}-\loovb{i}),	&&j\in[i,n-1]
	\end{aligned}\right.
	\n.
\EE
Then, by Assumption O.\ref{ass:Holder}, we known each component of $\vu^{\backslash i}$ is upper bounded by the following:
\BE
	\sup_j|\ve_j^{\t}\vu^{\backslash i}|
	&\leq&
	\sup_{i,j}\left|\left(\fl''(y_j-\vrx{j}^{\t}\vbeta^{i}_{\xi_{i,j}})-\fl''(y_j-\vrx{j}^{\t}\loovb{i})\right)\vrx{j}^{\t}(\lootvb{i}-\loovb{i})\right|
		\n\\
	&\leq&
	C_r\left|\vrx{j}^{\t}(\lootvb{i}-\loovb{i})\right|^{1+\alpha}
	.	\n
\EE
Also, by Assumption O.\ref{ass:Holder}, we have
\BE
	\sup_{i}\left|\reg''(\vbeta^i_{\vxi'})-\reg''(\loovb{i})\right|
	&\leq&
	C_r\left|\ve_j^{\t}(\lootvb{i}-\loovb{i})\right|^{\alpha}
	.	\n
\EE
Hence, with \eqref{eq:revision_add_eq1}, we have
\BEQ
\label{eq:revision_add_eq3}
\BS
	\lefteqn{\sup_i\|L(\lootvb{i})\|}
		\\
	&\leq
	\sup_i \|\XI\vu^{\backslash i}\| + \lambda \sup_i\left|\reg''(\vbeta^i_{\vxi'})-\reg''(\loovb{i})\right| \cdot \sup_i\|\lootvb{i}-\loovb{i}\|
		\\
	&\stackrel{\text{(i)}}{\leq}
	\op{1}\cdot \sup_i \|\vu^{\backslash i}\|+\lambda\sqrt{p}\sup_{i,j}\left|C_r\left(\ve_j^{\t}(\lootvb{i}-\loovb{i})\right)^{1+\alpha}\right|
		\\
	&\leq
	\op{1}\cdot \sqrt{n}\sup_{j\neq i}\left|C_l\left(\vrx{j}^{\t}(\lootvb{i}-\loovb{i})\right)^{1+\alpha}\right|+	\lambda\sqrt{p}\sup_{i,j}\left|C_r\left(\ve_j^{\t}(\lootvb{i}-\loovb{i})\right)^{1+\alpha}\right|
	,	
\end{split}
\EEQ
where inequality (i) is due Lemma \ref{lem:minev}. Next, we claim that
\BEQ
\label{eq:suptovbij}
\BS
	\sup_{i,j}|\ve_j ^{\t}(\lootvb{i}-\loovb{i})|
	&=
	\op{\frac{\plogn}{\sqrt{n}\cdot \ac^{4\rho+3}}} 
		\\ 
	\sup_{j\neq i}|\vrx{j}^{\t} (\lootvb{i}-\loovb{i})|
	&=
	\op{\frac{\plogn}{\sqrt{n}\cdot \ac^{4\rho+3}}}
	.
\end{split}
\EEQ
If these two claims are true, from \eqref{eq:revision_add_eq3} we have
\BE
	\sup_i\|L(\lootvb{i})\|
	&\leq&
	\op{\frac{\plogn }{n^{\frac{\alpha}{2}}\cdot \ac^{8\rho+6}}}
	,	\n
\EE
which completes the proof of \eqref{eq:suptovbhvb}. To show \eqref{eq:suptovbij}, by \eqref{eq:tovb} and Assumption O.\ref{ass:Smoothness2}, we have
\BE
	\sup_{i,j}|\ve_j ^{\t}(\lootvb{i}-\loovb{i})|
	&=&
	\sup_{i,j} |\fl'(\tr_i)\ve_j^{\t}\oA_i^{-1}\vrx{i}|
		\n\\
	&\leq&
	\op{1+\sup_{i}|\tr_i|^{\rho+1}} \cdot \sup_{i,j}|\ve_j^{\t}\oA_i^{-1}\vrx{i}|
	,	\n
\EE
and
\BE
	\sup_{j\neq i}|\vrx{j}^{\t} (\lootvb{i}-\loovb{i})|
	&=&
	\sup_{j\neq i} |\fl'(\tr_i)\vrx{j}^{\t}\oA_i^{-1}\vrx{i}|
		\n\\
	&\leq& 
	\op{1+\sup_{i} |\tr_i|^{\rho+1}} \cdot \sup_{j\neq i}|\vrx{j}^{\t}\oA_i^{-1}\vrx{i}|
	.	\n
\EE
Then, by \eqref{eq:suptri} and Lemma \ref{lem:supnormDi}, we have
\BEQ
\label{eq:suptovbij12}
\BS
	\sup_{i,j}|\ve_j ^{\t}(\lootvb{i}-\loovb{i})|
	&\leq
	\op{\frac{\plogn}{\ac^{(\rho+2)(\rho+1)}}}\cdot \sup_{i,j}|\ve_j^{\t}\oA_i^{-1}\vrx{i}|
		,\\
	\sup_{i\neq j}|\vrx{j} ^{\t}(\lootvb{i}-\loovb{i})|
	&\leq
	\op{\frac{\plogn}{\ac^{(\rho+2)(\rho+1)}}}\cdot \sup_{i\neq j}|\vrx{j}^{\t}\oA_i^{-1}\vrx{i}|
	.	
\end{split}
\EEQ
Recall that the minimal eigenvalue of $\oA_i$ is at least $\Omega_p(\ac)$ due to \eqref{eq:lowerbound_example_eq1}. Since $\vrx{j}$ (for $j \neq i$) and $\oA_i$ is independent of $\vrx{i}$, from Assumption O.\ref{ass:True} and Hanson-Wright inequality, we have for some constant $c>0$ independent of $n$,
\BE
	\lefteqn{\Pr\left(\sup_{i,j}|\ve_j^{\t}\oA_i^{-1}\vrx{i}|> \epsilon\right)}
		\n\\
	&\leq&
	\sum_{i,j}\Pr\left(|\ve_j^{\t}\oA_i^{-1}\vrx{i}|> \epsilon\Big| \sup_{i,j}\|\ve_j^{\t}\oA_i^{-1}\|\leq \frac{1}{\ac\sd}\right)+\Pr\left(\sup_{i,j}\|\ve_j^{\t}\oA_i^{-1}\|> \frac{1}{\ac\sd}\right)
		\n\\
	&\leq&n^2e^{-c\min(\epsilon^2\ac^2 n, \epsilon \ac n)} + o(1)
	,	\n
\EE
and together with Lemma \ref{lem:xwnorm}, we have
\BE
	\lefteqn{\Pr\left(\sup_{i\neq j}|\vrx{j}^{\t}\oA_i^{-1}\vrx{i}|> \epsilon\right)}
		\n\\
	&\leq&
	\sum_{i,j}\Pr\left(|\vrx{j}^{\t}\oA_i^{-1}\vrx{i}|> \epsilon\Big| \sup_{i,j}\|\vrx{j}^{\t}\oA_i^{-1}\|\leq \frac{2}{\ac\sd}\right)+\Pr\left(\sup_{i, j}\|\vrx{j}^{\t}\oA_i^{-1}\|> \frac{2}{\ac\sd}\right)
		\n\\
	&\leq&n^2e^{-c\min(\epsilon^2\ac^2 n, \epsilon \ac n)} + o(1)
	.	\n
\EE
Hence, we have
\[
	\sup_{i,j}|\ve_j^{\t}\oA_i^{-1}\vrx{i}|
	\spaceequal
	\op{\frac{1}{\ac\sqrt{n}}} 
	\qand 
	\sup_{j\neq i}|\vrx{j}^{\t}\oA_i^{-1}\vrx{i}|
	\spaceequal
	\op{\frac{1}{\ac\sqrt{n}}}
	.
\]
Therefore, if we plug these two equations in \eqref{eq:suptovbij12}, then obtain \eqref{eq:suptovbij}. As a result of \eqref{eq:suptovbhvb} and \eqref{eq:suptovbij}, we have for large enough $n$,
\[
	\sup_{i,j}\left|\ve_j^{\t}(\hvb-\loovb{i})\right|
	\spaceequal
	\op{\frac{\plogn }{n^{\frac{\alpha}{2}}\cdot \ac^{8\rho+7}}}
	.
\]


\subsection{Proof of Lemma \ref{lem:ABswitch}}\label{sec:ABswitch}

Recall 
\[
	\pA_{i}
	\ :=\
	\vX^{\t}\diag{\dfl''(\vy-\vX\hvb)}\vX-\fl''(y_i-\vrx{i}^{\t}\hvb)\vrx{i}\vrx{i}^{\t}+\lambda \diag{\cdreg''(\hvb)}
	.
\]
We first show that 
\BE
	\sup_i\frac{1}{n}|\trace(\pA_{i}^{-1}\vSigma-\pB^{-1}\vSigma)|
	&=&
	\op{\frac{1}{n\cdot \ac}}
	.	\label{eq:ABswitch1step1}
\EE
To make the equations more readable in the rest of the proof, we use the simplified notation $\hfl''_i=\fl''(y_i-\vrx{i}^{\t}\hvb)$.  By Lemma \ref{lem:MIL} and the fact that $\pA_i$ is semi-positive definite, we have 
\BE
	\frac{1}{n}\sup_i\left|\trace(\pA_{i}^{-1}\vSigma-\pB^{-1}\vSigma)\right|
	&=&
	\frac{1}{n}\sup_i\left|\trace(\vSigma^{\frac{1}{2}}(\pA_{i}^{-1}-(\pA_{i}+\hfl''_i\vrx{i}\vrx{i}^{\t})^{-1})\vSigma^{\frac{1}{2}})\right|
		\n\\
	&=&
	\frac{1}{n}\sup_i\left|\frac{\hfl''_i}{1+\hfl''_i\vrx{i}^{\t}\pA_{i}^{-1}\vrx{i}}\trace(\vSigma^{\frac{1}{2}}\pA_{i}^{-1}\vrx{i}\vrx{i}^{\t}\pA_{i}^{-1}\vSigma^{\frac{1}{2}})\right|
		\n\\
	&\leq&
	\frac{c_u}{n}\sup_i\left|\frac{\hfl''_i}{1+\hfl''_i\vrx{i}^{\t}\pA_{i}^{-1}\vrx{i}}\vrx{i}^{\t}\pA_{i}^{-2}\vrx{i}\right|
		\n\\
	&\leq&
	\frac{c_u}{n}\sup_i\frac{\hfl''_i\vrx{i}^{\t}\pA_{i}^{-1/2}\pA_i^{-1}\pA_{i}^{-1/2}\vrx{i}}{\hfl''_i\vrx{i}^{\t}\pA_{i}^{-1}\vrx{i}}
	\qleq \frac{c_u}{n\inf_i\sigma_{\min}(\pA_{i}) }.	\label{eq:revision_add_eq5}
\EE 

Further, for the minimal eigenvalue of $\pA_{i}$, we have
\BEQ
\label{eq:lowerbound_example_eqpAc}
\BS
	&\inf_i\sigma_{\min}(\pA_{i})\\
	&=
	\inf_i\min_{\|\vu\|=1}\vu^{\t}\left(\vX^{\t}\diag{\dfl''(\vy-\vX\hvb)}\vX-\fl''(y_i-\vrx{i}^{\t}\hvb)\vrx{i}\vrx{i}^{\t}+\lambda \diag{\cdreg''(\hvb)}\right)\vu
		\\
	&\stackrel{\text{(i)}}{\geq}
	\inf_i\min_{\|\vu\|=1}\vu^{\t}\left(\XI^{\t}\cdot \ac\vI\cdot \XI\right)\vu
		\\
	&\stackrel{\text{(ii)}}{\geq}
	\ac\sd \spaceequal \Omega_p\left(\ac\right)
	, 	\\
\end{split}
\EEQ
where Inequality (i) is due to Assumption O.\ref{ass:Smoothness}, and Inequality (ii) is due to Lemma \ref{lem:minev}. Hence, with \eqref{eq:revision_add_eq5}, we have
\[
	\frac{1}{n}\sup_i\left|\trace(\pA_{i}^{-1}\vSigma-\pB^{-1}\vSigma)\right|
	\qleq
	\op{\frac{1}{n\cdot \ac}}
	.
\]
This completes the proof of \eqref{eq:ABswitch1step1}. Now we want to bound $\sup_i\frac{1}{n}|\trace(\pA_{i}^{-1}\vSigma-\oA_{i}^{-1}\vSigma)|$. Let $\vDelta^i_{\fl''}$ and $\vDelta^i_{\reg''}$ denote two diagonal matrices where $\vDelta^i_{\reg''}=\diag{\dreg''(\hvb)-\dreg''(\loovb{i})}$ and $j$th diagonal component of $\vDelta^i_{\fl''}$ is defined by the following:
\begin{eqnarray}
	\ve_j^{\t}\vDelta^i_{\fl''}\ve_j
	\ := \
	\left\{\begin{aligned}
	&\fl''(y_j-\vrx{j}^{\t}\hvb)-\fl''(y_j-\vrx{j}^{\t}\loovb{i}), &&j\neq i\\
	&0, &&j=i
	\end{aligned}\right.
	.	\n
\end{eqnarray}
Hence, we have
\[
	\pA_{i}-\oA_i
	\spaceequal
	\vX^{\t}\vDelta^i_{\fl''}\vX+\lambda \vDelta^i_{\reg''}
	.
\]
Let $\vDelta^i_{|\fl''|}$ and $\vDelta^i_{|\reg''|}$ denote the diagonal matrices that include the absolute values of the element of $\vDelta^i_{\fl''}$ and $\vDelta^i_{\reg''}$ respectively. Define
\[
	\vI_{\fl''}
	\spaceequal
	\diag{\tilde{\sign}(\vDelta^i_{\fl''})}
	\qand
	\vI_{\reg''}
	\spaceequal
	\diag{\tilde{\sign}(\vDelta^i_{\reg''})}
	,
\]
where $\tilde{\sign}$ function is defined as follow:
\BE
	\tilde{\sign}(x)
	&=& \left\{\begin{aligned}
				&1,	&& x\geq 0,
					\\
				&-1,	&& x<0.
			\end{aligned}\right.
		\n
\EE 
Hence, we have
\[
	\vDelta^i_{\fl''}
	\spaceequal
	\left(\vDelta^i_{|\fl''|}\right)^{\frac{1}{2}}\vI_{\fl''}\left(\vDelta^i_{|\fl''|}\right)^{\frac{1}{2}} 
	\qand
	\vDelta^i_{\reg''}
	\spaceequal
	\left(\vDelta^i_{|\reg''|}\right)^{\frac{1}{2}}\vI_{\reg''}\left(\vDelta^i_{|\reg''|}\right)^{\frac{1}{2}}
	.
\]
Then by \eqref{eq:suptovbhvb}, \eqref{eq:suptovbij}, Proposition \ref{thm:looo}, Lemma \ref{lem:xwnorm} and Assumption O.\ref{ass:Holder}, we know that
\BEQ\label{eq:AB1dif1}
\BS
	\sup_{i,j}|\ve_j^{\t}\vDelta^i_{|\fl''|}\ve_j|
	&=
	\op{\frac{\plogn }{n^{\frac{\alpha^2}{2}}\cdot \ac^{8\rho+7}}}
		\\
	\sup_{i,j}|\ve_j^{\t}\vDelta^i_{|\reg''|}\ve_j|
	&=
	\op{\frac{\plogn }{n^{\frac{\alpha^2}{2}}\cdot \ac^{8\rho+7}}}
	.	
\end{split}
\EEQ
Then by Matrix Inversion Lemma, we have
\BE
	\lefteqn{\sup_{i}\frac{1}{n}\left|\trace\left(\oA_i^{-1}\vSigma-\left(\oA_i+\vX^{\t}(\vDelta^i_{|\fl''|})^{\frac{1}{2}}\vI_{\fl''}(\vDelta^i_{|\fl''|})^{\frac{1}{2}}\vX\right)^{-1}\vSigma\right)\right|}
		\n\\
	&=&
	\sup_{i,j}\left|\trace\left(\vSigma^{\frac{1}{2}}\oA_i^{-1}\vX^{\t}(\vDelta^i_{|\fl''|})^{\frac{1}{2}}\left((\vI^i_{\fl''})^{-1}+(\vDelta^i_{|\fl''|})^{\frac{1}{2}}\vX\oA_i^{-1}\vX^{\t}(\vDelta^i_{|\fl''|})^{\frac{1}{2}}\right)^{-1}(\vDelta^i_{|\fl''|})^{\frac{1}{2}}\vX\oA_i^{-1}\vSigma^{\frac{1}{2}}\right)\right|
		\n\\
	&\leq&
	c_u\cdot\sup_{i,j}\left|\trace\left(\oA_i^{-1}\vX^{\t}(\vDelta^i_{|\fl''|})^{\frac{1}{2}}\left((\vI^i_{\fl''})^{-1}+(\vDelta^i_{|\fl''|})^{\frac{1}{2}}\vX\oA_i^{-1}\vX^{\t}(\vDelta^i_{|\fl''|})^{\frac{1}{2}}\right)^{-1}(\vDelta^i_{|\fl''|})^{\frac{1}{2}}\vX\oA_i^{-1}\right)\right|\n\\
	&\leq&
	c_u\cdot\sup_{i,j}\left|\ve_j^{\t}\oA_i^{-1}\vX^{\t}(\vDelta^i_{|\fl''|})^{\frac{1}{2}}\left((\vI^i_{\fl''})^{-1}+(\vDelta^i_{|\fl''|})^{\frac{1}{2}}\vX\oA_i^{-1}\vX^{\t}(\vDelta^i_{|\fl''|})^{\frac{1}{2}}\right)^{-1}(\vDelta^i_{|\fl''|})^{\frac{1}{2}}\vX\oA_i^{-1}\ve_j\right|
		\n \\
	&\leq&
	c_u\cdot\sup_{i,j}\|(\vDelta^i_{|\fl''|})^{\frac{1}{2}}\vX\oA_i^{-1}\ve_j\|^2\frac{1}{\inf_i\min\left|\text{eigenvalue of $(\vI^i_{\fl''})^{-1}+(\vDelta^i_{|\fl''|})^{\frac{1}{2}}\vX\oA_i^{-1}\vX^{\t}(\vDelta^i_{|\fl''|})^{\frac{1}{2}}$}\right|}
		\n\\
	&\leq&
	\frac{c_u}{1+o_p(1)}\cdot \sup_i\|(\vDelta^i_{|\fl''|})^{\frac{1}{2}}\vX\oA_i^{-1}\ve_j\|^2
		\n\\
	&\leq&
	\op{\frac{\plogn }{n^{\frac{\alpha^2}{2}}\cdot \ac^{8\rho+9}}}
	,	\n \\
	\label{eq:replaceei1}
\EE
where the last two inequalities hold due to \eqref{eq:lowerbound_example_eqpAc} and \eqref{eq:AB1dif1}. Similarly, by Matrix Inversion Lemma, we have
\BE
	\lefteqn{\sup_{i}\frac{1}{n}|\trace(\pA_{i}^{-1}\vSigma-(\pA_{i}-\lambda \vDelta^i_{\reg''})^{-1}\vSigma)|}
		\n\\
	&=&
	\sup_{i}\left|\trace\left(\vSigma^{\frac{1}{2}}\pA_{i}^{-1}(\vDelta^i_{|\reg''|})^{\frac{1}{2}}\left(\frac{1}{\lambda} \vI_{\reg''}^{-1}+(\vDelta^i_{|\reg''|})^{\frac{1}{2}}\pA_{i}^{-1}(\vDelta^i_{|\reg''|})^{\frac{1}{2}}\right)^{-1}(\vDelta^i_{|\reg''|})^{\frac{1}{2}}\pA_{i}^{-1}\vSigma^{\frac{1}{2}}\right)\right|
		\n\\
	&\leq&c_u\cdot
	\sup_{i}\left|\trace\left(\pA_{i}^{-1}(\vDelta^i_{|\reg''|})^{\frac{1}{2}}\left(\frac{1}{\lambda} \vI_{\reg''}^{-1}+(\vDelta^i_{|\reg''|})^{\frac{1}{2}}\pA_{i}^{-1}(\vDelta^i_{|\reg''|})^{\frac{1}{2}}\right)^{-1}(\vDelta^i_{|\reg''|})^{\frac{1}{2}}\pA_{i}^{-1}\right)\right|\n\\
	&\leq& c_u\cdot
	\sup_{i,j}\left|\ve_j^{\t}\pA_{i}^{-1}(\vDelta^i_{|\reg''|})^{\frac{1}{2}}\left(\frac{1}{\lambda} \vI_{\reg''}^{-1}+(\vDelta^i_{|\reg''|})^{\frac{1}{2}}\pA_{i}^{-1}(\vDelta^i_{|\reg''|})^{\frac{1}{2}}\right)^{-1}(\vDelta^i_{|\reg''|})^{\frac{1}{2}}\pA_{i}^{-1}\ve_j\right|
		\label{eq:replaceei2}\\			
	&\leq&
	c_u\cdot \sup_{i,j}\|(\vDelta^i_{|\reg''|})^{\frac{1}{2}}\pA_{i}^{-1}\ve_j\|^2\cdot \frac{1}{\inf_{i}\min\left|\text{eigenvalue of $\frac{1}{\lambda} \vI_{\reg''}^{-1}+(\vDelta^i_{|\reg''|})^{\frac{1}{2}}\pA_{i}^{-1}(\vDelta^i_{|\reg''|})^{\frac{1}{2}}$}\right|}
		\n\\
	&\leq&
	\frac{c_u}{1+o_p(1)}\cdot \sup_{i,j}\|(\vDelta^i_{|\reg''|})^{\frac{1}{2}}\pA_{i}^{-1}\ve_j\|^2
		\n\\
	&\leq&
	\op{\frac{\plogn }{n^{\frac{\alpha^2}{2}}\cdot \ac^{8\rho+9}}} \n
	,
\EE
where the last two inequalities hold due to \eqref{eq:lowerbound_example_eqpAc} and \eqref{eq:AB1dif1}.



\subsection{Proof for Lemma \ref{lem:G}}\label{sec:G}

First note that, by using Assumptions O.\ref{ass:Smoothness}, O.\ref{ass:Smoothness2} and Lemma \ref{lem:supnormDi}, it is straightforward to conclude that
\BEQ
\label{eq:add_G_eq1}
	\Omega_p\left(\ac\right)
	\spaceequal 
	\inf_i \fl''(y_i-\vrx{i}^{\t}\hvb) 
	\qleq
	\sup_i\fl''(y_i-\vrx{i}^{\t}\hvb)
	\spaceequal
	\op{\frac{\plogn}{\ac^{3\rho}}}
	.
\EEQ
Calculating $G'(\theta)$ directly, we have
\BE
	\left|G'(\theta)\right|
	&\geq&
	\dotp{\frac{\hfl''}{(1+\theta \hfl'')^2}}
	\qgeq 
	\frac{1}{4}\min\left\{\inf_i \hfl''_i,\frac{1}{\theta^2\sup_{i}\hfl''_i}\right\}
		\n\\
	&=& 
	\Omega_p\left(\frac{\ac^{3\rho+1}}{(1+\theta^2)\plogn}\right)
	,	\label{eq:g'}
\EE
where $\hfl'',\hfl''_i$ are the shorthands for $\fl''(\vy-\vX\hvb),\fl''(y_i-\vrx{i}^{\t}\hvb)$ respectively, and the last inequality is due to \eqref{eq:add_G_eq1}. We remind the reader that as discussed in Section \ref{sec:assume}, this is the only place that Assumption O.\ref{ass:Smoothness} can not be replaced with $\inf_{x\in \bbR} \lambda\reg''(x) \qgeq \ac$. However, suppose that we make the following assumption
\BI

\item[O.6] a constant fraction $\gamma$ of the residuals $\{y_i-\vrx{i}^{\t}\hvb\}$ fall in the regions at which the curvature of $\ell$ is lower bounded by $\ac$.
\EI
Then, from \eqref{eq:g'} we can lower bound $|G'(\theta)|$ by
\[
	\left|G'(\theta)\right|
	\qgeq
	\frac{\dotp{\hfl''}}{(1+\theta \sup_i\hfl''_i)^2}
	\qgeq
	\gamma \frac{\ac^{6\rho+1}}{(1+\theta^2)\plogn}
	.
\]  
Hence, our results will hold even when we replace Assumption O.\ref{ass:Smoothness} by $\inf_{x\in \bbR} \lambda\reg''(x) \qgeq \ac$ and Assumption O.6.

Back to the proof of Lemma \ref{lem:G}, the next step is to prove that 
\BE
	\left|G(\hat{\theta})-G\left(\frac{1}{n}\trace(\pB^{-1}\vSigma)\right)\right|
	\spaceequal
	\op{\frac{\plogn\cdot \bc^{1+\alpha}}{n^{\frac{\alpha^2}{2}}\cdot \ac^{64\rho+16}}},
	\label{eq:add_goal_21prime}
\EE
First note that according to \eqref{eq:thetahateq} we have
\BE
	G(\hat{\theta})
	&=&
	1
	.	 \label{eq:onehand}
\EE 
To calculate $G(\frac{1}{n}\trace(\pB^{-1}\vSigma))$, let $\hX=\diag{\dfl''(\vy-\vX\hvb)}^{\frac{1}{2}}\vX$. Recall
\[
	\pA_{i}
	\ :=\
	\vX^{\t}\diag{\dfl''(\vy-\vX\hvb)}\vX-\fl''(y_i-\vrx{i}^{\t}\hvb)\vrx{i}\vrx{i}^{\t}+\lambda \diag{\cdreg''(\hvb)}
	.
\]
Using \eqref{eq:add_G_eq1} and  the matrix inversion lemma, we have
\BE
	\frac{1}{n}\trace(\hX\pB^{-1}\hX^{\t})
	&=&
	\frac{1}{n}\sum_{i=1}^n\hfl''_i\vrx{i}^{\t}\pB^{-1}\vrx{i}
		\n\\
	&=&
	\frac{1}{n}\sum_{i=1}^n\hfl''_i\left(\vrx{i}^{\t}\pA_{i}^{-1}\vrx{i}-\frac{\hfl''_i(\vrx{i}^{\t}\pA_{i}^{-1}\vrx{i})^2}{1+\hfl''_i\vrx{i}^{\t}\pA_{i}^{-1}\vrx{i}}\right)
		\n\\
	&=&
	1-\dotp{\frac{1}{1+\vrx{i}^{\t}\pA_{i}^{-1}\vrx{i}\cdot \hfl''}}
		\n\\
	&\stackrel{\text{(i)}}{=} &
	1-\dotp{\frac{1}{1+\vrx{i}^{\t}\oA_{i}^{-1}\vrx{i}\cdot \hfl''}}+\op{\frac{\plogn }{n^{\frac{\alpha^2}{2}}\cdot \ac^{11\rho+9}}}
		\n\\
	&\stackrel{\text{(ii)}}{=}&
	1-\dotp{\frac{1}{1+\frac{1}{n}\trace(\oA_{i}^{-1}\vSigma)\cdot \hfl''}}+\op{\frac{\plogn }{n^{\frac{\alpha^2}{2}}\cdot \ac^{11\rho+9}}}
		\n\\
	&\stackrel{\text{(iii)}}{=}&
	1-\dotp{\frac{1}{1+\frac{1}{n}\trace(\pB^{-1}\vSigma)\cdot \hfl''}}+\op{\frac{\plogn }{n^{\frac{\alpha^2}{2}}\cdot \ac^{11\rho+9}}}
	,	\label{eq:xBx}
\EE
where Equality (ii) is due to Lemma \ref{lem:minev} and the independency between  $\vrx{i}$ and $\oA_{i}$, Equality (iii) is due to Lemma \ref{lem:ABswitch}, and finally Equality (i) holds because of the following lemma:
\begin{lemma}\label{lem:ABswitch2}
Suppose Assumption O.\ref{ass:Convex}-O.\ref{ass:True} hold. For large enough $n$, we have
\[
	\sup_i\left|\vrx{i}^{\t}\pA_{i}^{-1}\vrx{i}-\vrx{i}^{\t}\oA_{i}^{-1}\vrx{i}\right|
	\spaceequal
	\op{\frac{\plogn }{n^{\frac{\alpha^2}{2}}\cdot \ac^{8\rho+9}}}
	.
\]
\end{lemma}
\begin{proof}
By replacing $\ve_j$ with $\vrx{i}$ in \eqref{eq:replaceei1} and \eqref{eq:replaceei2} and following similar steps as the ones presented in the proof of Lemma \ref{lem:ABswitch} for bounding $\sup_i\frac{1}{n}|\trace(\pA_{i}^{-1}\vSigma-\oA_{i}^{-1}\vSigma)|$, we can show that
\[
	\sup_i|\vrx{i}^{\t}(\pA_{i}^{-1}-\oA_{i}^{-1})\vrx{i}|
	\qleq
	\op{\frac{\plogn }{n^{\frac{\alpha^2}{2}}\cdot \ac^{8\rho+9}}}
	.
\]
\end{proof}

On the other hand, we can calculate $\frac{1}{n}\trace(\hX\pB^{-1}\hX^{\t})$ in a different way. We define
\[
	\vD\spaceequal\left(\lambda \diag{\cdreg''(\hvb)}\right)^{-1}.
\]
 Then we have
\BE
	\frac{1}{n}\trace(\hX\pB^{-1}\hX^{\t})
	&=&
	\frac{1}{n}\trace((I+\hX \vD\hX^{\t})^{-1}\hX \vD\hX^{\t})
		\n\\
	&=&
	\frac{1}{n}\sum_{i=1}^pD_i\hvcx{i}^{\t}(I+\hX \vD\hX^{\t})^{-1}\hvcx{i}
		\n\\
	&=&
	\frac{1}{n}\sum_{i=1}^pD_i\hvcx{i}^{\t}(\hQ_i+D_i\hvcx{i}\hvcx{i}^{\t})^{-1}\hvcx{i}
		\n\\
	&\stackrel{\text{(i)}}{=}&
	\frac{1}{n}\sum_{i=1}^pD_i\hvcx{i}^{\t}\left(\hQ_i^{-1}-\frac{D_i\hQ_i^{-1}\hvcx{i}\hvcx{i}^{\t}\hQ_i^{-1}}{1+D_i\hvcx{i}^{\t}\hQ_i^{-1}\hvcx{i}}\right)\hvcx{i}
		\n\\
	&=&
	\frac{p}{n}\cdot \frac{1}{p} \sum_{i=1}^p\frac{D_i\hvcx{i}^{\t}\hQ_i^{-1}\hvcx{i}}{1+D_i\hvcx{i}^{\t}\hQ_i^{-1}\hvcx{i}}
		\n\\
	&\stackrel{\text{(ii)}}{=}&
	\frac{1}{\delta}\dotp{\frac{1}{1+ \frac{\lambda \creg''(\hvb)}{\hvcx{i}^{\t}\hQ_i^{-1}\hvcx{i}}}}
	,	\label{eq:xBx2}
\EE
where $\hQ_i=\vI+\hX \vD\hX^{\t}-D_i\hvcx{i}\hvcx{i}^{\t}$, and $\hvcx{i}=\diag{\dfl''(\vy-\vX\hvb)}^{\frac{1}{2}}\vcx{i}$. Further, Equality (i) is due to Matrix Inversion Lemma and Equality (ii) is due to $D_i=\frac{1}{\lambda \creg''(\hb_i)}$ by definition. We claim that
\begin{lemma}\label{lem:trxbx}
Suppose Assumption O.\ref{ass:Convex}-O.\ref{ass:True} hold. For large enough $n$, we have
\[
	\sup_i\left|\hvcx{i}^{\t}\hQ_i^{-1}\hvcx{i}-\sigma_i^2\dotp{\frac{\hfl''}{1+\frac{1}{n}\trace(\pB^{-1}\vSigma)\cdot\hfl''}}\right|
	\spaceequal
	\op{\frac{\plogn\cdot \bc^{1+\alpha}}{n^{\frac{\alpha^2}{2}}\cdot \ac^{64\rho+15}}}
	.
\]
\end{lemma}
We will prove this lemma in the next section. In the rest of this section, we show how this lemma implies 
\[
	G\left(\frac{1}{n}\trace(\pB^{-1}\vSigma)\right)\spaceequal1+\op{\frac{\plogn\cdot \bc^{1+\alpha}}{n^{\frac{\alpha^2}{2}}\cdot \ac^{64\rho+16}}},
\]
and therefore with \eqref{eq:onehand}, it completes the proof of Lemma \ref{lem:G}. Note that if Lemma \ref{lem:trxbx} holds, then according to \eqref{eq:xBx2}, we have 
\BE
	\frac{1}{n}\trace(\hX\pB^{-1}\hX^{\t})
	&=&
	\frac{1}{\delta}\dotp{\frac{1}{1+\lambda\dotp{\frac{\hfl''}{1+\frac{1}{n}\trace(\pB^{-1}\vSigma)\hfl''}}^{-1}\frac{\creg''(\hvb)}{\sigma_i^2}}}\nonumber \\
	&&+\op{\frac{\plogn\cdot \bc^{1+\alpha}}{n^{\frac{\alpha^2}{2}}\cdot \ac^{64\rho+15}}\cdot \dotp{\frac{\hfl''}{1+\frac{1}{n}\trace(\pB^{-1}\vSigma)\cdot \hfl''}}^{-1}}
		\n\\
	&\stackrel{\text{(i)}}{=}&
	\frac{1}{\delta}\dotp{\frac{1}{1+\lambda\dotp{\frac{\hfl''}{1+\frac{1}{n}\trace(\pB^{-1}\vSigma)\hfl''}}^{-1}\frac{\creg''(\hvb)}{\sigma_i^2}}}+\op{\frac{\plogn\cdot \bc^{1+\alpha}}{n^{\frac{\alpha^2}{2}}\cdot \ac^{64\rho+16}}}
		\n\\
	&\stackrel{\text{(ii)}}{=}&
	\frac{1}{\delta}\dotp{\frac{1}{1+\lambda\dotp{\frac{\hfl''}{1+\frac{1}{n}\trace(\pB^{-1}\vSigma)\hfl''}}^{-1}\frac{\reg''(\hvb)}{\hat{\vsigma}^2}}}+\op{\frac{\plogn\cdot \bc^{1+\alpha}}{n^{\frac{\alpha^2}{2}}\cdot \ac^{64\rho+16}}}
	,	\label{eq:xBxsecondeq}
\EE
where Equality (i) holds due to \eqref{eq:upperbound} and \eqref{eq:add_G_eq1}, and Equality (iii) holds due to Lemma \ref{lem:xwnorm} and Assumption O.\ref{ass:True}. Note that we have obtained two different expressions for $\frac{1}{n}\trace(\hX\pB^{-1}\hX^{\t})$ in \eqref{eq:xBx} and \eqref{eq:xBxsecondeq}. By combining the two we obtain
\[
	G(\frac{1}{n}\trace(\pB^{-1}\vSigma))
	\spaceequal
	1+\op{\frac{\plogn\cdot \bc^{1+\alpha}}{n^{\frac{\alpha^2}{2}}\cdot \ac^{64\rho+16}}}
	.
\]
This completes the proof. Hence, the only claim that we have not proved yet is Lemma \ref{lem:trxbx}. This lemma will be proved in the next section. 

\subsection{Proof of Lemma \ref{lem:trxbx}}

Since the proof of this lemma is long, we first mention the roadmap of the proof in Section \ref{ssec:roadmaplemmatxbx} and then present the details in the subsequent sections. 
\subsubsection{Roadmap of the proof of Lemma \ref{lem:trxbx}}\label{ssec:roadmaplemmatxbx}

 Note that the goal of this lemma is to connect $\hvcx{i}^{\t}\hQ_i^{-1}\hvcx{i}$ with $\sigma_i^2\dotp{\frac{\hfl''}{1+\frac{1}{n}\trace(\pB^{-1}\vSigma)\cdot\hfl''}}$. In other words, we expect that $\hvcx{i}^{\t}\hQ_i^{-1}\hvcx{i}$ concentrates around $\sigma_i^2\dotp{\frac{\hfl''}{1+\frac{1}{n}\trace(\pB^{-1}\vSigma)\cdot\hfl''}}$ for all different values of $i$. One of the main challenges in proving this concentration is that since in the calculation of $\hQ_i^{-1}$, $\hvb$ is used, $\hQ_i^{-1}$ is dependent on $\hvcx{i}$. Hence, as the first step in our calculations we find a copy of $\hQ_i^{-1}$ from which $\hvcx{i}$ is removed. This requires us to first explain what happens if we remove one of the predictors from our model. 
 Hence, as the first step we study leave-one-predictor-out estimates (LOP) which. We remind the reader that the notations for the leave-one-predictor-out estimate are presented in Section \ref{sec:notation}.

\begin{theorem}\label{thm:lopo}
Let $\lophvb{i}$ be the original estimate $\hvb$ without $i$th component. Then under Assumptions O.\ref{ass:Convex}-O.\ref{ass:boundhvb}, we have 
\[
	\sup_{i,j}\ve_j^{\t}(\lopvy{i}-\XIb\lopvb{i})
	\spaceequal
	\op{\frac{\plogn}{\ac^{\rho+2}}}
	,
\]
and
\[
	\lopvb{i}
	\spaceequal 
	\lophvb{i}+(b_i-\beta_{0,i})\bA_i^{-1}\XIb^{\t}\diag{\dfl''(\lopvy{i}-\XIb\lopvb{i})}\vcx{i}+\bvepsilon^i,
\]
where
\[
	\bA_i
	\spaceequal
	\XIb^{\t}\diag{\dfl''(\lopvy{i}-\XIb\lopvb{i})}\XIb+\lambda \diag{\dreg''(\lopvb{i})}
	,
\]
and
\[ 
	b_i
	\spaceequal
	\argmin_{b\in\bbR} \frac{1}{2}\left(b-\beta_{0,i}-\frac{1}{a_i}\vcx{i}^{\t}\fl'(\lopvy{i}-\XIb\lopvb{i})\right)^2+\frac{\lambda}{a_i}\reg(b)
	.
\]
In the last equation, $a_i$ is defined as
\[
	a_i
	\spaceequal
	\vcx{i}^{\t}\left(\diag{\dfl''(\lopvy{i}-\XIb\lopvb{i})}^{-1}+\lambda\XIb\diag{\dreg''(\lopvb{i})}\XIb^{\t}\right)^{-1}\vcx{i}
	.
\] 
Moreover, for large enough $n$,
\begin{eqnarray}
	\sup_i\|\bvepsilon^i\| 
	&\spaceequal&
	\op{\frac{\plogn\cdot \bc^{1+\alpha}}{n^{\frac{\alpha}{2}}\cdot \ac^{32\rho+7}}}, \nonumber \\
	\sup_{i,j}\left|\ve_j^{\t}(\lophvb{i}-\lopvb{i})\right|
	&\spaceequal&
	\op{\frac{\plogn\cdot \bc^{1+\alpha}}{n^{\frac{\alpha}{2}}\cdot \ac^{32\rho+7}}}
	.
\end{eqnarray}
\end{theorem}

The proof of this theorem is presented in Section \ref{sec:lopo}.

Now based on the leave-one-predictor-out estimate, $\lopvb{i}$, we construct a new copy of $\hQ_i^{-1}$, called $\bQ_i $ in the following way:
\[
	\bQ_i 
	\spaceequal
	\vI+\diag{\dfl''(\lopvy{i}-\XIb\lopvb{i})}^{1/2}\XIb \bD^i\XIb^{\t}\diag{\dfl''(\lopvy{i}-\XIb\lopvb{i})}^{1/2},
\] 
where
\[
	\bD^i
	\spaceequal
	\left(\lambda\diag{\dreg''(\lopvb{i})}\right)^{-1}
	.
\]
Note that an $\bQ_i$ has two major properties: (i) It is independent of $\bvcx{i}$, and (ii) it is close to $\hQ_i$. The second property is confirmed in the following lemam:

\begin{lemma}\label{lem:q1}
Suppose Assumption O.\ref{ass:Convex}-O.\ref{ass:True} hold. Let  $\bvcx{i}=\diag{\dfl''(\lopvy{i}-\XIb\lopvb{i})}^{\frac{1}{2}}\vcx{i}$. Then 
\[
	\sup_i\left|\hvcx{i}^{\t}\hQ_i^{-1}\hvcx{i}-\bvcx{i}^{\t}\bQ_i^{-1}\bvcx{i}\right|
	\spaceequal
	\op{\frac{\plogn\cdot \bc^{1+\alpha}}{n^{\frac{\alpha^2}{2}}\cdot \ac^{64\rho+15}}}
	,
\]
\end{lemma}

The proof of this lemma is presented in Section \ref{sec:q1}.
The independence of $\bQ_i$ on $\bvcx{i}$ enables us to prove the concentration of $\bvcx{i}^{\t}\bQ_i^{-1}\bvcx{i}$; 
Due to Assumption O.\ref{ass:Smoothness2} and Theorem \ref{thm:lopo}, we have
\BE
	\sup_{i,j}\fl''(\ve_j^{\t}(\lopvy{i}-\XIb\lopvb{i}))
	&\leq&
	\op{\frac{\plogn}{\ac^{3\rho}}}
	.	\label{eq:add_trxbx_eq1}
\EE
Hence, with the facts that $\bQ_i$ and $\dfl''(\lopvy{i}-\XIb\lopvb{i})$ are independent of $\vcx{i}$, the minimal eigenvalue of $\bQ_i$ is at least $1$ and $\vcx{i}$ has i.i.d.~subGaussian components, from Hanson-Wright inequality, we have
\BE
	\sup_i\left|\bvcx{i}^{\t}\bQ_i^{-1}\bvcx{i}-\frac{\sigma_i^2}{n}\trace\left((\dbfl^{i})^{\frac{1}{2}}\bQ_i^{-1}(\dbfl^{i})^{\frac{1}{2}}\right)\right|
	&=&
	\op{\frac{\ln n}{\sqrt{n}}\cdot \sup_{i,j}\fl''(\ve_j^{\t}(\lopvy{i}-\XIb\lopvb{i}))}
		\n\\
	&=&
	\op{\frac{\plogn}{\sqrt{n}\cdot \ac^{3\rho}}}
	,	\n
\EE
where $\dbfl^{i}$ is a short hand for $\diag{\dfl''(\lopvy{i}-\XIb\lopvb{i})}$. To obtain the first equality we use similar argument as the ones used in the derivation of \eqref{eq:argument1}. 
Note that even though we have finally proved that $\bvcx{i}^{\t}\bQ_i^{-1}\bvcx{i}$ is concentrating, we have not proved that it is concentrating around $\sigma_i^2\dotp{\frac{\hfl''}{1+\frac{1}{n}\trace(\pB^{-1}\vSigma)\cdot\hfl''}}$ as required by Lemma \ref{lem:trxbx}.  Hence, our last step is to prove
\[
\sup_i\left|\frac{1}{n}\trace\left((\dbfl^{i})^{\frac{1}{2}}\bQ_i^{-1}(\dbfl^{i})^{\frac{1}{2}}\right)-\dotp{\frac{\hfl''}{1+\frac{1}{n}\trace(\pB^{-1}\vSigma)\cdot\hfl''}}\right| \ = \ \op{\frac{\plogn\cdot \bc^{1+\alpha}}{n^{\frac{\alpha^2}{2}}\cdot \ac^{64\rho+15}}}
.
\] 
We prove this in two steps. Our next lemma simplifies the expression  $\frac{1}{n}\trace\left((\dbfl^{i})^{\frac{1}{2}}\bQ_i^{-1}(\dbfl^{i})^{\frac{1}{2}}\right)$. 

\begin{lemma}\label{lem:q2}
Suppose Assumption O.\ref{ass:Convex}-O.\ref{ass:True} hold. For large enough $n$, we have
\[
	\sup_i\left|\frac{1}{n}\trace\left((\dbfl^{i})^{\frac{1}{2}}\bQ_i^{-1}(\dbfl^{i})^{\frac{1}{2}}\right)-\frac{1}{n} \trace\left( \diag{\dhfl''}^{\frac{1}{2}}(I+\hX \vD\hX^{\t})^{-1}\diag{\dhfl''}^{\frac{1}{2}}\right)\right|
\]
is at most
\[
	\op{\frac{\plogn\cdot \bc^{1+\alpha}}{n^{\frac{\alpha^2}{2}}\cdot \ac^{64\rho+15}}}
	.
\]
\end{lemma}
The proof of this lemma is presented in Section \ref{sec:q2}. Finally, we show that
\[
	\sup\left|\frac{1}{n} \trace\left( \diag{\dhfl''}^{\frac{1}{2}}(I+\hX \vD\hX^{\t})^{-1}\diag{\dhfl''}^{\frac{1}{2}}\right)-\dotp{\frac{\hfl''}{1+\frac{1}{n}\trace(\pB^{-1}\vSigma)\cdot\hfl''}}\right| \ = \ \op{\frac{\plogn }{n^{\frac{\alpha^2}{2}}\cdot \ac^{14\rho+9}}}
	.
\]
By Matrix Inversion Lemma, Assumption O.\ref{ass:Smoothness} and \eqref{eq:add_trxbx_eq1}, we have
\BE
	\lefteqn{\frac{1}{n}\trace\left( \diag{\dhfl''}^{\frac{1}{2}}(I+\hX \vD\hX^{\t})^{-1}\diag{\dhfl''}^{\frac{1}{2}}\right)}
		\n\\
	&=&
	\frac{1}{n}\trace\left(\diag{\dhfl''}-\diag{\dhfl''}\vX\pB^{-1}\vX^{\t}\diag{\dhfl''}\right)
		\n\\
	&=&
	\frac{1}{n}\sum_{i=1}^n\hfl''_i-\hfl_i''^2\vrx{i}^{\t}\pB^{-1}\vrx{i}
		\n\\
	&=&
	\frac{1}{n}\sum_{i=1}^n\frac{\hfl_i''}{1+\hfl_i''\cdot \vrx{i}^{\t}\pA_{i}^{-1}\vrx{i}}
		\n\\
	&\stackrel{\text{(i)}}{=}&
	\frac{1}{n}\sum_{i=1}^n\frac{\hfl_i''}{1+\hfl_i''\cdot \vrx{i}^{\t}\oA_{i}^{-1}\vrx{i}}+\op{\frac{\plogn }{n^{\frac{\alpha^2}{2}}\cdot \ac^{14\rho+9}}}
		\n\\
	&\stackrel{\text{(ii)}}{=}&
	\frac{1}{n}\sum_{i=1}^n\frac{\hfl_i''}{1+\hfl_i''\cdot \frac{1}{n}\trace(\oA_{i}^{-1}\vSigma)}+\op{\frac{\plogn }{n^{\frac{\alpha^2}{2}}\cdot \ac^{14\rho+9}}}
		\n\\
	&\stackrel{\text{(iii)}}{=}&
	\frac{1}{n}\sum_{i=1}^n\frac{\hfl_i''}{1+\hfl_i''\cdot \frac{1}{n}\trace(\pB^{-1}\vSigma)}+\op{\frac{\plogn }{n^{\frac{\alpha^2}{2}}\cdot \ac^{14\rho+9}}}
	,	\n
\EE
where Equality (i) holds due to Lemma \ref{lem:ABswitch2}, Equality (ii) holds due to Lemma \ref{lem:minev} and independency between  $\vrx{i}$ and $\oA_{i}$, and Equality (iii) holds due to Lemma \ref{lem:ABswitch}. This completes the proof. 

$\hfill \square$


\subsubsection{Proof of Theorem \ref{thm:lopo}}\label{sec:lopo}

First note that by the definition of $\bA_i$, we have
\BEQ
\label{eq:lowerbound_example_eq6}
\BS
	\inf_i\sigma_{\min}(\bA_i) 
	&=
	\inf_i\min_{\|\vu\|=1}\vu^{\t}\left(\XIb^{\t}\diag{\dfl''(\lopvy{i}-\XIb\lopvb{i})}\XIb+\lambda \diag{\dreg''(\lopvb{i})}\right)\vu
		\\
	&\stackrel{\text{(i)}}{\geq}
	\inf_i\min_{\|\vu\|=1}\vu^{\t}\left(\XIb^{\t}\cdot \ac\vI\cdot \XIb\right)\vu
		\\
	&\stackrel{\text{(ii)}}{\geq}
	\ac\sd \spaceequal \Omega_p\left(\ac\right)
	,
\end{split} 	
\EEQ
where Inequality (i) is due to Assumption O.\ref{ass:Convex} and O.\ref{ass:Smoothness} and Inequality (ii) is due to Lemma \ref{lem:minev}. Hence, the inverse of $\bA_i$ exists and the minimal eigenvalue of $\bA_i$ is at least $\Omega_p\left(\ac\right)$. Then note that since $\lopvb{i}$ can be considered as the solution for the generalized linear regression problem with data given by $(\XIb,\lopvy{i})$, we can follow the same proof of bounding $y_j-\vrx{j}^{\t}\hvb$ in Lemma \ref{lem:supnormDi} and obtain
\BE
	\sup_{i,j}\ve_j^{\t}(\lopvy{i}-\XIb\lopvb{i})
	&\leq&
	\op{\frac{\plogn}{\ac^{\rho+2}}}
	.	\label{eq:add_lopo_eq1}
\EE
To prove the rest of Theorem \ref{thm:lopo}, we first prove the following weaker result:

\begin{lemma}\label{lem:weakresult}
 Under Assumptions O.\ref{ass:Convex}-O.\ref{ass:True}, we have  
\[
	\lopvb{i}
	\spaceequal
	\lophvb{i}+(-\beta_{0,i})\bA_i^{-1}\XIb^{\t}\diag{\dfl''(\lopvy{i}-\XIb\lopvb{i})}\vcx{i}+\bvepsilon^i_{\tw}
	,
\]
where  
\[
	\bA_i
	\spaceequal
	\XIb^{\t}\diag{\dfl''(\lopvy{i}-\XIb\lopvb{i})}\XIb+\lambda \diag{\dreg''(\lopvb{i})}
	.
\]
Moreover, for large enough $n$,
\BEQ
	\sup_i\|\bvepsilon^i_{\tw}\| 
	\spaceequal
	\op{\frac{\plogn}{\ac^{6\rho+3}}}
	\qand
	\sup_{i,j}\left|\ve_j^{\t}(\lophvb{i}-\lopvb{i})\right|
	\spaceequal
	\op{\frac{\plogn}{\ac^{6\rho+3}}}
	.	\label{eq:add_goal_1}
\EEQ
\end{lemma}
Before we prove this result, let us explain some of its main features and the role it will play in our overall proof of Theorem  \ref{thm:lopo}. First, note that there are two main differences between this result and the proof of Theorem \ref{thm:lopo}. 
\BI

\item[(i)] $b_i$ is replaced with $0$. 

\item[(ii)]  Lemma \ref{lem:weakresult} requires $\|\bvepsilon^i_{\tw}\|$ and $\left|\ve_j^{\t}(\lophvb{i}-\lopvb{i})\right|$ to be $\op{\frac{\plogn}{\ac^{6\rho+3}}}$ rather than \\
$\op{\frac{\plogn\cdot \bc^{1+\alpha}}{n^{\frac{\alpha}{2}}\cdot \ac^{32\rho+7}}}$ which is required by Theorem \ref{thm:lopo}.

\EI
We can use the same strategy to prove both Lemma \ref{lem:weakresult} and Theorem \ref{thm:lopo}. We first prove Lemma \ref{lem:weakresult}. This result helps us bound the value of $b_i$. This bound on $b_i$ will then enable us to  prove Theorem \ref{thm:lopo}. Let us define $b_0=0$ and first show the weaker result for $b_0$. Later, we will replace $b_0$ with $b_i$ for $i\geq 1$ and prove Theorem \ref{thm:lopo} at the end of this subsection.

\begin{proof}[Proof of Lemma \ref{lem:weakresult}]
 Define
\BE 
	\loptvb{i}
	&=&
	\lopvb{i}-(b_0-\beta_{0,i})\bA_i^{-1}\XIb^{\t}\diag{\dfl''(\lopvy{i}-\XIb\lopvb{i})}\vcx{i}
	,	\label{eq:deftovb21}
\EE
and $\eloptvb{i}$ be $\loptvb{i}$ with $b_0$ inserted at $i$th component, i.e,
\BE
	\eloptb{i}_j
	&=&
	\left\{\begin{aligned}
			&\loptib{i}_j, 	&& j<i
				\\
			&b_0, 	&& j=i
				\\
			&\loptib{i}_{j-1}, 	&& j>i
		\end{aligned}\right.
	.	\n
\EE
Note that
\begin{eqnarray}
	\|\bvepsilon^i_{\tw}\|
	&\spaceequal& \|\lopvb{i} -	\lophvb{i}-(-\beta_{0,i})\bA_i^{-1}\XIb^{\t}\diag{\dfl''(\lopvy{i}-\XIb\lopvb{i})}\vcx{i}\| \nonumber \\
	&=& \|\lopvb{i} - (\lopvb{i} - \loptvb{i} + \lophvb{i} )\|=  \|\loptvb{i}-\hvb^{\backslash i}\|
	\qleq 
	\|\eloptvb{i}-\hvb\|
	.
\end{eqnarray}
To bound $\|\bvepsilon^i_{\tw}\|$, we use a trick similar to the one used in the proof of Proposition \ref{thm:looo} in Section \ref{sec:looo}. Define
\begin{eqnarray}
	L(\vbeta)
	\spaceequal
	-\vX^{\t}\fl'(\vy-\vX\vbeta)+\lambda\reg'(\vbeta)
	.	\label{eq:defL_eq1}
\end{eqnarray}
Similar to the proof of Proposition \ref{thm:looo} in Section \ref{sec:looo} it is straightforward to show that
\BEQ
	\sup_i\|\hvb-\eloptvb{i}\|
	\qleq
	\op{\frac{1}{\ac}}\cdot \sup_i\|L(\hvb)-L(\eloptvb{i})\|
	\spaceequal
	\op{\frac{1}{\ac}}\cdot \sup_i\|L(\eloptvb{i})\|
	.	\label{eq:suptovbhvb21}
\EEQ
Hence, we would like to show that
\BE
	\sup_i\|L(\eloptvb{i})\|
	&\leq&
	\op{\frac{\plogn}{\ac^{6\rho+2}}}
	.	\label{eq:L21}
\EE
Toward this goal we first define $L^{\backslash i}(\eloptvb{i})$ the entire $L(\eloptvb{i})$ without  its $i$th component, and prove that $\sup_i\|L^{\backslash i}(\eloptvb{i})\|$. Then, we will look at the $\i^{\rm th}$ component of  $L(\eloptvb{i})$ and find an upper bound for that component too. \\

\noindent Let us start with bounding $\sup_i\|L^{\backslash i}(\eloptvb{i})\|$. According to the definition of $\lopvb{i}$, we have 
\BE
	-\XIb^{\t}\fl'(\lopvy{i}-\XIb\lopvb{i})+\lambda\reg'(\lopvb{i})
	&=&
	\vzero
	. 	\label{eq:jthout21}
\EE
Furthermore, from \eqref{eq:deftovb21} we have
\BEQ
\label{eq:add_lopo_eq21}
\BS
	\vzero
	&=
	\XIb^{\t}\diag{\dfl''(\lopvy{i}-\XIb\lopvb{i})}\left(\XIb(\loptvb{i}-\lopvb{i})+(b_0-\beta_{0,i})\vcx{i}\right)
		\\
	&\quad
	+\lambda \diag{\dreg''(\lopvb{i})}(\loptvb{i}-\lopvb{i})
	.
\end{split}
\EEQ
Hence, 
\BE
	L^{\backslash i}(\eloptvb{i})
	&=&
	-\XIb^{\t}\fl'\left(\lopvy{i}-\XIb\loptvb{i}-(b_0-\beta_{0,i})\vcx{i}\right)+\lambda \reg'(\loptvb{i})
		\n\\
	&=&
	\XIb^{\t}\left(\fl'(\lopvy{i}-\XIb\lopvb{i})-\fl'\left(\lopvy{i}-\XIb\loptvb{i}-(b_0-\beta_{0,i})\vcx{i}\right)\right)-\lambda\left(\reg'(\lopvb{i})- \reg'(\loptvb{i})\right)
		\n\\
	&=&
	\XIb^{\t}\diag{\dfl''\left(\vbeta_{\Xi}\right)}\left(\XIb(\loptvb{i}-\lopvb{i})+(b_0-\beta_{0,i})\vcx{i}\right)+\lambda \diag{\dreg''(\vbeta^i_{\vxi'})}(\loptvb{i}-\lopvb{i})
		\n\\
	&=&
	\XIb^{\t}\diag{\dfl''\left(\vbeta_{\Xi}\right)-\dfl''(\lopvy{i}-\XIb\lopvb{i})}\left(\XIb(\loptvb{i}-\lopvb{i})+(b_0-\beta_{0,i})\vcx{i}\right)
		\n\\
	&&
	+\lambda \diag{\dreg''(\vbeta^i_{\vxi'})-\dreg''(\lopvb{i})}(\loptvb{i}-\lopvb{i}).
		\n
\EE
In these equations, we have used the definitions $\vbeta_{\Xi}=\lopvy{i}-(\Xi\XIb\loptvb{i}+(I-\Xi)\XIb\lopvb{i})-\Xi(b_0-\beta_{0,i})\vcx{i}$, $\Xi=\diag{\xi_1,\cdots \xi_n}$ and $\beta^i_{\vxi',j}=\xi'_j\loptib{i}_j+(1-\xi'_j)\lopb{i}_j$ for some $\xi_j,\xi'_j\in [0,1]$. Furthermore, to obtain the last equality we have used \eqref{eq:add_lopo_eq21}. By Assumption O.\ref{ass:Holder} and Lemma \ref{lem:minev} with similar proof for \eqref{eq:revision_add_eq3}, we have
\BEQ
\BS
	\sup_i\|L^{\backslash i}(\eloptvb{i})\|
	&\leq
	\op{4\delta}\cdot \sqrt{n}\sup_{i,j}C_l\left|\ve_j^{\t}\XIb(\loptvb{i}-\lopvb{i})+(b_0-\beta_{0,i})x_{ij}\right|^{1+\alpha}
		\\
	&\quad
	+\lambda\sqrt{p}\sup_{i,j}C_r\left|\loptib{i}_j-\lopb{i}_j\right|^{1+\alpha}
		\n
\end{split}
\EEQ
Our next goal is to show that
\begin{subequations}
\begin{align}
	\sup_{i,j}|\ve_j^{\t}\XIb(\loptvb{i}-\lopvb{i})|
	&\spaceequal
	\op{\frac{\plogn}{\sqrt{n}\cdot \ac^{3\rho+1}}\sup_i|b_0-\beta_{0,i}|}
	,	\label{eq:suptovbijeq21}\\
	\sup_{i,j}|\loptib{i}_j-\lopb{i}_j|
	&\spaceequal
	\op{\frac{\plogn}{\sqrt{n}\cdot \ac^{3\rho+1}}\sup_i|b_0-\beta_{0,i}|}
	.	\label{eq:suptovbijeq31}
\end{align}
\end{subequations}
Note that if we prove these two claims, then we can combine them with Lemma \ref{lem:xwnorm} and obtain 
\begin{equation}\label{eq:add_lopo_1_eq1}
	\sup_i\|L^{\backslash i}(\eloptvb{i})\| \leq
	\op{\frac{\plogn}{n^{\frac{\alpha}{2}}}\cdot \frac{\sup_i|b_0-\beta_{0,i}|^{1+\alpha}}{\ac^{6\rho+2}}}.
\end{equation}
Since $b_0=0$, according to Lemma \ref{lem:xwnorm}, $\sup_i|b_0-\beta_{0,i}|=\op{\plogn}$, which proves an upper bound for $\sup_i\|L^{\backslash i}(\eloptvb{i})\|$. Hence, let us discuss how \eqref{eq:suptovbijeq21} and \eqref{eq:suptovbijeq31} can be proved. To prove these equations, note that, by \eqref{eq:deftovb21}, we just need to show
\BEQ
\label{eq:add_lopo_1_eq4}
\BS
	\sup_{i,j}\left|\ve_j^{\t}\XIb\bA_i^{-1}\XIb^{\t}\diag{\dfl''(\lopvy{i}-\XIb\lopvb{i})}\vcx{i}\right|
	&=
	\op{\frac{\plogn}{\sqrt{n}\cdot \ac^{3\rho+1}}}
	,	\\
	\sup_{i,j}\left|\ve_j^{\t}\bA_i^{-1}\XIb^{\t}\diag{\dfl''(\lopvy{i}-\XIb\lopvb{i})}\vcx{i}\right|
	&=
	\op{\frac{\plogn}{\sqrt{n}\cdot \ac^{3\rho+1}}}
	.
\end{split}	
\EEQ
We use a technique similar to the one used for proving \eqref{eq:suptovbij} in Section \ref{sec:looo}. Recall that, according to \eqref{eq:lowerbound_example_eq6}, the minimal eigenvalue of $\bA_i$ is $\Omega_p\left(\ac\right)$. Hence, with Lemma \ref{lem:minev}, \eqref{eq:add_lopo_eq1} and Assumption O.\ref{ass:Smoothness2}, we have
\BE
	\sup_{i,j}\left\|\diag{\dfl''(\lopvy{i}-\XIb\lopvb{i})}\XIb\bA_i^{-1}\XIb^{\t}\ve_j\right\|
	&\leq&
	\op{\frac{1}{\ac}}\cdot \sup_{i,j}\fl''(\lopy{i}_j-\ve_j^{\t}\XIb\lopvb{i})
		\n\\
	&\leq& 
	\op{\frac{\plogn}{\ac^{3\rho+1}}}
	,	\n\\
	\sup_{i,j}\left\|\diag{\dfl''(\lopvy{i}-\XIb\lopvb{i})}\XIb\bA_i^{-1}\ve_j\right\|
	&\leq&
	\op{\frac{1}{\ac}} \cdot \sup_{i,j}\fl''(\lopy{i}_j-\ve_j^{\t}\XIb\lopvb{i})
		\n\\
	&\leq& 
	\op{\frac{\plogn}{\ac^{3\rho+1}}}
	.	\n
\EE
Then, since $\vcx{i}$ is independent of $\lopvy{i},\lopvb{i},\XIb$ and $\bA_i$, we conclude that \eqref{eq:add_lopo_1_eq4} holds, which in turn implies \eqref{eq:suptovbijeq21} and \eqref{eq:suptovbijeq31}.

Now let us find an upper bound for the $i^{\rm th}$ component of $L(\eloptvb{i})$ denoted as $L_i(\eloptvb{i})$. By Taylor expansion, we have
\BEQ
\label{eq:add_lopo_1_eq2}
\BS
	L_{i}(\eloptvb{i})
	&=
	-\vcx{i}^{\t}\fl'\left(\lopvy{i}-\XIb\loptvb{i}-(b_0-\beta_{0,i})\vcx{i}\right)+\lambda \reg'(b_0)
		\\
	&=
	-\vcx{i}^{\t}\fl'\left(\lopvy{i}-\XIb\loptvb{i}-(b_0-\beta_{0,i})\vcx{i}\right)+\vcx{i}^{\t}\fl'(\lopvy{i}-\XIb\lopvb{i})
		\\
	&\quad
	+\lambda \reg'(b_0)-\vcx{i}^{\t}\fl'(\lopvy{i}-\XIb\lopvb{i})
		\\
	&=
	\vcx{i}^{\t}\diag{\dfl''\left(\vbeta_{\Xi}\right)}\left(\XIb(\loptvb{i}-\lopvb{i})+(b_0-\beta_{0,i})\vcx{i}\right)+\lambda \reg'(b_0)-\vcx{i}^{\t}\fl'(\lopvy{i}-\XIb\lopvb{i})
		\\
	&=
	\underbrace{\vcx{i}^{\t}\diag{\dfl''\left(\vbeta_{\Xi}\right)-\dfl''(\lopvy{i}-\XIb\lopvb{i})}\left(\XIb(\loptvb{i}-\lopvb{i})+(b_0-\beta_{0,i})\vcx{i}\right)}_{\text{part 1}}
		\\
	&\quad
	+\underbrace{\lambda \reg'(b_0)-\vcx{i}^{\t}\fl'(\lopvy{i}-\XIb\lopvb{i})}_{\text{part 2}}
		\\
	&\quad
	+\underbrace{\vcx{i}^{\t}\diag{\dfl''(\lopvy{i}-\XIb\lopvb{i})}\left(\XIb(\loptvb{i}-\lopvb{i})+(b_0-\beta_{0,i})\vcx{i}\right)}_{\text{part 3}}
	.	\n\\
\end{split}		
\EEQ
In the rest of the proof, we obtain separate upper bounds for part 1, part 2, and part 3. For part 1, similar to the proof of \eqref{eq:add_lopo_1_eq1}, we have that by \eqref{eq:suptovbijeq21}-\eqref{eq:suptovbijeq31}, Lemma \ref{lem:xwnorm} and Assumption O.\ref{ass:Holder}, we have
\BE
	\text{part 1}
	&\leq&
	\sup_i\|\vcx{i}\|\cdot \sqrt{n}\sup_{i,j}C_l\left|\ve_j^{\t}\XIb(\loptvb{i}-\lopvb{i})+(b_0-\beta_{0,i})x_{ij}\right|^{1+\alpha}
		\n\\
	&\leq&
	\op{\frac{\plogn}{n^{\frac{\alpha}{2}}}\cdot \frac{\sup_i|b_0-\beta_{0,i}|^{1+\alpha}}{\ac^{6\rho+2}}}
	.	\n
\EE
For part 2, note that $\reg'(b_0)=\reg'(0)=O(1)$. Then, since $\vcx{i}$ is independent of $\lopvy{i}-\XIb\lopvb{i}$, and from Assumption O.\ref{ass:True}, $\vcx{i}$ has i.i.d.~subGaussian components. Hence, with Hanson Wright inequality,  Assumption O.\ref{ass:Smoothness2} and \eqref{eq:add_lopo_eq1}, we have 
\BE
	\text{part 2}
	&\leq&
	\sup_i |\lambda \reg'(b_0)-\vcx{i}^{\t}\fl'(\lopvy{i}-\XIb\lopvb{i})|
		\n\\
	&\leq&
	\op{1}+\op{\frac{\|\fl'(\lopvy{i}-\XIb\lopvb{i})\|}{\sqrt{n}}\cdot \ln n}
		\n\\
	&=&
	\op{\frac{\plogn}{\ac^{(\rho+2)(\rho+1)}}}
	.	\n
\EE
For part 3, we have 
\[
	\text{part 3}
	\qleq
	\sup_i\|\vcx{i}\|\cdot \sup_{i,j}\fl''(\ve_j^{\t}(\lopvy{i}-\XIb\lopvb{i}))\cdot \sqrt{n}\cdot \sup_{i,j}|\ve_j^{\t}\XIb(\loptvb{i}-\lopvb{i})+(b_0-\beta_{0,i})x_{ij}|
	.
\]
Apply Lemma \ref{lem:xwnorm}, Assumption O.\ref{ass:Smoothness2}, \eqref{eq:add_lopo_eq1} and \eqref{eq:suptovbijeq21} squentially, we have
\BE
	\text{part 3}
	&\leq&
	\op{1} \cdot \op{\frac{\plogn}{\ac^{(\rho+2)\rho}}} \cdot \sqrt{n}\cdot \op{\frac{\plogn}{\sqrt{n}}\cdot \frac{\sup_i|b_0-\beta_{0,i}|}{\ac^{3\rho+1}}}
		\n\\
	&=&
	\op{\frac{\plogn}{\ac^{(\rho+2)\rho}}\cdot \frac{\sup_i|b_0-\beta_{0,i}|}{\ac^{3\rho+1}}}
	.	\n
\EE
Note that since $b_0=0$, by Lemma \ref{lem:xwnorm} we know $\sup_i|b_0-\beta_{0,i}|=\op{\plogn}$. Hence, by combining  the above three upper bounds we conclude that
\[
	\sup_i|L_{i}(\eloptvb{i})|
	\qleq
	\op{\frac{\plogn}{\ac^{6\rho+2}}}
	.
\]
Note that by combining this result with \eqref{eq:add_lopo_1_eq1}, we obtain
\[
	\sup_i\|L(\eloptvb{i})\|
	\qleq
	\op{\frac{\plogn}{\ac^{6\rho+2}}}
	.
\]
Therefore, according to \eqref{eq:suptovbhvb21}, we have
\[
	\|\bvepsilon^i_{\tw}\|
	\qleq
	\sup_i\|\hvb-\eloptvb{i}\|
	\qleq
	\op{\frac{\plogn}{\ac^{6\rho+3}}}
	.
\]
Combine with \eqref{eq:suptovbijeq31}, we have
\[
 	\sup_{i,j}\left|\ve_j^{\t}(\lophvb{i}-\lopvb{i})\right|
	\spaceequal
	\op{\frac{\plogn}{\ac^{6\rho+3}}}
	.
\]
This completes the proof of Lemma \ref{lem:weakresult}. \end{proof}

Now we would like to prove Theorem \ref{thm:lopo}. As discussed before we replace $b_0$ with $b_i$ in the proof of Lemma \ref{lem:weakresult} and update the proof accordingly with additional Assumption O.\ref{ass:boundhvb}. With a slight abuse of notation we redefine $\loptvb{i}$ and $\eloptvb{i}$ by replacing $b_0$  with $b_i$. In other words, in the rest of the proof we have
\BE
	\loptvb{i}
	&=&
	\lopvb{i}-(b_i-\beta_{0,i})\bA_i^{-1}\XIb^{\t}\diag{\dfl''(\lopvy{i}-\XIb\lopvb{i})}\vcx{i}
	,	\label{eq:deftovb22}
\EE
and
\BE
	\eloptb{i}_j
	&=&
	\left\{\begin{aligned}
			&\loptib{i}_j, 	&& j<i
				\\
			&b_i, 	&& j=i
				\\
			&\loptib{i}_{j-1}, 	&& j>i
		\end{aligned}\right.
	.	\n
\EE
Then, we can follow the same steps as the ones in the proof of Lemma \ref{lem:weakresult} and conclude that 
\begin{subequations}
\begin{align}
	\sup_{i,j}|\ve_j^{\t}\XIb(\loptvb{i}-\lopvb{i})|
	&\spaceequal
	\op{\frac{\plogn}{\sqrt{n}\cdot \ac^{3\rho+1}}\sup_i|b_i-\beta_{0,i}|}
	,	\label{eq:suptovbijeq22}\\
	\sup_{i,j}|\loptib{i}_j-\lopb{i}_j|
	&\spaceequal
	\op{\frac{\plogn}{\sqrt{n}\cdot \ac^{3\rho+1}}\sup_i|b_i-\beta_{0,i}|}
	,	\label{eq:suptovbijeq32}
\end{align}
\end{subequations} 
and
\BE
	\sup_i\|L^{\backslash i}(\eloptvb{i})\|
	&\leq&
	\op{\frac{\plogn}{n^{\frac{\alpha}{2}}}\cdot \frac{\sup_i|b_i-\beta_{0,i}|^{1+\alpha}}{\ac^{6\rho+2}}}
	.	\label{eq:add_9.38}
\EE
Similarly, we want to obtain an upper bound for the $i^{\rm th}$ component of $L(\eloptvb{i})$ denoted with $L_i(\eloptvb{i})$. By Taylor expansion, we have
\BE
	\lefteqn{L_{i}(\eloptvb{i})}
		\n\\
	&=&
	-\vcx{i}^{\t}\fl'\left(\lopvy{i}-\XIb\loptvb{i}-(b_i-\beta_{0,i})\vcx{i}\right)+\lambda \reg'(b_i)
		\n\\
	&=&
	\vcx{i}^{\t}\diag{\dfl''\left(\vbeta_{\Xi}\right)}\left(\XIb(\loptvb{i}-\lopvb{i})+(b_i-\beta_{0,i})\vcx{i}\right)+\lambda \reg'(b_i)-\vcx{i}^{\t}\fl'(\lopvy{i}-\XIb\lopvb{i})
		\n\\
	&=&
	\vcx{i}^{\t}\diag{\dfl''\left(\vbeta_{\Xi}\right)-\dfl''(\lopvy{i}-\XIb\lopvb{i})}\left(\XIb(\loptvb{i}-\lopvb{i})+(b_i-\beta_{0,i})\vcx{i}\right)
		\n\\
	&&-a_i(b_i-\beta_{0,i})+\vcx{i}^{\t}\diag{\dfl''(\lopvy{i}-\XIb\lopvb{i})}\left(\XIb(\loptvb{i}-\lopvb{i})+(b_i-\beta_{0,i})\vcx{i}\right)
	.	\n
\EE
For the last equality we have used the following equality which is a simple conclusion of the definition of $b_i$ in \eqref{eq:bcondition2}:
\[
	a_i(b_i-\beta_{0,i})+\lambda \reg'(b_i)
	\spaceequal
	\vcx{i}^{\t}\fl'(\lopvy{i}-\XIb\lopvb{i})
	.
\]
In the following calculations, we use $\bdfl_i''$ as a shorthand for the matrix $\diag{\dfl''(\lopvy{i}-\XIb\loptvb{i})}$. According to the matrix inversion lemma, we have 
\BE
	\vcx{i}^{\t}\bdfl''_i\left(\XIb(\loptvb{i}-\lopvb{i})+(b_i-\beta_{0,i})\vcx{i}\right)
	&=&
	(b_i-\beta_{0,i})\left(\vcx{i}^{\t}\bdfl''_i\vcx{i}-\vcx{i}^{\t}\bdfl''_i\XIb\bA_i^{-1}\XIb^{\t}\bdfl''_i\vcx{i}\right)
		\n\\
	&=&
	(b_i-\beta_{0,i})a_i
	.	\n
\EE
Hence, when we replace $b_0$ with $b_i$, then part 2 and part 3 in \eqref{eq:add_lopo_1_eq2} cancel each other and only part 1 remains. In other words, we have 
\BE
	\sup_i|L_{i}(\eloptvb{i})|
	&=&
	\sup_{i,\Xi}\left|\vcx{i}^{\t}\diag{\dfl''(\vbeta_{\Xi})-\dfl''(\lopvy{i}-\XIb\loptvb{i})}\left(\XIb(\loptvb{i}-\lopvb{i})+(b_i-\beta_{0,i})\vcx{i}\right)\right|
		\n\\
	&\stackrel{\text{(i)}}{\leq}&
	\sup_i\|\vcx{i}\|\cdot \sqrt{n}\sup_{i,j}C_l\left|\ve_j^{\t}\XIb(\loptvb{i}-\lopvb{i})+(b_i-\beta_{0,i})x_{ij}\right|^{1+\alpha}
		\n\\
	&\stackrel{\text{(ii)}}{\leq}&
	\op{\frac{\plogn}{n^{\frac{\alpha}{2}}}\cdot \frac{\sup_i|b_i-\beta_{0,i}|^{1+\alpha}}{\ac^{6\rho+2}}}
	,	\n
\EE
where Inequality (i) is due to Assumption O.\ref{ass:Holder} and Inequality (ii) is due to \eqref{eq:suptovbijeq22}-\eqref{eq:suptovbijeq32} and Lemma \ref{lem:xwnorm}. Hence, if we combine this equation with \eqref{eq:add_9.38}, then we obtain
\[
	\sup_i\|L(\eloptvb{i})\|
	\qleq
	\op{\frac{\plogn}{n^{\frac{\alpha}{2}}}\cdot \frac{\sup_i|b_i-\beta_{0,i}|^{1+\alpha}}{\ac^{6\rho+2}}}
	.
\]
Therefore, similar to \eqref{eq:suptovbhvb21}, we have
\BE
	\|\bvepsilon^i\|
	\qleq
	\sup_i\|\hvb-\eloptvb{i}\|
	&\leq& \op{\frac{1}{\ac}}\cdot \sup_i\|L(\eloptvb{i})\|
		\n\\
	&\leq& \op{\frac{\plogn}{n^{\frac{\alpha}{2}}}\cdot \frac{\sup_i|b_i-\beta_{0,i}|^{1+\alpha}}{\ac^{6\rho+3}}}
	.	\label{eq:add_lopo_eq5}
\EE
Note that \eqref{eq:add_lopo_eq5} and \eqref{eq:suptovbijeq32} together imply that
\[
	\sup_{i,j}\left|\ve_j^{\t}(\lophvb{i}-\lopvb{i})\right|
	\spaceequal
	\op{\frac{\plogn}{n^{\frac{\alpha}{2}}}\cdot \frac{\sup_i|b_i-\beta_{0,i}|^{1+\alpha}}{\ac^{6\rho+3}}}
	.
\]
As a corollary of \eqref{eq:add_lopo_eq5}, we have 
\begin{corollary}\label{cor:bibound}
Under Assumption O.\ref{ass:Convex}-O.\ref{ass:True}, as $n \rightarrow \infty$, we have
\[
	\sup_i|b_i-\hb_i|
	\qleq
	\sup_i\|\hvb-\eloptvb{i}\|
	\qleq
	 \op{\frac{\plogn}{n^{\frac{\alpha}{2}}}\cdot \frac{\sup_i|b_i-\beta_{0,i}|^{1+\alpha}}{\ac^{6\rho+3}}}
	.
\]
\end{corollary}

Finally, to complete the proof of Theorem \ref{thm:lopo}, we just need to bound $\sup_{i}|b_i-\beta_{0,i}|$ with $\op{\frac{\plogn}{\ac^{13\rho+2}}\cdot\bc}$ under additional Assumption O.\ref{ass:boundhvb}. By Lemma \ref{lem:xwnorm}, we know we just need to bound $|b_i|$ by $\op{\frac{\plogn}{\ac^{13\rho+2}}\cdot \bc}$. Let $\eta(\cdot)$ denote the proximal operator of $\reg$, defined as
\[
	\eta(x,\lambda)
	\spaceequal
	\argmin_{y\in \bbR}\frac{1}{2}(x-y)^2+\lambda \reg(y)
	.
\]
Then recall the definition of $b_i$, we have
\BE 
	b_i
	&=&
	\eta\left(\beta_{0,i}+\frac{1}{a_i}\vcx{i}^{\t}\fl'(\lopvy{i}-\XIb\lopvb{i}),\frac{\lambda}{a_i}\right)
	,	\label{eq:bcondition2}
\EE
where 
\[
	a_i
	\spaceequal
	\vcx{i}^{\t}(\diag{\dfl''(\lopvy{i}-\XIb\lopvb{i})}^{-1}+\lambda\XIb\diag{\dreg''(\lopvb{i})}\XIb^{\t})^{-1}\vcx{i}
	.
\] 
Our first lemma summarizes a few properties of the prox function $\eta$.
\begin{lemma}\label{lem:eta'}
Let $f(x)$ be a convex function. If $f$ is twice-differentiable, then we have 
\[
	\frac{\partial \eta_f(x,\theta)}{\partial x}
	\spaceequal
	\frac{1}{1+\theta f''(\eta_f(x,\theta))}
	,
\]
where $\eta_f$ is the proximity operator of $f$, satisfying
\[
	\eta_{f}(x,\theta)
	\spaceequal
	\argmin_{y\in \bbR}\frac{1}{2}(x-y)^2+\theta f(y)
	.
\]
Hence, $\eta_f(x,\theta)$ is Lipchitz continuous with constant 1.
\end{lemma}
\begin{proof}
Since $f$ is convex, we know that $\eta_f(x,\theta)$ is uniquely defined for each $\theta$, and satisfies 
\[
	\eta_f(x,\theta)-x+\theta f'(\eta_f(x,\theta))
	\spaceequal
	0
	.
\]
Since $f$ is twice-differentiable, by taking a derivative with respect to $x$ from both sides of the above equation we obtain
\[
	\frac{\partial \eta_f(x,\theta)}{\partial x}-1+\theta f''(\eta_f(x,\theta))\cdot \frac{\partial \eta_f(x,\theta)}{\partial x}
	\spaceequal
	0
	,
\]
which completes the proof of the lemma.
\end{proof}
According to Assumption O.\ref{ass:Convex}, there exists a constant $\mu_{\min}$ such that $\reg$ achieves its minimum at $\mu_{\min}$. Hence,  $\eta(\mu_{\min};\cdot)\equiv \mu_{\min}$. Further, by Lemma \ref{lem:eta'}, we have $|\eta'|\leq 1$. Hence, \eqref{eq:bcondition2} implies that
\BE
	|b_i|
	&\leq&
	|\beta_{0,i}+\frac{1}{a_i}\vcx{i}^{\t}\fl'(\lopvy{i}-\XIb\lopvb{i})-\mu_{\min}|+|\eta(\mu_{\min},\lambda/a_i)|
		\n\\
	&\leq&
	|\beta_{0,i}|+|\frac{1}{a_i}\vcx{i}^{\t}\fl'(\lopvy{i}-\XIb\lopvb{i})|+2|\mu_{\min}|
		\n\\
	&\stackrel{\text{(i)}}{\leq}&
	\op{\ln n}+\frac{1}{|a_i|}\op{\frac{\|\fl'(\lopvy{i}-\XIb\lopvb{i})\|}{\sqrt{n}}\cdot \ln n}
		\n\\
	&\stackrel{\text{(ii)}}{\leq}&
	\op{\ln n}+\frac{1}{|a_i|}\op{\frac{\plogn}{\ac^{(\rho+2)(\rho+1)}}}
	,	\label{eq:add_9.57}
\EE
where Inequality (i) holds due to Lemma \ref{lem:xwnorm} and the facts that $\vcx{i}$ is independent of $\lopvy{i}-\XIb\lopvb{i}$, and Inequality (ii) is due to Assumption O.\ref{ass:Smoothness2} and \eqref{eq:add_lopo_eq1}. Hence, we just need to lower bound $|a_i|$. Note that, by definition of $a_i$, we have
\[
	\inf_i|a_i|
	\qgeq
	\frac{\inf_i\|\vcx{i}\|}{\max_i  \sigma_{\max}\left(\diag{\dfl''(\lopvy{i}-\XIb\lopvb{i})}^{-1}+\lambda\XIb\diag{\dreg''(\lopvb{i})}\XIb^{\t}\right)}
	.
\]
Note that by Assumption O.\ref{ass:Smoothness}, we know the maximum eigenvalue of $\diag{\dfl''(\lopvy{i}-\XIb\lopvb{i})}^{-1}$ is at most $\op{1/\ac}$. For the maximum eigenvalue of $\lambda\XIb\diag{\dreg''(\lopvb{i})}\XIb^{\t}$, note that by Assumption O.\ref{ass:Smoothness2}, we have
\[
	\sup_{i,j}\reg''(\ve_j^{\t}\lopvb{i})
	\qleq
	\op{1+\sup_{i,j}(\lopb{i}_j)^{\rho}}
	.
\]
According to Assumption O.\ref{ass:boundhvb} and \eqref{eq:add_goal_1} stated in Lemma \ref{lem:weakresult}, we have
\BE
	\sup_{i,j}\reg''(\ve_j^{\t}\lopvb{i})
	&\leq&
	\op{\frac{\plogn}{\ac^{9\rho}}+\bc}
	.	\label{eq:add_10.49}
\EE
Hence, due to Lemma \ref{lem:minev}, we have
\BE
	\inf_i|a_i|
	&\geq&
	\frac{\inf_i\|\vcx{i}\|}{1+\op{\frac{\plogn}{\ac^{9\rho}}+\bc}}
	\qgeq
	\Omega_p\left(\frac{1}{\frac{\plogn}{\ac^{9\rho}}+\bc}\right)
	,	\label{eq:add_10.48}
\EE
where the last inequality is due to Lemma \ref{lem:xwnorm}. Hence, by using Lemma \ref{lem:xwnorm} again, we have
\BEQ
\label{eq:add_eq_10.40pm}
\BS
	\sup_i|b_i-\beta_{0,i}|
	&\leq
	\op{\ln n}+\op{\frac{\plogn}{\ac^{9\rho}}+\bc}\cdot \op{\frac{\plogn}{\ac^{(\rho+2)(\rho+1)}}}
		\\
	&\leq
	\op{\frac{\plogn}{\ac^{13\rho+2}}\cdot\bc}
	.	
\end{split}
\EEQ
This completes the proof of Theorem \ref{thm:lopo}.


\subsubsection{Proof of Lemma \ref{lem:q1}}\label{sec:q1}

By Matrix Inversion Lemma, we have
\BEQ
	\hvcx{i}^{\t}\hQ_i^{-1}\hvcx{i}
	\spaceequal
	\hvcx{i}^{\t}\hvcx{i}-\vcx{i}^{\t}\hat{M}_i\vcx{i}
	\qand
	\bvcx{i}^{\t}\bQ_i^{-1}\bvcx{i}
	\spaceequal
	\bvcx{i}^{\t}\bvcx{i}-\vcx{i}^{\t}\bar{M}_i\vcx{i}
	,	\label{eq:Q}
\EEQ
where
\BEQ
\BS
	\hat{\vM}_i
	&=
	\diag{\dfl''(\vy-\vX\hvb)}\XIb\left(\XIb^{\t}\diag{\dfl''(\vy-\vX\hvb)}\XIb+\lambda\diag{\cdreg''(\lophvb{i})}\right)^{-1}
		\\
	&\quad
	\times\XIb^{\t}\diag{\dfl''(\vy-\vX\hvb)}
	,	\n
\end{split}
\EEQ
and
\[
	\bar{\vM}_i
	\spaceequal
	\vv^{\t}\left(\XIb^{\t}\diag{\dfl''(\lopvy{i}-\XIb\lopvb{i})}\XIb+\lambda\diag{\cdreg''(\lopvb{i})}\right)^{-1}\vv
	,
\]
where $\vv=\XIb^{\t}\diag{\dfl''(\lopvy{i}-\XIb\lopvb{i})}$. Hence, we need to show 
\[
	\sup_i\left|\vcx{i}^{\t}\left(\hat{\vM}_i-\diag{\dfl''(\vy-\vX\hvb)}-\bar{\vM}_i+\diag{\dfl''(\lopvy{i}-\XIb\lopvb{i})}\right)\vcx{i}\right|
\]
is at most
\[
	\op{\frac{\plogn\cdot \bc^{1+\alpha}}{n^{\frac{\alpha^2}{2}}\cdot \ac^{64\rho+15}}}
	.
\]
Let
\BEQ
\BS
	\check{\vM}_i
	&=
	\diag{\dfl''(\vy-\vX\hvb)}\XIb\left(\XIb^{\t}\diag{\dfl''(\lopvy{i}-\XIb\lopvb{i})}\XIb+\lambda\diag{\cdreg''(\lopvb{i})}\right)^{-1}
		\\
	&\quad
	\times \XIb^{\t}\diag{\dfl''(\vy-\vX\hvb)}
	.	\n
\end{split}
\EEQ
We just need to show the following three equations:
\BEQ
\label{eq:add_q1_eq1}
\BS
	\sup_i\left|\vcx{i}^{\t}\diag{\dfl''(\lopvy{i}-\XIb\lopvb{i})-\dfl''(\vy-\vX\hvb)}\vcx{i}\right|
	&=
	\op{\frac{\plogn\cdot \bc^{1+\alpha}}{n^{\frac{\alpha^2}{2}}\cdot \ac^{32\rho+7}}}
	,	\\
	\sup_i\left|\vcx{i}^{\t}(\hat{\vM}_i-\check{\vM}_i)\vcx{i}\right|
	&=
	\op{\frac{\plogn\cdot \bc^{1+\alpha}}{n^{\frac{\alpha}{2}}\cdot \ac^{38\rho+9}}}
	,	\\
	\sup_i\left|\vcx{i}^{\t}(\check{\vM}_i-\bar{\vM}_i)\vcx{i}\right|
	&=
	\op{\frac{\plogn\cdot \bc^{1+\alpha}}{n^{\frac{\alpha^2}{2}}\cdot \ac^{64\rho+15}}}
	.
\end{split}	
\EEQ
To show the first equation, recall the proof of Theorem \ref{thm:lopo} at the end of Section \ref{sec:lopo}. We have
\BEQ
\label{eq:difl}
\BS
	\lefteqn{\sup_{i,j}|\fl''(y_j-\vrx{j}^{\t}\hvb)-\fl''(\lopy{i}_j-\ve_j^{\t}\XIb\lopvb{i})|}
		\\
	&\leq
	\sup_{i,j}|\fl''(y_j-\vrx{j}^{\t}\hvb)-\fl''(y_j-\vrx{j}^{\t}\eloptvb{i})|+\sup_{i,j}|\fl''(y_j-\vrx{j}^{\t}\eloptvb{i})-\fl''(\lopy{i}_j-\ve_j^{\t}\XIb\lopvb{i})|
		\\
	&\stackrel{\text{(i)}}{\leq}
	C_l\left(\sup_{i,j}|\vrx{j}^{\t}(\hvb-\eloptvb{i})|^{\alpha}+\sup_{i,j}|\ve_j^{\t}\XIb(\loptvb{i}-\lopvb{i})+(b_i-\beta_{0,i})x_{ij}|^{\alpha}\right)
		\\
	&\stackrel{\text{(ii)}}{\leq}
	\op{\frac{\plogn\cdot \bc^{1+\alpha}}{n^{\frac{\alpha^2}{2}}\cdot \ac^{32\rho+7}}}
	+C_l\left(\sup_{i,j}|\ve_j^{\t}\XIb(\loptvb{i}-\lopvb{i})+(b_i-\beta_{0,i})x_{ij}|^{\alpha}\right)
		\\
	&\stackrel{\text{(iii)}}{\leq}
	\op{\frac{\plogn\cdot \bc^{1+\alpha}}{n^{\frac{\alpha^2}{2}}\cdot \ac^{32\rho+7}}}
	,
\end{split}	
\EEQ
where Inequality (i) is due to Assumption O.\ref{ass:Holder}, Inequality (ii) is due to \eqref{eq:add_lopo_eq5}, \eqref{eq:add_eq_10.40pm} and Lemma \ref{lem:xwnorm}, and inequality (iii) is due to \eqref{eq:suptovbijeq22}, \eqref{eq:suptovbijeq32}, \eqref{eq:add_eq_10.40pm} and Lemma \ref{lem:xwnorm}. Hence, according to Lemma \ref{lem:xwnorm}, we have
\[
	\sup_i\left|\vcx{i}^{\t}\diag{\dfl''(\lopvy{i}-\XIb\lopvb{i})-\dfl''(\vy-\vX\hvb)}\vcx{i}\right|
	\spaceequal
	\op{\frac{\plogn\cdot \bc^{1+\alpha}}{n^{\frac{\alpha^2}{2}}\cdot \ac^{32\rho+7}}}
	.
\]
To show the second equation in \eqref{eq:add_q1_eq1}, by Theorem \ref{thm:lopo} and Assumption O.\ref{ass:Holder}, we have
\BEQ
	\sup_{i,j}|\creg''(\lophb{i}_j)-\creg''(\lopb{i}_j)|
	\qleq
	C_r\sup_{i,j}|\lophb{i}_j-\lopb{i}_j|^{\alpha}
	\qleq
	\op{\frac{\plogn\cdot \bc^{1+\alpha}}{n^{\frac{\alpha^2}{2}}\cdot \ac^{32\rho+7}}}
	.	\label{eq:difl2}
\EEQ
Based on \eqref{eq:difl} and \eqref{eq:difl2}, we replace $\ve_j$ by $\XIb^{\t}\diag{\dfl''(\vy-\vX\hvb)}\vcx{i}$ in \eqref{eq:replaceei1}, \eqref{eq:replaceei2} and follow similar steps as the ones presented in the proof of Lemma \ref{lem:ABswitch} to obtain
\BE
	\lefteqn{\sup_i|\vcx{i}^{\t}(\hat{\vM}_i-\check{\vM}_i)\vcx{i}|}
		\n\\
	&\leq&
	\sup_{i}\left\|\XIb^{\t}\diag{\dfl''(\vy-\vX\hvb)}\vcx{i}\right\|^2\cdot \op{\frac{\plogn\cdot \bc^{1+\alpha}}{n^{\frac{\alpha}{2}}\cdot \ac^{33\rho+9}}}
		\n\\
	&\leq&
	\op{\frac{\plogn\cdot \bc^{1+\alpha}}{n^{\frac{\alpha}{2}}\cdot \ac^{38\rho+9}}}
	,	\n
\EE
where the last inequality is due to Assumption O.\ref{ass:Smoothness2}, Lemma \ref{lem:minev} and Lemma \ref{lem:xwnorm}.  To obtain the last equation in \eqref{eq:add_q1_eq1}, note that $\check{\vM}_i$ and $\bar{\vM}_i$ have the following forms:
\[
	\check{\vM}_i
	\spaceequal
	\diag{\dhfl''}\vW_i\diag{\dhfl''} \qand \bar{\vM}_i\spaceequal (\diag{\dhfl''}+\cDelta_{i})\vW_i(\diag{\dhfl''}+\cDelta_{i})
	,
\]
where $\dhfl''$ is a shorthand for $\dfl''(\vy-\vX\hvb)$ and $\cDelta_{i}, \vW_i$ are defined in the following way: 
\begin{eqnarray}
	\cDelta_{i}
	&\spaceequal&
	\diag{\dfl''(\lopvy{i}-\XIb\lopvb{i})-\dfl''(\vy-\vX\hvb)}, \nonumber \\
	\vW_i
	&\spaceequal&
	\XIb\left(\XIb^{\t}\diag{\dfl''(\lopvy{i}-\XIb\lopvb{i})}\XIb+\lambda\diag{\cdreg''(\lopvb{i})}\right)^{-1}\XIb^{\t}
	.	\n
\end{eqnarray}
Hence, we have
\BE
	\lefteqn{\sup_i|\vcx{i}^{\t}(\bar{\vM}_i-\check{\vM}_i)\vcx{i}|}
		\n\\
	&\leq&
	\sup_i|\vcx{i}^{\t}\cDelta_i\vW_i\cDelta_i\vcx{i}|+2\sup_i|\vcx{i}^{\t}\diag{\dhfl''}\vW_i\cDelta_i\vcx{i}|
		\n\\
	&\leq&
	\sup_{\|\vu\|=1}\vu^{\t}\vW_i\vu\cdot \sup_i\|\cDelta_i\vcx{i}\|^2+2\sup_i\|\vW_i\diag{\dhfl''}\vcx{i}\|\cdot \sup_i\|\cDelta_i\vcx{i}\|
	.	\n
\EE
By \eqref{eq:difl}, we have
\[
	\sup_{i,j}|\ve_j^{\t}\cDelta_i\ve_j|
	\spaceequal
	\sup_{i,j}|\fl''(y_j-\vrx{j}^{\t}\hvb)-\fl''(\lopy{i}_j-\ve_j^{\t}\XIb\lopvb{i})|
	\qleq
	\op{\frac{\plogn\cdot \bc^{1+\alpha}}{n^{\frac{\alpha^2}{2}}\cdot \ac^{32\rho+7}}}
	.
\]
Due to Assumption O.\ref{ass:Smoothness} and Lemma \ref{lem:minev}, the maximum eigenvalue of $\vW_i$ is at most $\op{1/\ac}$. Hence, with Lemma \ref{lem:xwnorm} and the fact that $\cDelta_i$s are diagonal matrices, we have
\BE
	\lefteqn{\sup_i|\vcx{i}^{\t}(\bar{\vM}_i-\check{\vM}_i)\vcx{i}|}
		\n\\
	&\leq&
	\op{\frac{1}{\ac}}\cdot\op{\frac{\plogn\cdot \bc^{2+2\alpha}}{n^{\alpha^2}\cdot \ac^{64\rho+14}}}\cdot \op{1}
		\n\\
	&&
	+\op{\frac{1}{\ac}}\sup_{i}\fl''(y_i-\vrx{i}^{\t}\hvb)\cdot \op{1}\cdot \op{\frac{\plogn\cdot \bc^{1+\alpha}}{n^{\frac{\alpha^2}{2}}\cdot \ac^{32\rho+7}}}
		\n\\
	&=&
	\op{\frac{\plogn\cdot \bc^{1+\alpha}}{n^{\frac{\alpha^2}{2}}\cdot \ac^{64\rho+15}}}
	,	\n
\EE
where the last equality is due to Assumption O.\ref{ass:Smoothness2} and Lemma \ref{lem:supnormDi}. Hence, we have completed the proof of this lemma.


\subsubsection{Proof of Lemma \ref{lem:q2}}\label{sec:q2}

Note that by replacing $\vcx{i}$ by $\ve_j$ in the proof of Lemma \ref{lem:q1}, we can follow the same steps and show that
\[
	\sup_i\left|\frac{1}{n}\trace\left((\dbfl^{i})^{\frac{1}{2}}\bQ_i^{-1}(\dbfl^{i})^{\frac{1}{2}}\right)-\frac{1}{n} \trace\left( \diag{\dhfl''}^{\frac{1}{2}}\hQ_i^{-1}\diag{\dhfl''}^{\frac{1}{2}}\right)\right|
\]
is at most
\[
	\op{\frac{\plogn\cdot \bc^{1+\alpha}}{n^{\frac{\alpha^2}{2}}\cdot \ac^{64\rho+15}}}
	.
\]
Hence, since
\BE
	&&\sup_i\left|\frac{1}{n}\trace\left(\diag{\dhfl''}^{\frac{1}{2}}\hQ_i^{-1}\diag{\dhfl''}^{\frac{1}{2}}\right)-\frac{1}{n} \trace\left(\diag{\dhfl''}^{\frac{1}{2}}(\vI+\hX \vD\hX^{\t})^{-1}\diag{\hfl''}^{\frac{1}{2}}\right)\right|
		\n\\
	&=&
	\sup_i\left|\frac{1}{n}\trace\left(\diag{\dhfl''}^{\frac{1}{2}}\left(\hQ_i^{-1}-(\hQ_i+D_i\hvcx{i}\hvcx{i}^{\t})^{-1}\right)\diag{\dhfl''}^{\frac{1}{2}}\right)\right|
	,	\n
\EE
we just need to show that
\BEQ
\label{eq:add_q2_eq1}
\BS
	\sup_i\left|\frac{1}{n}\trace\left(\diag{\dhfl''}^{\frac{1}{2}}\left(\hQ_i^{-1}-(\hQ_i+D_i\hvcx{i}\hvcx{i}^{\t})^{-1}\right)\diag{\dhfl''}^{\frac{1}{2}}\right)\right|
	\leq \op{\frac{\plogn}{n^{\alpha^2/2}}}. 
\end{split}		
\EEQ
By Matrix Inversion Lemma, we have
\BE
	\lefteqn{\sup_i\left|\frac{1}{n}\trace\left(\diag{\dhfl''}^{\frac{1}{2}}\left(\hQ_i^{-1}-(\hQ_i+D_i\hvcx{i}\hvcx{i}^{\t})^{-1}\right)\diag{\dhfl''}^{\frac{1}{2}}\right)\right|}
		\n\\
	&=&
	\sup_i\left|\frac{1}{n}\trace\left(\frac{D_i\hQ_i^{-1}\hvcx{i}\hvcx{i}^{\t}\hQ_i^{-1}}{1+D_i\hvcx{i}^{\t}\hQ_i^{-1}\hvcx{i}}\diag{\dhfl''}\right)\right|
		\n\\
	&=&
	\sup_i\left|\frac{1}{n}\trace\left(\frac{\hQ_i^{-1}\hvcx{i}\hvcx{i}^{\t}\hQ_i^{-1}}{\lambda \creg''(\hb_i)+\hvcx{i}^{\t}\hQ_i^{-1}\hvcx{i}}\diag{\dhfl''}\right)\right|
		\n\\
	&=&
	\sup_i\left|\frac{1}{n}\cdot \frac{\trace\left(\hQ_i^{-1}\hvcx{i}\hvcx{i}^{\t}\hQ_i^{-1}\diag{\dhfl''}\right)}{\lambda \creg''(\hb_i)+\hvcx{i}^{\t}\hQ_i^{-1}\hvcx{i}}\right|
		\n\\
	&\leq&
	\sup_i\left|\frac{1}{n}\cdot \frac{\hvcx{i}^{\t}\hQ_i^{-2}\hvcx{i}}{\lambda \creg''(\hb_i)+\hvcx{i}^{\t}\hQ_i^{-1}\hvcx{i}}\right|\cdot \sup_i\fl''(y_i-\vrx{i}^{\t}\hvb)
	.	\n
\EE
Hence, due to the definition of $\hQ_i$, we know that the minimal eigenvalue of $\hQ_i$ is at least $1$ and $\hQ_i$ is a semi-positive definite. Therefore, we have 
\BE
	\lefteqn{\sup_i\left|\frac{1}{n}\trace\left(\diag{\dhfl''}^{\frac{1}{2}}\left(\hQ_i^{-1}-(\hQ_i+D_i\hvcx{i}\hvcx{i}^{\t})^{-1}\right)\diag{\dhfl''}^{\frac{1}{2}}\right)\right|}
		\n\\
	&\leq &
	\frac{\sup_i\left|\fl''(y_i-\vrx{i}^{\t}\hvb)\right|\cdot \hvcx{i}^{\t}\hQ_i^{-1/2}\hQ_i^{-1}\hQ_i^{-1/2}\hvcx{i} }{n\hvcx{i}^{\t}\hQ_i^{-1}\hvcx{i}} 
	\qleq
	\frac{\sup_i\left|\fl''(y_i-\vrx{i}^{\t}\hvb)\right|\cdot \hvcx{i}^{\t}\hQ_i^{-1}\hvcx{i} \cdot \op{1}}{n\hvcx{i}^{\t}\hQ_i^{-1}\hvcx{i}} 
	\n\\
	&\leq&
	\op{\frac{\sup_i\left|\fl''(y_i-\vrx{i}^{\t}\hvb)\right|}{n}}
	.	\n
\EE
Note that, due to Lemma \ref{lem:supnormDi} and Assumption O.\ref{ass:Smoothness2}, we have 
\[
	\sup_i\fl''(y_i-\vrx{i}^{\t}\hvb)
	\qleq
	\op{\frac{\plogn}{\ac^{3\rho}}}
	.
\]
Hence,  \eqref{eq:add_q2_eq1} holds. 

%% file: suppvarlo.tex
\section{Consistency of $\text{LOOCV}$ estimate}

\subsection{Proof of Lemma \ref{lem:crossterm}}\label{sec:crossterm}
Note that
\BE
	\text{var}(P_1)
	&=&
	\frac{1}{n^2}\sum_{i=1}^n\bbE\left(\fl(y_i-\vrx{i}^{\t}\loovb{i})-\bbE[\fl(y_i-\vrx{i}^{\t}\loovb{i})\big|\mathcal{D}_i]\right)^2
		\n\\
	&&
	+\frac{1}{n^2}\sum_{i\neq j}\bbE\left(\fl(y_i-\vrx{i}^{\t}\loovb{i})-\bbE[\fl(y_i-\vrx{i}^{\t}\loovb{i})\big|\mathcal{D}_i]\right)
		\n\\
	&&
	\quad \times \left(\fl(y_j-\vrx{j}^{\t}\loovb{j})-\bbE[\fl(y_j-\vrx{j}^{\t}\loovb{j})\big|\mathcal{D}_j]\right)
	.	\n
\EE
Our goal is to bound $\text{var}(P_1)$ by $O\left(\frac{1}{n\ac^{8\rho+10}}\right)$. We have
\BEQ
\label{eq:Con_eq_2}
\BS
	\lefteqn{\text{var}(P_1)
	\leq
	\underbrace{\frac{1}{n}\bbE\fl(y_1-\vrx{1}^{\t}\loovb{1})^2}_{\text{part 3}}}
		\\
	&+\underbrace{\bbE\left(\fl(y_1-\vrx{1}^{\t}\loovb{1})-\bbE[\fl(y_1-\vrx{1}^{\t}\loovb{1})\big|\mathcal{D}_1]\right)\left(\fl(y_2-\vrx{2}^{\t}\loovb{2})-\bbE[\fl(y_2-\vrx{2}^{\t}\loovb{2})\big|\mathcal{D}_2]\right)}_{\text{part 4}}
	.
\end{split}			
\EEQ
To bound part 3, note that, according to Assumptions O.\ref{ass:Smoothness2} and O.\ref{ass:True}, we have
\BE
	\bbE\left(\fl(y_1-\vrx{1}^{\t}\loovb{1})\right)^2
	&\leq&
	O(1)\left(1+\bbE\left(y_1-\vrx{1}^{\t}\loovb{1}\right)^{2\rho+4}\right)
		\n\\
	&\leq&
	O(1)\left(1+\bbE\frac{\|\loovb{1}-\vbeta_0\|^{4\rho+8}}{n^{2\rho+4}}+\bbE w_1^{4\rho+8}\right)
		\n\\
	&=&
	O\left(1+\frac{1}{n^{2\rho+4}}\bbE\|\loovb{1}-\vbeta_0\|^{4\rho+8}\right)
	.	\n
\EE
Hence, to bound part 3 by $O\left(\frac{1}{n\ac^{4\rho+8}}\right)$, we will first show that $\bbE\|\hvb-\vbeta_0\|^{2r}$ is bounded by $O(n^{r}/\ac^{2r})$ for every integer number $r$. Bounding $\bbE\|\loovb{1}-\vbeta_0\|^{2r}$ will be similar. Note that $\hvb$ should satisfy the following
\BE
	0
	&=&
	-\vX^{\t}\fl'(\vy-\vX\hvb)+\lambda \reg'(\hvb)
	.	\n
\EE
Hence, by applying Taylor expansion for $\fl'$ at $\vw$ and $\reg'$ at $\vbeta_0$, we have
\BE
	\vX^{\t}\fl'(\vw)-\lambda\reg'(\vbeta_0)
	&=&
	\left(\vX^{\t}\diag{\dfl''(\vw_{\vxi})}\vX+\lambda\diag{\dreg''(\vbeta_{\vxi'})}\right)(\hvb-\vbeta_0)
	,	\n
\EE
where $\ve_j^{\t}\vbeta_{\vxi'}=\xi'_j\ve_j^{\t} \vbeta_0+(1-\xi'_j)\ve_j^{\t} \hvb$ for some $\xi'_j \in [0,1]$ and $j$th diagonal component of $\diag{\dfl''(\vw_{\vxi})}$ is $\fl''(\xi_{j}w_j+(1-\xi_{j}) \vrx{j}^{\t}(\vbeta_0-\hvb))$ for some $\xi_{j}\in [0,1]$. Then, by Matrix Inversion Lemma and Assumption O.\ref{ass:Smoothness}, it is straightforward to show that
\BEQ
\label{eq:Con_eq_3}	
\BS
	\|\hvb-\vbeta_0\|
	&=
	\left\|\left(\vX^{\t}\diag{\dfl''(\vw_{\vxi})}\vX+\lambda\diag{\dreg''(\vbeta_{\vxi'})}\right)^{-1}(\vX^{\t}\fl'(\vw)-\lambda\reg'(\vbeta_0))\right\|
		\\
	&\leq
	\left\|\left(\vX^{\t}\diag{\dfl''(\vw_{\vxi})}\vX\right)^{-1}(\vX^{\t}\fl'(\vw)-\lambda\reg'(\vbeta_0))\right\|
		\\
	&\leq
	\ac^{-1}\left\|\left(\vX^{\t}\vX\right)^{-1}(\vX^{\t}\fl'(\vw)-\lambda\reg'(\vbeta_0))\right\|
		\\
	&\leq
	\ac^{-1}\left\|\left(\vX^{\t}\vX\right)^{-1}\vX^{\t}\fl'(\vw)\right\|+\frac{\lambda}{\ac}\left\|\left(\vX^{\t}\vX\right)^{-1}\reg'(\vbeta_0)\right\|
	.
\end{split}		
\EEQ
Hence, we have 
\BEQ
	\bbE \left(\frac{\|\hvb-\vbeta_0\|^2}{n}\right)^r
	\qleq
	O\left(\frac{1}{\ac^{2r}}\right)\cdot \left(\underbrace{\bbE\frac{\left\|\left(\vX^{\t}\vX\right)^{-1}\vX^{\t}\fl'(\vw)\right\|^{2r}}{n^r}}_{\text{part 5}}+ \underbrace{\bbE\frac{\left\|\left(\vX^{\t}\vX\right)^{-1}\reg'(\vbeta_0)\right\|^{2r}}{n^r}}_{\text{part 6}}\right)
	. 	\label{eq:Con_eq_4}
\EEQ
For part 5, we have
\BE
	\bbE\frac{\left\|\left(\vX^{\t}\vX\right)^{-1}\vX^{\t}\fl'(\vw)\right\|^{2r}}{n^r}
	&\leq&
	\bbE\left(\frac{\sigma_{\max}\left(\vX(\vX^{\t}\vX)^{-2}\vX^{\t}\right)\|\fl'(\vw)\|^2}{n}\right)^r
		\n\\
	&=&
	\bbE\left(\frac{\sigma_{\max}\left(\left(\vX^{\t}\vX\right)^{-1}\right)\|\fl'(\vw)\|^2}{n}\right)^r
		\n\\
	&\stackrel{\text{(i)}}{=}&
	\bbE\frac{1}{\sigma_{\min}^r(\vX^{\t}\vX)}\bbE\left(\frac{\sum_{i=1}^n\fl'(w_i)^2}{n}\right)^r
		\n\\
	&\leq&
	\bbE\frac{1}{\sigma_{\min}^r(\vX^{\t}\vX)}\bbE \fl'(w)^{2r}
		\n\\
	&\leq&
	\bbE\frac{1}{\sigma_{\min}^r(\vX^{\t}\vX)}\bbE (1+|w|^{\rho+1})^{2r}\cdot O(1)
		\n\\
	&\leq&
	\bbE\frac{1}{\sigma_{\min}^r(\vX^{\t}\vX)}\cdot O(1)
			,\n
\EE
where $\sigma_{\max}(\bm{M})$ and $\sigma_{\min}(\bm{M})$ denote the maximum and minimum eigenvalues of matrix $\bm{M}$ respectively. Furthermore, Equality (i) holds since $\vw$ and $\vX$ are independent. Finally, the  last two inequalities are due to Assumptions O.\ref{ass:Smoothness2} and O.\ref{ass:True}.

To bound $\bbE\frac{1}{\sigma_{\min}^r(\vX^{\t}\vX)}$, we claim the following lemma:
\begin{lemma}\label{lem:socomp}
For all fixed $r\geq 0$, we have 
\BE
	\bbE\frac{1}{\sigma_{\min}^r(\vX^{\t}\vX)}
	\spaceequal
	O(1)
	.	\label{eq:socomp_eq1}
\EE
\end{lemma}
\noindent The proof of this lemma can be found in Section \ref{sec:socomp}.
Hence, with Assumption O.\ref{ass:True}, we have
\[
	\bbE\frac{\left\|\left(\vX^{\t}\vX\right)^{-1}\vX^{\t}\fl'(\vw)\right\|^{2r}}{n^r}
	\qleq
	O(1)
	.
\]
Similarly, part 6 is $O(1)$ as well. Therefore, according to \eqref{eq:Con_eq_4}, for all $r\in \bbN$ we have
\BE
	\bbE \left(\frac{\|\hvb-\vbeta_0\|^2}{n}\right)^r
	\qleq 
	O\left(\frac{1}{\ac^{2r}}\right)
	.	\label{eq:Con_eq_5}
\EE
Similarly, for all $r\in \bbN$ we have
\BE
	\bbE \left(\frac{\|\loovb{1}-\vbeta_0\|^2}{n}\right)^r
	\qleq 
	O\left(\frac{1}{\ac^{2r}}\right)
	.	\label{eq:Con_eq_6}
\EE
Hence, part 3 in \eqref{eq:Con_eq_2} is bounded by $O\left(\frac{1}{n\ac^{4\rho+8}}\right)$. To bound part 4 in \eqref{eq:Con_eq_2}, consider the following definitions:
\BE
	\delta_1
	&=&
	\fl(y_1-\vrx{1}^{\t}\loovb{1})-\bbE[\fl(y_1-\vrx{1}^{\t}\loovb{1})\big|\mathcal{D}_1]
	,	\n\\
	\delta_2
	&=&
	\fl(y_2-\vrx{2}^{\t}\loovb{2})-\bbE[\fl(y_2-\vrx{2}^{\t}\loovb{2})\big|\mathcal{D}_2]
	,	\n\\
	\delta_{12}
	&=&
	\fl(y_1-\vrx{1}^{\t}\loovb{\{1,2\}})-\bbE[\fl(y_1-\vrx{1}^{\t}\loovb{\{1,2\}})\big|\mathcal{D}_{1}]
	,	\n\\
	\delta_{21}
	&=&
	\fl(y_2-\vrx{2}^{\t}\loovb{\{1,2\}})-\bbE[\fl(y_2-\vrx{2}^{\t}\loovb{\{1,2\}})\big|\mathcal{D}_2]
	,	\n	
\EE
where $\loovb{\{1,2\}}$ is the minimizer of \eqref{eq:model} without the first and second observations $(\vrx{1},y_1)$ and $(\vrx{2},y_2)$, i.e.,
\[
	\loovb{\{1,2\}}
	\spaceequal
	\argmin_{\vbeta \in \bbR^{p}}\sum_{i=3}^n \fl(y_i-\vrx{i}^{\t}\vbeta)+\lambda \sum_{i=1}^p \reg(\beta_i)
	.
\]
Since $\loovb{\{1,2\}}$ is independent of both the first and second observations, it is straightforward to show that
\BE
	0
	&=&
	\bbE \delta_1\delta_{21}
	\spaceequal
	\bbE \delta_{12}\delta_2
	\spaceequal
	\bbE \delta_{12}\delta_{21}
	.	\n
\EE
Hence, we have 
\BE
	\lefteqn{\text{part 4}
	\spaceequal
	\bbE \delta_1\delta_2
	\spaceequal
	\bbE (\delta_1-\delta_{12})(\delta_2-\delta_{21})}
		\n\\
	&=&
	\text{part 7}+\text{part 8}+\text{part 9}+\text{part 10}
	,	\n
\EE
where
\BE
	\text{part 7}
	&=&
	\bbE \left(\fl(y_1-\vrx{1}^{\t}\loovb{1})-\fl(y_1-\vrx{1}^{\t}\loovb{\{1,2\}})\right)
		\n\\
	&&
	\times \left(\fl(y_2-\vrx{2}^{\t}\loovb{2})-\fl(y_2-\vrx{2}^{\t}\loovb{\{1,2\}})\right)
		\n\\
	\text{part 8}
	&=&
	\bbE \left(\fl(y_1^{\text{new}}-(\vrx{1}^{\text{new}})^{\t}\loovb{1})-\fl(y_1^{\text{new}}-(\vrx{1}^{\text{new}})^{\t}\loovb{\{1,2\}})\right)
		\n\\
	&&
	\times \left(\fl(y_2-\vrx{2}^{\t}\loovb{2})-\fl(y_2-\vrx{2}^{\t}\loovb{\{1,2\}})\right)
		\n\\
	\text{part 9}
	&=&
	\bbE \left(\fl(y_1-\vrx{1}^{\t}\loovb{1})-\fl(y_1-\vrx{1}^{\t}\loovb{\{1,2\}})\right)
		\n\\
	&&
	\times \left(\fl(y_2^{\text{new}}-(\vrx{2}^{\text{new}})^{\t}\loovb{2})-\fl(y_2^{\text{new}}-(\vrx{2}^{\text{new}})^{\t}\loovb{\{1,2\}})\right)
		\n\\
	\text{part 10}
	&=&
	\bbE \left(\fl(y_1^{\text{new}}-(\vrx{1}^{\text{new}})^{\t}\loovb{1})-\fl(y_1^{\text{new}}-(\vrx{1}^{\text{new}})^{\t}\loovb{\{1,2\}})\right)
		\n\\
	&&
	\times \left(\fl(y_2^{\text{new}}-(\vrx{2}^{\text{new}})^{\t}\loovb{2})-\fl(y_2^{\text{new}}-(\vrx{2}^{\text{new}})^{\t}\loovb{\{1,2\}})\right)
	,	\n
\EE
where $(\vrx{1}^{\text{new}},y_1^{\text{new}})$ and $(\vrx{2}^{\text{new}},y_2^{\text{new}})$ are two independent copies of $(\vrx{1},y_1)$ and $(\vrx{2},y_2)$. We will show that part 7 can be bounded by $O\left(\frac{1}{n\ac^{8\rho+10}}\right)$ and then part 8, 9 and 10 can be bounded by $O\left(\frac{1}{n\ac^{8\rho+10}}\right)$ following a similar argument. To bound part 7, note that $\loovb{1}, \loovb{2}$ and $\loovb{\{1,2\}}$ should satisfy the following:
\BE
	0
	&=&
	-\vX^{\t}\fl'(\vy-\vX\loovb{1})+\fl'(y_1-\vrx{1}\loovb{1})\vrx{1}+\lambda \reg'(\loovb{1})
	,	\n\\
	0
	&=&
	-\vX^{\t}\fl'(\vy-\vX\loovb{2})+\fl'(y_2-\vrx{2}\loovb{2})\vrx{2}+\lambda \reg'(\loovb{2})
	,	\n\\
	0
	&=&
	-\vX^{\t}\fl'(\vy-\vX\loovb{\{1,2\}})+\fl'(y_1-\vrx{1}\loovb{1,2})\vrx{1}+\fl'(y_2-\vrx{2}\loovb{1,2})\vrx{2}+\lambda \reg'(\loovb{\{1,2\}})
	.	\n
\EE
Hence, applying Taylor expansion, we have
\BE
	\loovb{1}-\loovb{\{1,2\}}
	&=&
	\fl'(y_2-\vrx{2}^{\t}\loovb{1,2})(\vB_{2,\vxi,\vxi'})^{-1}\vrx{2}
	,	\n\\
	\loovb{2}-\loovb{\{1,2\}}
	&=&
	\fl'(y_1-\vrx{1}^{\t}\loovb{1,2})(\vB_{1,\vxi,\vxi'})^{-1}\vrx{1}
	,	\n\\
\EE
and 
\BE
	\lefteqn{\left(\fl(y_1-\vrx{1}^{\t}\loovb{1})-\fl(y_1-\vrx{1}^{\t}\loovb{\{1,2\}})\right)\left(\fl(y_2-\vrx{2}^{\t}\loovb{2})-\fl(y_2-\vrx{2}^{\t}\loovb{\{1,2\}})\right)}
		\n\\
	&=&
	\vrx{1}^{\t}(\loovb{1}-\loovb{\{1,2\}})\fl'(y_1-\vrx{1}^{\t}\loovb{1}_{\xi})\vrx{2}^{\t}(\loovb{2}-\loovb{\{1,2\}})\fl'(y_2-\vrx{2}^{\t}\loovb{2}_{\xi})
		\n\\
	&=&
	\vrx{1}^{\t}(\vB_{1,\vxi,\vxi'})^{-1}\vrx{2}\vrx{2}^{\t}(\vB_{2,\vxi,\vxi'})^{-1}\vrx{1}\fl'(y_2-\vrx{2}^{\t}\loovb{1,2})
		\n\\
	&&
	\times\fl'(y_1-\vrx{1}^{\t}\loovb{1}_{\xi})\fl'(y_1-\vrx{1}^{\t}\loovb{1,2})\fl'(y_2-\vrx{2}^{\t}\loovb{2}_{\xi})
	,	\n
\EE
where $\vB_{1,\vxi,\vxi'},\vB_{2,\vxi,\vxi'}$ are defined by
\BE
	\vB_{1,\vxi,\vxi'}
	&=&
	\vX_{\backslash\{1,2\}}^{\t}\diag{\dfl''(\vy-\vX_{\backslash \{1,2\}}\loovb{1}_{\xi})}\vX_{\backslash \{1,2\}}+\lambda\diag{\dreg''(\loovb{1}_{\xi})}
	,	\n\\
	\vB_{2,\vxi,\vxi'}
	&=&
	\vX_{\backslash\{1,2\}}^{\t}\diag{\dfl''(\vy-\vX_{\backslash \{1,2\}}\loovb{2}_{\xi'})}\vX_{\backslash \{1,2\}}+\lambda\diag{\dreg''(\loovb{2}_{\xi'})}
	,	\n
\EE
and $\loovb{1}_{\xi},\loovb{2}_{\xi'}$ lie between $\loovb{1}$ and $\loovb{\{1,2\}}$ or $\loovb{2}$ and $\loovb{\{1,2\}}$ respectively. Then, by the Cauchy inequality, we have 
\BE
	(\text{part 7})^4
	&\leq&
	\bbE(\vrx{1}^{\t}(\vB_{1,\vxi,\vxi'})^{-1}\vrx{2})^4\times \bbE(\vrx{2}^{\t}(\vB_{2,\vxi,\vxi'})^{-1}\vrx{1})^4
		\n\\
	&&\times \left(\bbE\left(\fl'(y_2-\vrx{2}^{\t}\loovb{1,2})\fl'(y_1-\vrx{1}^{\t}\loovb{1}_{\xi})\fl'(y_1-\vrx{1}^{\t}\loovb{1,2})\fl'(y_2-\vrx{2}^{\t}\loovb{2}_{\xi})\right)^2\right)^2
		\n\\
	&\leq&
	O\left(\frac{1}{\ac^{32\rho+32}}\right)\cdot \bbE(\vrx{1}^{\t}(\vB_{1,\vxi,\vxi'})^{-1}\vrx{2})^4\times \bbE(\vrx{2}^{\t}(\vB_{2,\vxi,\vxi'})^{-1}\vrx{1})^4
	,	\n
\EE   
where the proof of the last inequality is similar to the proof we gave for bounding part 3 above. To bound $\bbE(\vrx{1}^{\t}(\vB_{1,\vxi,\vxi'})^{-1}\vrx{2})^4$, note that $\vrx{1}$ is independent of $(\vB_{1,\vxi,\vxi'})^{-1}\vrx{2}$ and therefore, we have
\[
	\vrx{1}^{\t}(\vB_{1,\vxi,\vxi'})^{-1}\vrx{2}\big| \mathcal{D}_1
	\ide
	\mathcal{N}(0,\frac{\|(\vB_{1,\vxi,\vxi'})^{-1}\vrx{2}\|^2}{n})
	.
\]
Further, by Matrix Inversion Lemma, we can bound $\|(\vB_{1,\vxi,\vxi'})^{-1}\vrx{2}\|$ by $\frac{1}{\ac}\|(\vX_{\backslash\{1,2\}}^{\t}\vX_{\backslash\{1,2\}})^{-1}\vrx{2}\|$. Hence, we have
\BE
	\bbE(\vrx{1}^{\t}(\vB_{1,\vxi,\vxi'})^{-1}\vrx{2})^4
	&=&
	3\bbE \frac{\|(\vB_{1,\vxi,\vxi'})^{-1}\vrx{2}\|^4}{n^2}
		\n\\
	&\leq&
	O\left(\frac{1}{n^2\ac^4}\right)\cdot \bbE\|(\vX_{\backslash\{1,2\}}^{\t}\vX_{\backslash\{1,2\}})^{-1}\vrx{2}\|^4
		\n\\
	&\leq&
	O\left(\frac{1}{n^2\ac^4}\right)\cdot \bbE ~3\cdot \trace\left(\frac{(\vX_{\backslash\{1,2\}}^{\t}\vX_{\backslash\{1,2\}})^{-4} }{n^2}\right)
	,	\n
\EE
where the last equality is due to the fact that $\vrx{2}$ is a Gaussian vector and is independent of $\vX_{\backslash\{1,2\}}$. Due to \cite{graczyk2003complex}[Theorem 4], we have
\[
	\bbE(\vrx{1}^{\t}(\vB_{1,\vxi,\vxi'})^{-1}\vrx{2})^4
	\qleq
	O\left(\frac{1}{n^2\ac^4}\right)
	.
\]
Similarly, we have
\[
	\bbE(\vrx{2}^{\t}(\vB_{2,\vxi,\vxi'})^{-1}\vrx{1})^4
	\qleq
	O\left(\frac{1}{n^2\ac^4}\right)
	.
\]
Hence, we have proved that part 7 is $O\left(\frac{1}{n\ac^{8\rho+10}}\right)$. By using similar techniques we can prove that parts 8 to 10 are bounded by $O\left(\frac{1}{n\ac^{8\rho+10}}\right)$. Therefore, part 4 is bounded by $O\left(\frac{1}{n\ac^{8\rho+10}}\right)$. Together with the fact that part 3 is bounded by $O\left(\frac{1}{n\ac^{4\rho+8}}\right)$, we have shown that the variance of part 1 is bounded by $O\left(\frac{1}{n\ac^{8\rho+10}}\right)$.

\subsection{Proof of Lemma \ref{lem:socomp}}\label{sec:socomp}

Let $\bm{W}$ denote a random matrix drawn from the standard Wishart distribution $W_{p}(n,\vI)$. Let $\lambda_1\geq \lambda_2\ldots\geq\lambda_p$ denote the eigenvalues of $\bm{W}$. Then, it is straightforward to see that bounding $\bbE\frac{1}{\sigma_{\min}^r(\vX^{\t}\vX)}$ is equivalent to bounding $\bbE\frac{n^r}{\lambda_p^r}$. According to \cite{chen2005condition}[Lemma 3.3], offers the following upper bound for the probability density function of $\lambda_p$:
\BE
	f_p(\lambda_p)
	\qleq
	\kappa_{n,p}e^{-\frac{1}{2}\lambda_p}\lambda_p^{(n-p-1)/2}
	,	\label{eq:induction_eq_1}
\EE
where $\kappa_{n,p}=\frac{2^{(n-p-1)/2}\Gamma((n+1)/2)}{\Gamma(p/2)\Gamma(n-p+1)}$ and $\Gamma(x)$ is the Gamma function. Hence, as long as $n$ and $p$ are sufficiently large, we know $\bbE\frac{n^r}{\lambda_p^r}$ exists. Further, we just need to consider $r$ to be integer since we have
\[
	\bbE\left(\frac{n}{\lambda_p}\right)^r 
	\qleq
	\bbE\left(\frac{n}{\lambda_p}\right)^{\lfloor r \rfloor}+\bbE\left(\frac{n}{\lambda_p}\right)^{\lceil r \rceil}
\]
Next, let us denote $c_{\delta}=\left(\frac{\delta}{\delta-1}\right)^2$ where $\delta=\frac{n}{p}$ is defined in Section \ref{sec:obj}. Then by \eqref{eq:induction_eq_1}, we have 
\BE
	\bbE \left(\frac{n}{\lambda_p}\right)^r 
	&=&
	\bbE\left(\frac{n}{\lambda_p}\right)^r \bbI_{\lambda_p\geq n/c_{\delta}}+\bbE\left(\frac{n}{\lambda_p}\right)^r  \bbI_{\lambda_p\leq n/c_{\delta}}
		\n\\
	&\leq& c_{\delta}^r+\int_{0}^{\frac{n}{c_{\delta}}} \left(\frac{n}{\lambda_p}\right)^r \cdot \kappa_{n,p}e^{-\frac{1}{2}\lambda_p}\lambda_p^{(n-p-1)/2} \dif \lambda_p
		\n\\
	&\leq& c_{\delta}^r+n^r\int_{0}^{\frac{n}{c_{\delta}}}  \kappa_{n,p}\lambda_p^{(n-p-1-2r)/2} \dif \lambda_p
		\n\\
	&=&
	c_{\delta}^r +\frac{\Gamma(\frac{n+1}{2})}{\Gamma(\frac{p}{2})\left(\frac{n}{2}\right)^{\frac{n-p+1}{2}}}\cdot \frac{\left(\frac{n}{\sqrt{c_{\delta}}}\right)^{n-p+1}\cdot c_{\delta}^r}{\Gamma(n-p+1)\cdot (n-p+1-2r)}
	.	\n
\EE
From \cite{chen2005condition}[Eq 2.6 and Lemma 4.1], we have
\[
	\frac{\Gamma(\frac{n+1}{2})}{\Gamma(\frac{p}{2})\left(\frac{n}{2}\right)^{\frac{n-p+1}{2}}}
	\qleq
	1
	.
\] 
Hence, with Stirling's approximation, we have
\BE
	\bbE \left(\frac{n}{\lambda_p}\right)^r
	&\leq&
	c_{\delta}^r+\frac{\left(\frac{n}{\sqrt{c_{\delta}}}\right)^{n-p+1}\cdot c_{\delta}^r}{\Gamma(n-p+1)\cdot (n-p+1-2r)}
		\n\\
	&\leq&
	O(c_{\delta}^r)\cdot \left(1+\frac{\left(\frac{n}{\sqrt{c_{\delta}}}\right)^{n-p+1}}{\left(\frac{n-p}{e}\right)^{n-p}\cdot (n-p+1-2r)}\right)
		\n\\
	&\leq&
	O(c_{\delta}^r)\cdot \left(1+\left(\frac{\delta}{(\delta-1)\sqrt{c_{\delta}}}\right)^{n-p+1}\right)
		\n\\
	&=&
	O(1)
	,	\n
\EE
where to obtain the last equality we plugged in the value of $c_{\delta}$. This completes the proof.

%% file: sec-appendix.tex
\section{Proof of Lemma \ref{lem:ABOUTASSUMPO5}}\label{sec:disbi}

Our proof here uses the proof of Theorem \ref{thm:lopo}. Hence, we suggest that the reader reads the proof of Theorem \ref{thm:lopo} before this. We first remind the quantity $b_i$ was defined in the statement of Theorem \ref{thm:lopo}. 
Our goal here is to first prove that under the assumptions O.\ref{ass:Convex} - O.\ref{ass:True}, we have $\sup_i|b_i|= \op{\frac{\plogn}{\ac^{10\rho+5}}}$. We will connect $\hb_i$ with $b_i$ and bound $\sup_i|\hb_i|$ later in the proof.\\

To show $\sup_{i}|b_i|$ is bounded by $\op{\plogn}$, note that if condition (a) holds, then we have
\BE
	\sup_{i,j}\reg''(\lopb{i}_j)
	&\leq&
	\sup_{i,j}\reg''(\lophb{i}_j)+O(1)\cdot \sup_{i,j}|\lophb{i}_j-\lopb{i}_j|
		\n\\
	&=&
	\sup_{j}\reg''(\hb_j)+O(1)\cdot \sup_{i,j}|\lophb{i}_j-\lopb{i}_j|
		\n\\
	&\leq&
	\op{\frac{\plogn}{\ac^{6\rho+3}}}
	,	\n
\EE
where the last equality is due to \eqref{eq:add_goal_1}. Note that, in the proof of Theorem \ref{thm:lopo}, the only place we use assumption O.\ref{ass:boundhvb} is to obtain an upper bound on $\sup_{i,j}\reg''(\lopb{i}_j)$ in \eqref{eq:add_10.49}. A similar argument shows that Condition (b) proves $\sup_{i,j}\reg''(\lopb{i}_j) = \op{\frac{\plogn}{\ac^{6\rho+3}}}$ as well. Hence, applying this new bound in \eqref{eq:add_10.48}, we have
\[
	\inf_i|a_i|
	\qgeq
	\Omega_p\left(\frac{\ac^{6\rho+3}}{\plogn}\right)
	.
\]  
Then by \eqref{eq:add_9.57}, we have $\sup_i|b_i|\leq \op{\frac{\plogn}{\ac^{10\rho+5}}}$. For the case when condition (c) holds, note that from the definition of $b_i$ we have
\[
	a_i(b_i-\beta_{0,i})+\lambda (\reg'(b_i)-\reg'(\beta_{0,i}))
	\spaceequal
	\vcx{i}^{\t}\fl'(\lopvy{i}-\XIb\lopvb{i})-\lambda \reg'(\beta_{0,i})
	.
\]
By using the Taylor expansion, we obtain
\[
	(a_i+\lambda \reg''(b_{\xi}))(b_i-\beta_{0,i})
	\spaceequal
	\vcx{i}^{\t}\fl'(\lopvy{i}-\XIb\lopvb{i})-\lambda \reg'(\beta_{0,i})
	.
\]
where $b_{\xi}=\xi b_i+(1-\xi)\beta_{0,i}$ for some $\xi\in[0,1]$. Note that by the definition of $a_i$, we know $a_i>0$. Hence, we know $a_i+\lambda \reg''(b_{\xi})= \Omega(1)$. Therefore, we have
\BE
	\sup_{i}|b_i|
	&\leq&
	\sup_i|\beta_{0,i}|+\op{\sup_i|\vcx{i}^{\t}\fl'(\lopvy{i}-\XIb\lopvb{i})|+\lambda |\reg'(\beta_{0,i})|}
		\n\\
	&\stackrel{\text{(i)}}{\leq}&
	\op{\sup_i(1+|\beta_{0,i}|^{\rho+1})}+\op{\sup_i|\vcx{i}^{\t}\fl'(\lopvy{i}-\XIb\lopvb{i})|}
		\n\\
	&\stackrel{\text{(ii)}}{\leq}&
	\op{\plogn}+\op{\sup_i|\vcx{i}^{\t}\fl'(\lopvy{i}-\XIb\lopvb{i})|}
	,	\n
\EE
where Inequality (i) is due to assumption O.\ref{ass:Smoothness2} and Inequality (ii) is due to Lemma \ref{lem:xwnorm}. Then, note that $\vcx{i}$ is independent of $\lopvy{i}-\XIb\lopvb{i}$ and from Assumption O.\ref{ass:True}, components of $\vcx{i}$ are i.i.d.~mean 0 subGaussian random variables. Hence, by Hanson-Wright inequality, Assumption O.\ref{ass:Smoothness2} and \eqref{eq:add_lopo_eq1}, we have
\[
	\sup_{i}|b_i|
	\qleq
	\op{\frac{\plogn}{\ac^{(\rho+2)(\rho+1)}}}
	.
\]
So far we have showed that if one of the conditions (a), (b), or (c) holds, then $\sup_i|b_i| = \op{\frac{\plogn}{\ac^{10\rho+5}}}$.  The next step is to use this fact and prove that $\sup_i|\hb_i| = \op{\frac{\plogn}{\ac^{10\rho+5}}}$. Note that Corollary \ref{cor:bibound} only requires assumptions O.\ref{ass:Convex} - O.\ref{ass:True}, hence, we can apply this corollary and obtain that
\BE
	\sup_i|\hb_i|
	&\leq&
	\sup_i|\hb_i-b_i|+\sup_i|b_i|
		\n\\
	&\leq&
	\op{\frac{\plogn}{n^{\frac{\alpha}{2}}}\cdot \frac{\sup_i|b_i-\beta_{0,i}|^{1+\alpha}}{\ac^{6\rho+3}}}+\sup_i|b_i|
		\n\\
	&=&
	\op{\frac{\plogn}{\ac^{10\rho+5}}}
	,	\n
\EE
where the last equality is due to Lemma \ref{lem:xwnorm}. This shows that Assumption O.\ref{ass:boundhvb} holds with $\bc=\op{\frac{\plogn}{\ac^{15\rho}}}$.








%% file: sec-proof.tex


\section{Heuristic derivation of AMP risk estimate}\label{sec:construct}

First, we show how one can heuristically derive the risk estimate formula we presented in \eqref{eq:amprhat} when $\vrx{i}\stackrel{\text{i.i.d.}}{\sim}\mathcal{N}(\boldsymbol{0},\frac{1}{n}\vI)$. This formula is derived from the approximate message passing algorithm (AMP). AMP was first introduced as a fast iterative algorithm for solving regularized least squares problem \cite{DMM09Message}. It has since been extended to more general models and optimization problems \cite{DM16High,ma2018optimization,bradic2015robustness}. We can follow the the strategy proposed in \cite{M11PhD,donoho2010message,rangan2011generalized} and obtain the following AMP algorithm for solving \eqref{eq:model}:
\BI

\item Set initialization $\vbeta^0$ be independent of $\vX$ (usually we set $\vbeta^0=\vzero$).

\item Update $\vz^t$ and $\vbeta^{t+1}$ for $t\geq 0$ by
	\BE
		\vbeta^{t+1}
		&=&
		\eta(\vbeta^t+\vX^{\t}\frac{\psi(\vz^t,\theta^t)}{\dotp{\psi'(\vz^t,\theta^t)}},\tau_t)
		,	\n\\
		\vz^t
		&=&
		\vy-\vX\vbeta^t+\psi(\vz^{t-1},\theta^{t-1})
		,	\label{eq:amptupdaterule}
	\EE
	where $\theta^{t}$ is the solution of the following equation
	\BE
		\dotp{\psi'(\vz^{t},\theta^t)}
		&=&
		\frac{1}{\delta}\dotp{\eta'(\vbeta^{t}+\vX^{\t}\frac{\psi(\vz^t,\theta^t)}{\dotp{\psi'(\vz^t,\theta^t)}},\tau_t)}
		.	\label{eq:thetaexist}
	\EE
\EI
In these equations $\eta$ is the proximal operator of $\reg$, i.e., $\eta(x,\tau) = \argmin_{y\in \bbR}\frac{1}{2}(x-y)^2+\tau \reg(y)$, and $\psi(x,\theta):=\theta\fl'(\eta_{\fl}(x,\theta))$, where $\eta_{\fl}$ is the proximal operator of $\fl$, i.e. $\eta_{\fl}(x,\theta) = \argmin_{y\in \bbR}\frac{1}{2}(x-y)^2+\theta \fl(y).$ Furthermore,  $\psi'(x,\theta)$ denotes the derivative of $\psi(\cdot, \cdot)$ with respect to its first input argument, and $\{\tau_t\}_{t\geq 0}$ is a sequence of tuning parameter. Here, we assume that $\{\tau_t\}_{t\geq 0}$ is a converging sequence. The role of these parameters will be clarified later. We emphasize on a few features of AMP below:

\begin{itemize}

\item The existence of a solution for \eqref{eq:thetaexist} is guaranteed by \cite{DM16High}; by the convexity of the regularizer $\reg(x)$ and Lemma \ref{lem:eta'}, the right hand side (RHS) of \eqref{eq:thetaexist} is always in $[0,1/\delta]\subset [0,1]$, while the left hand side (LHS) is equal to zero for $\theta^t=0$ and is equal to one when $\theta^t=\infty$. Hence, given the continuity of the LHS and RHS functions the existence of a solution is guaranteed.  

\item  An important feature of AMP that has made its asymptotic analysis possible is that, intuitively speaking, $\vz^t$ can be considered as a random vector with Gaussian marginals. Furthermore, to calculate the mean and variance of the marginal distribution of $z_i^t$ it is safe to assume that $\vx_{i}$ is independent of $\vbeta^t$. This independence is in fact happening because of the term $\psi(\vz^{t-1},\theta^{t-1})$ that is added to the residual. This term is known as the Onsager correction term. In the calculation of the mean and variance of $z_i^t$, one can ignore the existence of this term and assume that its only is to make $\vx_{i}$ independent of $\vbeta^t$.  For further discussion regarding these heuristic arguments and the existing rigorous proofs the reader may refer to \cite{metzler2016denoising}. 

\end{itemize}
Suppose that for a converging sequence $\{\tau_t\}_{t\geq 0}$, the AMP estimates converge to $(\vbeta^{\infty}_{\tau^*}, \vz^{\infty}_{\tau^*}, \theta^{\infty}_{\tau^*},\tau^*)$. Also, define
\[
	\gamma^*
	\ :=\ 
	\tau^*\cdot \dotp{\frac{\fl''(\vy-\vX\vbeta^{\infty}_{\tau^*})}{1+\theta^{\infty}_{\tau^*} \fl''(\vy-\vX\vbeta^{\infty}_{\tau^*})}}
	.
\]
Then, $(\vbeta,\vz,\theta,\tau,\gamma)=(\vbeta^{\infty}_{\tau^*}, \vz^{\infty}_{\tau^*}, \theta^{\infty}_{\tau^*},\tau^*, \gamma^*)$ satsfies:
\begin{subequations}\label{eq:main_limitequation}
\begin{align}
	\vbeta
	&\spaceequal 
	\eta(\vbeta+\vX^{\t}\frac{\psi(\vz,\theta)}{\dotp{\psi'(\vz,\theta)}},\tau)
	,	\label{eq:main_limitequation1_1}\\
	\vz
	&\spaceequal 
	\vy-\vX\vbeta+\psi(\vz,\theta)
	,	\label{eq:main_limitequation1_2}\\
	\dotp{\psi'(\vz,\theta)}
	&\spaceequal 
	\frac{1}{\delta}\dotp{\eta'(\vbeta+\vX^{\t}\frac{\psi(\vz,\theta)}{\dotp{\psi'(\vz,\theta)}},\tau)}
	,	\label{eq:main_limitequation1_3}\\
	\gamma
	&\spaceequal
	\tau\cdot \dotp{\frac{\fl''(\vy-\vX\vbeta)}{1+\theta \fl''(\vy-\vX\vbeta)}}
	.	\label{eq:main_limitequation1_4}
\end{align}
\end{subequations}
Our next lemma helps us interpret the fixed point of AMP. 
\begin{lemma}\label{lem:AMP=GLM2}
Under Assumption O.\ref{ass:Convex}, \eqref{eq:main_limitequation} is equivalent to the following set of equations:
\begin{subequations}\label{eq:main_limitequation2}
\begin{align}
	\vzero
	&\spaceequal
	-\vX^{\t}\fl'(\vy-\vX\vbeta)+\gamma\reg'(\vbeta)
	,	\label{eq:main_limitequation2_1} \\
	\gamma
	&\spaceequal
	\dotp{\frac{\fl''(\vy-\vX\vbeta)}{\frac{1}{\tau}+\frac{1}{\delta\gamma}\dotp{\frac{1}{1+\tau\reg''(\vbeta)}}\cdot \fl''(\vy-\vX\vbeta)}}
	,	\label{eq:main_limitequation2_2} \\
	\theta
	&\spaceequal
	\frac{1}{\delta\gamma}\dotp{\frac{\tau}{1+\tau\reg''(\vbeta)}}
	,	\label{eq:main_limitequation2_3}\\
	\vz
	&\spaceequal
	\vy-\vX\vbeta+\theta \cdot \fl'(\vy-\vX\vbeta)
	,	\label{eq:main_limitequation2_4}
\end{align}
\end{subequations}
\end{lemma}
\begin{proof}
From the definition of $\psi$ we have $z-\psi(z, \theta)\equiv \eta_{\fl}(z,\theta).$ Hence, \eqref{eq:main_limitequation1_2} is equivalent to
\BE
	\eta_{\fl}\left(\vz,\theta\right)
	&=&
	\vy-\vX\vbeta
	.	\label{eq:limiteqb_1}
\EE
Next, from Lemma \ref{lem:eta'} and the definition of $\psi$, we have 
\BE
	\dotp{\frac{\theta\fl''(\eta_{\fl}(\vz,\theta))}{1+\theta\fl''(\eta_{\fl}(\vz,\theta))}}
	\ \equiv \
	\dotp{\theta\fl''(\eta_{\fl}(\vz,\theta))\frac{\partial \eta_{\fl}(\vz,\theta)}{\partial z}}
	\ \equiv \ 
	\dotp{\psi'(\vz,\theta)}
	.	\label{eq:limiteqb_2}
\EE
Hence, from Lemma \ref{lem:eta'}  we conclude that (\eqref{eq:main_limitequation1_1},\eqref{eq:main_limitequation1_3}) is equivalent to \eqref{eq:main_limitequation1_1} together with the following equation
\BE
	\dotp{\frac{\theta\fl''(\eta_{\fl}(\vz,\theta))}{1+\theta\fl''(\eta_{\fl}(\vz,\theta))}}
	&=&
	\frac{1}{\delta}\dotp{\frac{1}{1+\tau\reg''\left(\eta'\left(\vbeta+\vX^{\t}\frac{\psi(\vz,\theta)}{\dotp{\psi'(\vz,\theta)}},\tau\right)\right)}}
		\n\\
	&=&
	\frac{1}{\delta}\dotp{\frac{1}{1+\tau\reg''\left(\vbeta\right)}}
	.	\label{eq:limiteqb_3}	
\EE
From the definition of $\eta$ and Assumption O.\ref{ass:Convex}, we conclude that \eqref{eq:main_limitequation1_1} is equivalent to 
\BE
	\vzero
	&=&
	\vbeta-\left(\vbeta+\vX^{\t}\frac{\psi(\vz,\theta)}{\dotp{\psi'(\vz,\theta)}}\right)+\tau\reg'(\vbeta)
		\n\\
	&=&
	-\vX^{\t}\frac{\psi(\vz,\theta)}{\dotp{\psi'(\vz,\theta)}}+\tau\reg'(\vbeta)
		\n\\
	&=&
	-\frac{1}{\dotp{\frac{\fl''(\eta_{\fl}(\vz,\theta))}{1+\theta\fl''(\eta_{\fl}(\vz,\theta))}}}\vX^{\t}\fl'(\eta_{\fl}(\vz,\theta))+\tau\reg'(\vbeta)
	.	\label{eq:limiteqb_4}
\EE
Hence, (\eqref{eq:main_limitequation1_1}-\eqref{eq:main_limitequation1_4}) is equivalent to (\eqref{eq:limiteqb_4}, \eqref{eq:limiteqb_1}, \eqref{eq:limiteqb_3}, \eqref{eq:main_limitequation1_4}). If we plug \eqref{eq:limiteqb_1} in \eqref{eq:limiteqb_4} and \eqref{eq:limiteqb_3}, then we conclude that \eqref{eq:main_limitequation2} is equivalent to the following equation:
\begin{subequations}\label{eq:part_b_limitequation2}
\begin{align}
	\frac{1}{\dotp{\frac{\fl''(\vy-\vX\vbeta)}{1+\theta\fl''(\vy-\vX\vbeta)}}}\vX^{\t}\fl'(\vy-\vX\vbeta)
	&\spaceequal
	\tau\reg'(\vbeta)
	,	\label{eq:part_b_limitequation2_1}\\
	\eta_{\fl}\left(\vz,\theta\right)
	&\spaceequal
	\vy-\vX\vbeta
	,	\label{eq:part_b_limitequation2_2}\\
	\dotp{\frac{\theta\fl''(\vy-\vX\vbeta)}{1+\theta\fl''(\vy-\vX\vbeta)}}
	&\spaceequal
	\frac{1}{\delta}\dotp{\frac{1}{1+\tau\reg''\left(\vbeta\right)}}
	,	\label{eq:part_b_limitequation2_3}\\
	\gamma
	&\spaceequal
	\tau\cdot \dotp{\frac{\fl''(\vy-\vX\vbeta)}{1+\theta \fl''(\vy-\vX\vbeta)}}
	.	\label{eq:part_b_limitequation2_4}
\end{align}
\end{subequations}
Then, if we plug \eqref{eq:part_b_limitequation2_4} in \eqref{eq:part_b_limitequation2_1} and  \eqref{eq:part_b_limitequation2_3}, we conclude that \eqref{eq:part_b_limitequation2} is equivalent to the following set of equations:
\begin{subequations}\label{eq:part_b_limitequation3}
\begin{align}
	\vzero
	&\spaceequal
	-\vX^{\t}\fl'(\vy-\vX\vbeta)+\gamma\reg'(\vbeta)
	,	\label{eq:part_b_limitequation3_1}\\
	\eta_{\fl}\left(\vz,\theta\right)
	&\spaceequal
	\vy-\vX\vbeta
	,	\label{eq:part_b_limitequation3_2}\\
	\theta
	&\spaceequal
	\frac{1}{\delta\gamma}\dotp{\frac{\tau}{1+\tau\reg''\left(\vbeta\right)}}
	,	\label{eq:part_b_limitequation3_3}\\
	\gamma
	&\spaceequal
	\tau\cdot \dotp{\frac{\fl''(\vy-\vX\vbeta)}{1+\theta \fl''(\vy-\vX\vbeta)}}
	.	\label{eq:part_b_limitequation3_4}
\end{align}
\end{subequations}
Then, plug \eqref{eq:part_b_limitequation3_3} in \eqref{eq:part_b_limitequation3_4}, we have \eqref{eq:part_b_limitequation3} is equivalent to \eqref{eq:main_limitequation2} which is the following:
\begin{subequations}
\begin{align}
	\vzero
	&\spaceequal
	-\vX^{\t}\fl'(\vy-\vX\vbeta)+\gamma\reg'(\vbeta)
	,	\n \\
	\gamma
	&\spaceequal
	\dotp{\frac{\fl''(\vy-\vX\vbeta)}{\frac{1}{\tau}+\frac{1}{\delta\gamma}\dotp{\frac{1}{1+\tau\reg''(\vbeta)}}\cdot \fl''(\vy-\vX\vbeta)}}
	,	\n \\
	\theta
	&\spaceequal
	\frac{1}{\delta\gamma}\dotp{\frac{\tau}{1+\tau\reg''(\vbeta)}}
	,	\n \\
	\eta_{\fl}\left(\vz,\theta\right)
	&\spaceequal
	\vy-\vX\vbeta
	.	\n
\end{align}
\end{subequations}
Finally, due to the definition of $\eta_{\fl}$ function, we conclude that $\vz=\vy-\vX\vbeta+\theta\fl'(\vy-\vX\vbeta)$ is the unique solution of  
\[
	\eta_{\fl}\left(\vz,\theta\right)
	\spaceequal
	\vy-\vX\vbeta
	.
\]
Hence, we conclude that \eqref{eq:main_limitequation} is equivalent to \eqref{eq:main_limitequation2}. 
\end{proof}

Note that \eqref{eq:main_limitequation2_1} implies the AMP estimate $\vbeta^{\infty}_{\tau^*}$ is the is the solution of \eqref{eq:model} with tuning parameter $\gamma^*$, i.e., $\vbeta^{\infty}_{\tau^*}=\hvb_{\gamma^*}$. Next, from Lemma \ref{lem:AMP=GLM1}, we know that given $(\vX,\vy,\vbeta)$, \eqref{eq:main_limitequation2_2} defines a bijection mapping between $\gamma$ and $\tau$. Then since $(\vbeta^{\infty}_{\tau^*},\gamma^*)=(\hvb_{\lambda},\lambda)$, we know $\hat{\tau}$ defined in \eqref{eq:taudef} exists and $\hat{\tau}=\tau^*$. Finally, since $(\vbeta^{\infty}_{\tau^*},\gamma^*,\tau^*)=(\hvb_{\lambda},\lambda,\hat{\tau})$, according to \eqref{eq:thetadef}, \eqref{eq:main_limitequation2_3} and \eqref{eq:main_limitequation2_4}, we know 
\[
	\vz^{\infty}_{\tau^*}
	\spaceequal
	\vy-\vX\hvb_{\lambda}+ \hat{\theta}\cdot\fl'(\vy-\vX\hvb_{\lambda})
	. 
\]
As is clear from \eqref{eq:main_limitequation2_4}, $\vz =\vy-\vX\vbeta+\theta \cdot \fl'(\vy-\vX\vbeta)$ acts like an estimate of the residual. Also, as described before the main objective of the term $\theta \cdot \fl'(\vy-\vX\vbeta)$ is to make $\vrx{i}$ almost independent of $\vbeta$. Hence, at the intuitive level one would expect $\vz$ to act like a leave-one-out cross validation estimate of the residuals.  The heuristic leads to \eqref{eq:amprhat} as an estimate of the out-of-sample prediction error.

In Examples \ref{ex:LAD_ELAS} and \ref{ex:PH_ELAS} we claimed that a constant fraction of $\vy-\vX\hvb_{\lambda}$ remain bounded. Now, we want to use the AMP framework to heuristically argue that this is in fact the case. Let $\mathcal{S}_{\lambda} = \{\lambda: \forall \tau>0, \text{the AMP estimate $\vbeta^{\infty}_{\tau^*}$ is the solution of \eqref{eq:model} with tuning parameter $\lambda$}\}$. Then for any $\lambda\in\mathcal{S}_{\lambda}$, Lemma \ref{lem:AMP=GLM2} implies that 

\[\vz^{\infty}_{\tau^*} = \vy-\vX^{\t}\hvb_\lambda+ \hat{\theta}\cdot\fl'(\vy-\vrx{i}^{\t} \hvb_{\lambda}).
\]
Note that the loss functions in both Example 3 and Example 4 has bounded first derivatives and the regularization functions in both examples are elastic-net satisfying $\inf_x \reg(x)''\geq 2\gamma$. Therefore, from \eqref{eq:thetadef}, we have $\hat{\theta} = \op{\frac{1}{\gamma}} = \op{1}$ and thus $y_i-\vrx{i}\hvb_{\lambda} = \op{\ve_i^{\t}\vz^{\infty}_{\tau^*}}$ . Since the empirical CDF of $\vz^{\infty}_{\tau^*}$ converges to that of a Gaussian, the fraction of $\vz^{\infty}_{\tau^*}$ that remains in any bounded interval will converge to non-zero number. Hence, a non-zero fraction of the residuals $\vy-\vX\hvb_{\lambda}$ will be bounded.